\documentclass{article}

\usepackage{russ}
\usepackage[german, english]{babel}
\usepackage{amssymb,amsmath}
\usepackage{amsmath, amsthm}
\usepackage{amsfonts,epsfig}
\usepackage{graphics,graphicx}
\usepackage{chngcntr}
\usepackage{mdframed}
\usepackage{xcolor}
\usepackage{amsthm}
\theoremstyle{plain}

\usepackage{amsfonts, epsfig}
\usepackage{hyperref}
\usepackage{geometry}
\usepackage{pifont}
\usepackage{blindtext}
\usepackage[titletoc,title]{appendix}
\usepackage[affil-it]{authblk}
\usepackage{enumerate}
\usepackage[shortlabels]{enumitem}

\usepackage{tikz-qtree}
\usepackage{authblk}

\newcommand{\concat}{\ensuremath{+\!\!\!\!+\,}}
\newcommand{\nconcat}{+\!\!+}
\newtheorem{theorem}{Theorem}
\newtheorem{corollary}{Corollary}
\newtheorem{proposition}{Proposition}
\newtheorem{definition}{Definition}
\newtheorem{algorithm}{Algorithm}
\newtheorem{remark}{Remark}
\newtheorem*{properties*}{Properties}
\newtheorem{example}{Example}
\newtheorem{lemma}{Lemma}
\newcommand{\integers}{\mathbb{Z}}

\newcommand{\BGF}{\boldsymbol{\Phi}}
\newcommand{\BXI}{\boldsymbol{\Xi}}
\newcommand{\BGAMMA}{\boldsymbol{\Gamma}}
\newcommand{\B}{\rm \mathbf{B}}

\newcommand{\complex}{\mathbb{C}}

\newcommand{\A}{{\rm \mathbf{A}}}

\newcommand{\Unit}{\boldsymbol{\rm e}}

\newcommand{\Hes}{{\rm \mathbf{H}}}

\newcommand{\0}{{\rm \boldsymbol{0}}}
\newcommand{\2}{\textit{\large 2}}

\newcommand{\R}{\mathfrak{R}}

\newcommand{\SG}{\mathbb{S}}
\newcommand{\I}{\mathbb{I}}
\newcommand{\BI}{\boldsymbol{\rm I}}
\newcommand{\J}{\mathfrak{J}}

\newcommand{\F}{\boldsymbol{\rm F}}
\newcommand{\G}{\boldsymbol{\rm G}}
\newcommand{\y}{\boldsymbol{\rm y}}
\newcommand{\x}{\boldsymbol{\rm x}}
\newcommand{\BR}{\boldsymbol{\rm R}}

\newcommand{\per}{\ell}
\newcommand{\SEPs}{{\mathcal{S}}}
\newcommand{\NSEPs}{{\mathcal{E}}}
\newcommand{\card}{{\rm card}}
\newcommand{\FZ}{\mathcal{Z}}
\newcommand{\FX}{\mathfrak{X}}
\newcommand{\FY}{\mathfrak{Y}}

\newcommand{\ISCs}{\mathfrak{C}}
\newcommand{\z}{\mathfrak{z}}
\newcommand{\1}{\boldsymbol{1}}
\newcommand{\n}{{\rm n}}
\newcommand{\br}{\boldsymbol{\rm r}}
\newcommand{\lbracket}{[\![}
\newcommand{\rbracket}{]\!]}
\newcommand{\FRREF}{\boldsymbol{\rm FRREF}}
\begin{document}
	
	\bibliographystyle{plain}
	
\title{\large{\textbf{Explicit and Compact Representations for the One-Sided Green's Function\\  and the Solution of Linear Difference Equations with Variable Coefficients\footnote{This paper is an integrated version of an earlier publication with the homonym title. It includes some remarkable theoretical supplements and updated references.}}}}
\author{A. G. Paraskevopoulos$^{\ddagger}$ and M. Karanasos$^{\dagger}$\thanks{The authors would like to thank Dr. Stavros Dafnos for his valuable comments and  writing assistance on various topics of this manuscript.}   \\
	$^{\ddagger}$University of Piraeus, $^{\dagger}$Brunel University London
}
	\maketitle
	\begin{abstract}
\noindent Leibniz' combinatorial formula for determinants is modified to establish a condensed and easily handled compact representation for Hessenbergians, referred to here as Leibnizian representation. Alongside, the elements of a fundamental solution set associated with linear difference equations with variable coefficients of order $p$ are explicitly represented by $p$ banded Hessenbergian solutions, built up solely of the variable coefficients. This yields banded Hessenbergian  representations for the elements both of the product of companion matrices and of the determinant ratio formula of the one-sided Green's function (Green's function for short).
Combining the above results, the elements of the foregoing notions are endowed with compact representations formulated here by Leibnizian and nested sum representations. 
We show that the elements of the fundamental solution set can be expressed in terms of the first banded Hessenbergian fundamental solution, called principal determinant function. We also show that the Green's function coincides with the principal determinant function,  when both functions are restricted to a fairly large domain. These results yield, an explicit and compact representation  of the Green's function restriction along with an explicit and compact solution representation of the previously stated type of difference equations in terms of the variable coefficients, the initial conditions and the forcing term. 
The equivalence of the Green's function solution representation and the well known single determinant solution representation is derived from first principles. Algorithms and automated software are employed to illustrate the main results of this paper.
		

\bigskip

MSC: 0.5-0.8, 15A99, 39A06, 65Q10, 68R05 

\end{abstract}

\section{Introduction}
Linear difference equations with variable coefficients of order $p$ (briefly VC-LDEs($p$)) are broadly used to model discrete-time non-stationary stochastic processes such as autoregressive moving average (ARMA) models with time-dependent  coefficients. This type of models, compared to those with constant coefficients, turn out to be more realistic and sensitive to abrupt and structural changes, as they are efficient approximations to non-linear ones, while linearity maintains interpretation and forecasting advantages (see \cite{Granger2008}). 
Efficient explicit representations to the solution of VC-LDEs($p$) having order greater than one ($p>1$) is a long-standing research topic. 
There are two dominant schemes for an explicit solution representation of VC-LDEs($p$), those of determinant representation (see \cite{Riz81, KitRep93}) and those of compact representation (see \cite{MaEx98, LimDay11}). A pioneer work to link the two representation schemes was recently accomplished by Marrero and Tomeo in \cite{MaTo16}, establishing there the equivalence between the combinatorial solution representation of VC-LDEs($p$) obtained by Mallik in \cite{MaEx98} and the single determinant  solution representation obtained by Kittappa in \cite{KitRep93}.  
They also established in \cite{MaTo16} a nested sum representation for Hessenbergians and, as a consequence, a compact solution representation of VC-LDEs($p$). 
	
In this context, we provide here a more condensed combinatorial representation for Hessenbergians, referred to as \textit{Leibnizian representation} (see eq. (\ref{compact-form}) in Theorem \ref{Compact form of Hessenbergian}). An algorithm for the symbolic computation of the Leibnizian representation of Hessenbergians is provided in Appendix \ref{Appendix: Algorithms}, Algorithm \ref{algo. Compact}. Unlike the Leibniz combinatorial formula for $k$th order determinants, which consists of $k!$ singed elementary products (SEPs) and their summation index ranges over the symmetric group of permutations, the Leibnizian representation of Hessenbergians, obtained here, is a sum of $2^{k-1}$ distinct non-trivial SEPs, whose summation index ranges over integer intervals. The latter has profound effects to the asymptotic properties of the solutions of VC-LDEs($p$).
	
Judging from the missing cross references, the solution representations obtained in the above cited references have not been utilized in time series  modelling.
In contrast, the large family of time-varying coefficient models, such as ARMA processes with variable coefficients (TV-ARMA), employ the one-sided Green's function representation (Green's function for short) of a particular solution of the associated VC-LDE, because it facilitates the development of elegant and generic expressions for their fundamental properties, including the Wold-Cram\'{e}r solution representation, the asymptotic properties as well as the optimal linear forecasts. This approach was originated by Miller \cite{MiLD68} (see also \cite{Hal78, SinPei87, KowSi91}) as an alternative to the standard characteristic polynomial formulation, which has been broadly used in the solution representation of LDEs($p$) and associated ARMA models with constant coefficients, but loses its strength in the presence of variable coefficients (see \cite{Hal78}). 
	
Explicit representations of the Green's function strongly depend upon the availability of a fundamental set of solutions (see \cite{MiLD68}, eq. (2.6), p. 39  or  \cite{AGLDE00}, eq. (2.11.7), p. 77)), whose elements (called linearly independent or fundamental solutions) must be explicitly expressed and computationally tractable.\footnote{An explicit expression evaluated in polynomial running time will be referred to as computationally tractable.}  The lack, in the general case, of a fundamental solution set associated with VC-LDEs($p$), whose elements fulfil the above mentioned characteristics, has led to a dichotomy between the explicit representation of the properties of TV-ARMA models in terms of the Green's function and the recursive computation of this function (see for example \cite{Hal78}).
	
The two research paths in the literature concerning solution representations of VC-LDEs($p$) and corresponding representations of TV-ARMA models, respectively,  have been increased over time, but the overlap between them has not. The results of the present paper establish common ground developments between them and provide the mathematical framework for a unified theory of TV-ARMA  models including processes with deterministic or stochastic variable coefficients (see \cite{KaPaCa21}). An application of this theory is the modelling of stock volatilities during financial crises, presented in~\cite{KaMRV14}.
	
In the present research, we introduce a fundamental set of solutions associated with VC-LDEs($p$), whose terms are banded Hessenbergians initiated by $p$ distinct unit vectors of the same magnitude $p$, respectively (see Subsection \ref{Fundamental and General Homogeneous Solutions}, Proposition \ref{KSIs solutions} and Theorem \ref{Theo: fundamental solution set}). The nonzero entries of the associated Hessenberg matrices are the variable coefficients of a VC-LDE($p$) evaluated at consecutive point instances. The banded Hessenbergian form of the aforementioned $p$ fundamental solutions is strongly suggested by their simultaneous construction, as a result of the infinite Gaussian elimination algorithm (see in Subsection \ref{subs. Infinite Gaussian Elimination}). Banded Hessenbergians are computationally tractable due to the linear time complexity needed for their evaluation (see the discussion below Corollary \ref{cor. Green's function  homogeneous solution representation}). The first fundamental solution gives rise to the \textit{principal determinant function}, denoted by $\xi_{t,r}$ (see Definition \ref{Principal determinant}).
In Proposition \ref{prop. global feature} we show that the elements of the aforementioned fundamental solution set, and therefore of the general homogeneous solution of VC-LDEs($p$), can be expressed in terms of the function $\xi_{t,r}$. A particular solution is also expressed as a linear combination of $\xi_{t,r}$ times the forcing terms of the VC-LDE($p$) (see eq. (\ref{particular solution determinant representation2}) in Proposition \ref{particular solution determinant representation}) and therefore as a Hessenbergian, but not, in the general case, as a banded one. 
	
Two of the main results of this paper concern explicit and compact representations of the Green's function. The first, recovers the defining formula of the Green's function, usually denoted by $H(t,r)$,  as a ratio of two determinants, but now their elements are banded Hessenbergians built up solely of the variable coefficients (see Theorem \ref{theo: Explicit Green's Function}). The second, in Theorem \ref{theo. extension of the Green's function domain}, shows that the restriction of the principal determinant function $\xi_{t,r}$ on a suitably defined domain $\FZ$, coincides with the restriction of the Green's function on $\FZ$. It turns out that $\FZ$ exactly matches the domain restriction of $\xi_{t,r}$, involved in the solution of VC-LDEs. This result allows us to exchange the roles between $\xi_{t,r}$ and $H(t,r)$ in the solution formulas of VC-LDEs($p$) (see Corollary \ref{cor. Green's function  homogeneous solution representation}).
To the extend of our Knowledge there are no fully explicit representations of the homogeneous and non-homogeneous solutions of VC-LDEs($p$) (see  eqs. (\ref{Green's representation of the homog. solution}) and (\ref{nonhomogeneous solution}) respectively), exclusively in terms of the Green's function, the variable coefficients, the initial conditions and the forcing terms. As a consequence, the main notions associated with VC-LDEs are solely expressed in terms of $\xi_{t,r}$, including the fundamental set of solutions (see Proposition \ref{prop. global feature}), the product of companion matrices (see Theorem \ref{theo: explicit PCM}), the Green's function (see eq. (\ref{eq: Explicit Green's Function})) and the  homogeneous and nonhomogeneous solutions (see eqs. (\ref{eq. general homogeneous solution}) and (\ref{nonhomogeneous solution2}), respectively). 
By replacing the Green's function with the principal determinant function in the solution formulas, employed by the previously cited works on time varying coefficient models, the above noted dichotomy is reconciled (see \cite{KaPaCa21}). 
 
In Proposition \ref{equivalence} we show from first principles the equivalence between the Green's function solution representation and the single determinant solution representation obtained in \cite{KitRep93}. This equivalence result capitalizes on an alternative  building process yielding the same fundamental solution set, but originated by Cramer's rule rather than the infinite Gaussian elimination algorithm (see the discussion below Proposition \ref{equivalence}).
	
Thanks to the Leibnizian and nested sum representations of Hessenbergians, the paper concludes with the compact representations of the Green's function restriction, involved in the general solution of VC-LDEs($p$), and of the solution itself (see Section \ref{Compact Representations}).  In  Algorithm \ref{algo. general solution}  of Appendix  the Leibnizian compact representation of the Green's function and the solution of VC-LDEs($p$) are verified by a symbolic computation.
	
\section{Leibnizian Representation of  Hessenbergians}	\label{sec: Hessenbergians}
In all that follows the set $\integers$ (resp. $\integers_{a}$) stands for the set of integers (resp. the set of integers greater than or equal to $a\in\integers$) and $\complex$ for the algebraic field of complex numbers. The group of permutations on $\{1,2,\dots ,k\}$ is denoted by $\SG_k$ and the signature, $sgn(\per)$, of $\per\in \SG_k$ is assigned to $-1$ if $\per$ is an odd permutation of $\SG_k$ and $+1$ if $\per$ is an even one. The building blocks of the well known  
Leibnizian determinant expansion for the $k$th order square matrix $\A=[a_{i,j}]_{1\le i,j \le k}$ over $\complex$, that is
\begin{equation} \label{Leibniz formula} \det(\A)=\sum_{\per\in \SG_k} sgn(\per)\prod_{i=1}^{k}a_{i,\per_i},
\end{equation}
are the singed elementary products, which can be formally defined as follows:
Let  $\per\in \SG_k$. A signed elementary product (SEP) of a square matrix $\A=[a_{i,j}]_{1\le i,j \le k}$ over $\complex$ is an ordered pair 
$\displaystyle(\per,sgn(\per)\prod_{i=1}^{k}a_{i,\per_i})$ in $\SG_k\times\complex$, where the second component of the ordered pair is the numerical value of the SEP in $\complex$. We infer that two SEPs of $\A$, say  $(\per,sgn(\per)\prod_{i=1}^{k}a_{i,\per_i})$ and  $(l,sgn(l)\prod_{i=1}^{k}a_{i,l_i})$,  are equal if and only if $\per=l$.
In all that follows we shall use the standard notation of SEPs: 	$sgn(\per)a_{1,\per_1}a_{2,\per_2}...a_{k,\per_k},\ \ \per\in \SG_k$.
The set of SEPs associated with $\A$ will be denoted here as $\SEPs_{\A}$. As a consequence of the above discussion, every SEP in $\SEPs_{\A}$ is associated with a permutation $\per$ up to the bijection:
\begin{equation}\label{bijective of SEPs}
\SG_k\ni \per\mapsto sgn(\per)a_{1,\per_1}\dots a_{k,\per_k}\in\SEPs_{\A}.
\end{equation}
It follows from the bijection in (\ref{bijective of SEPs}) that the number of distinct SEPs in $\SEPs_{\A}$ is $k!$, since  $\card(\SG_k)=k!$.
	
The $k$th order lower Hessenberg matrix $\Hes_k=[h_{i,j}]_{1 \le i,j\le k}$  and the infinite order lower Hessenberg matrix $\Hes=[h_{i,j}]_{i,j\ge 1}$ over $\complex$ are both satisfy the defining condition $h_{i,j}=0$, whenever $j-i>1$, and displayed below:
\begin{equation} \label{Hessenberg Matrix}
\Hes_k=\left[\begin{array}{cccccc}
h_{1,1}      &  h_{1,2}   &   0        & ... &   0      & 0 \\	h_{2,1}      &  h_{2,2}   &   h_{2,3}  & ... &   0      & 0  \vspace{-0.05in}\\
\vdots     &  \vdots  &    \vdots        & \vdots\vdots\vdots &   \vdots           & \vdots  \vspace{-0.05in}\\
h_{k-1, 1} & h_{k-1,2}& c_{k-1,3} & ... & h_{k-1, k-1} & h_{k-1, k} \\
h_{k,1}      &  h_{k,2}   & h_{k,3}     & ... & h_{k, k-1}     & h_{k, k} \end{array}\right],\ \ \ \Hes=\left[\begin{array}{ccccc}
h_{1,1}      &  h_{1,2}   &   0    &   0     &   ...  \\
h_{2,1}      &  h_{2,2}   &   h_{2,3} &   0 &    ...  \\
h_{3,1}      &  h_{3,2}   &   h_{3,3} &   h_{3,4} &   ...  \vspace{-0.05in}\\ 
\vdots    &  \vdots    &   \vdots  &   \vdots  &       
\end{array}\right].
\end{equation}
From here onwards a matrix $\Hes_k$ is considered as a term of the infinite chain of lower Hessenberg matrices  $\Hes_1\sqsubset\Hes_2\sqsubset\dots \sqsubset\Hes_k\sqsubset\dots \sqsubset\Hes$,
where the notation $\Hes_k\sqsubset\Hes$ means that $\Hes_k$ is a top submatrix of $\Hes$ consisting of the first $k$ rows and columns of $\Hes$. The determinant of $\Hes_k$ for $k\ge 1$ satisfies the well known recurrence\vspace{-0.05in}
\begin{equation} \label{Hessenbergian recurrence}
\det(\Hes_k)=h_{k,k}\det(\Hes_{k-1})+\sum_{i=1}^{k-1}(-1)^{k-i}h_{k,i}\prod_{j=i}^{k-1}h_{j,j+1}\det(\Hes_{i-1}),
\end{equation}
where $\det(\Hes_0)=1$ and $\det(\Hes_1)=h_{1,1}$ (for a proof of the recurrence formula in eq.  (\ref{Hessenbergian recurrence}) see \cite{NaFd02}).
	
The zero valued entries of $\Hes_k$ positioned above the superdiagonal, that is the entries $h_{ij}$  whose indices satisfy $j-i>1$, will be called \emph{trivial}, while the remaining entries of $\Hes_k$, including the entries of the superdiagonal, will be called \emph{non-trivial}. A SEP of $\det(\Hes_k)$ will be called \emph{trivial} if it contains at least one trivial entry. Otherwise it is called \emph{non-trivial}.
Throughout this paper the set of distinct  non-trivial SEPs associated with $\det(\Hes_k)$ is denoted by $\NSEPs_{k}$. 
If $i,j\in\integers$, we adopt the integer interval notation: $\lbracket i,j\rbracket\stackrel{{\rm def}}{=}[i,j]\cap\integers$ and   $\I_{k-1}\stackrel{{\rm def}}{=}\lbracket 0,2^{k-1}-1\rbracket$.
\subsection{Non-trivial SEPs and their String Structure}
The non-trivial entries $h_{i,j},\ j\le i$, positioned below and including the main diagonal of $\Hes$, will be called  \emph{standard factors}, while the sign-opposite entries of the super-diagonal, i.e. the entries $-h_{i,i+1}$, will be called \emph{non-standard factors}. By assigning $c_{i,j}=h_{i,j}$, whenever $j\not=i+1$, and $c_{i,i+1}=-h_{i,i+1}$ the matrices in eqs. (\ref{Hessenberg Matrix}) take the form:
\begin{equation} \label{Modified Hessenberg Matrix}
\Hes_k\!=\!\left[\!\!\begin{array}{cccccc}
c_{1,1}    \!\!  & -c_{1,2}  \!\!  &  0     \!\!\!   & ... \!\!\! &   0           & 0  \\
c_{2,1}    \!\!  &  \ \  c_{2,2}  \!\!  &  -c_{2,3} \!\!\! & ...\!\!\!  &   0        & 0  \vspace{-0.05in}\\
\vdots       \!\!   &  \vdots        & \!\!  \vdots     \!\!\!    &\vdots\vdots\vdots \!\!\! &    \vdots        \!\!   & \vdots  \vspace{-0.02in}\\	c_{k-1, 1} \!\!& c_{k-1,2}\!\! & c_{k-1,3}\!\! & ...\!\!\! &\ \ c_{k-1, k-1}\!\!  &-c_{k-1, k} \\
c_{k,1}   \!\!    & c_{k,2}\!\!   & c_{k,3} \!\!\!   &...\!\!\!  & c_{k, k-1}  \!\!  &c_{k, k}
\end{array}\right],\ \ \Hes\!=\!\left[\!\!\begin{array}{ccccc}
c_{1,1}      &  -c_{1,2}   &    0    &\!\!    0       &  ... \\
c_{2,1}       &   c_{2,2}   &   -c_{2,3} &\!\!      0  &  ...  \\
c_{3,1}      &   c_{3,2}   &    c_{3,3} &\!\!     -c_{3,4}  &  ...  \vspace{-0.03in}\\ 
\vdots    &  \vdots    &    \vdots  &   \vdots   &   
\end{array}\hspace{-0.01in}\!\!\right].
\end{equation}
Writing the Hessenbergian recurrence in eq. (\ref{Hessenbergian recurrence})  in terms of the entries of $\Hes_k$ in eq. (\ref{Modified Hessenberg Matrix}), after some algebraic manipulations (see for details Proposition \ref{positively signed} (i) in Appendix) can be equivalently rewritten as 
\begin{equation}\label{modified recurrence1}
\det(\Hes_k)=\displaystyle c_{k,k}\det(\Hes_{k-1})+\sum_{j=1}^{k-1}\prod_{i=j}^{k-1}c_{k,j}c_{i,i+1}\det(\Hes_{j-1}).
\end{equation}
The recurrence in eq. (\ref{modified recurrence1}) consists exclusively of positively singed non-trivial SEPs,  which gives a comparative advantage to the Hessenberg matrix forms in eqs. (\ref{Modified Hessenberg Matrix}) over those in eqs. (\ref{Hessenberg Matrix}), respectively.
The recurrence of $\det(\Hes_k)$ in eq. (\ref{modified recurrence1}) can be written in an expanded form as:
\begin{align}\label{modified recurrence2}
\det(\Hes_k)=&c_{1,2}c_{2,3}\dots c_{k-1, k}c_{k,1}\det(\Hes_0)+c_{2,3}c_{3,4}\dots c_{k-1, k}c_{k,2}\det(\Hes_1)+ c_{3,4}c_{3,5}\dots c_{k-1, k}c_{k,3}\det(\Hes_{2})+... +\notag\\
&c_{k-1, k}c_{k, k-1}\det(\Hes_{k-2})+c_{k,k}\det(\Hes_{k-1}). \end{align} 
We show in Proposition \ref{positively signed} (ii) of Appendix \ref{Appendix: ProofsOfPropositions}, that the total number of distinct non-trivial SEPs involved in the recurrence (\ref{modified recurrence2}) and therefore in the recurrences in eqs. (\ref{Hessenbergian recurrence}) and (\ref{modified recurrence1}), is $\card(\NSEPs_{k})=2^{k-1}$. This number is considerably less than the number $k!$ of distinct SEPs involved in the Leibniz' determinant expansion in eq. (\ref{Leibniz formula}).
	
Next, we define the notion of  \textit{strings}\footnote{The analysis of non-trivial SEPs of Hessenbergians via strings, defined on Hessenberg matrices of infinite order, was first introduced by the authors in \cite{PaKa14}.} associated with an infinite order Hessenberg matrix $\Hes$ in eq. (\ref{Hessenberg Matrix}) or (\ref{Modified Hessenberg Matrix}). Informally speaking, strings are pieces of non-trivial SEPs and pieces of strings are also strings. A formal Definition is given below:
\begin{definition}[Strings]\label{def.strings}
		A finite product of consecutive factors associated with $\Hes$ is said to be a string, denoted by $C[j,m;\per]=c_{j,\per_j}c_{j+1,\per_{j+1}}\dots c_{m,\per_m}$, if there is some $k\in\integers_{1}$ and a non-trivial SEP, say $C\in\NSEPs_{k}$,\vspace{0.05in} such that $C$ includes $C[j,m;\per]$, that is:
		$$\hspace{-1in}\begin{array}{lcl} C= c_{1,\per_1}\dots \!\!\!&\underbrace{ c_{j,\per_j}\dots c_{m,\per_m}}&\!\!\!\dots c_{k,\per_k}. \\
		& C[j,m;\per] &
		\end{array}$$\\
		An initial string determined by $m$ and $\per$ is defined by $C[1,m;\per]$ and is shortly denoted as $C[m;\per]$.
	\end{definition} 
	Formally the string $C[j,m;\per]$ is uniquely determined by $(j,m,\per)\in\integers^2_1\times\SG_k$ for any $k\ge m$. Let $C[j,m;\per]$ be a string. If $j=m$, then $C[m,m;\per]=c_{m,\per_m}$.  If $i\ge j$ and $n\le m$, then Definition \ref{def.strings} implies that $C[i,n;\per]$ is also a string, since it is included in the same SEP as $C[j,m;\per]$. Accordingly,  $C[i,n;\per]$ might be said to be a \textit{substring} of $C[j,m;\per]$.  Adopting the convention $c_{0,0}=h_{0,0}=1$, the class of all initial strings determined by $i\ge 0$ is denoted by $\ISCs[i]$. The first three classes of initial strings are: $\ISCs[0]=\{c_{0,0}\}$,  $\ISCs[1]=\{c_{1,1}, c_{1,2}\}$ and $\ISCs[2]=\{c_{1,1}c_{2,2}, c_{1,1}c_{2,3}, c_{1,2}c_{2,1}, c_{1,2}c_{2,3}\}$. The major difference between initial strings and non-trivial SEPs is demonstrated by the following example: The initial string $c_{1,1}c_{2,3}$ is included in the non-trivial SEP $c_{1,1}c_{2,3}c_{3,2}$, but the string under discussion is not a SEP. On the other hand, every non-trivial SEP is an initial string, since every SEP is included in itself.  As a consequence,  every string is included in an initial string.
	
	A non-trivial element (or factor) $c_{i,m}$ of $\Hes$ is said to be an \emph{immediate successor} (IS) of the initial string $C[i-1;\per]=c_{1,\per_1}\dots c_{i-1,\per_{i-1}}$, whenever $c_{1,\per_1}\dots c_{i-1,\per_{i-1}}c_{i,m}$ is an initial string too. For instance, the immediate successors of $c_{1,1}$ are $c_{2,2}$ and $c_{2,3}$. 
	Some elementary properties of strings, as arrays of standard and non-standard factors, are summarized below:
	\begin{proposition}[Properties of Strings]\label{properties} $\mbox{}$\\
		1) Every non-trivial entry $c_{i,j}$ of\ \  $\Hes$ in eq. {\rm (\ref{Modified Hessenberg Matrix})} is an IS of some initial string.\\
		2) Every initial string of $\Hes$ has two ISs. One of these is non-standard (therefore the other must be standard).\\
		3) Let $c_{i-1,\per_{i-1}}$ be any standard factor of an initial string, $C[k;\per]$ for some $k\ge i-1$. Then the only possible ISs  of the initial string $C[i-1;\per]$ are  $c_{i,i}$ and  $c_{i,i+1}$.\\
		4) Let $c_{i,\per_i}$ be the standard IS of the initial string $C[i-1;\per]$. If the number of consecutive non-standard predecessors of $c_{i,\per_i}$ is $m$, then $\per_i=i-m$.
	\end{proposition}
	\begin{proof}
		1) It follows directly from the recurrence in eq. (\ref{modified recurrence2}). 
		2) Trivially, the  ISs of $c_{1,1}$ are $c_{2,2}$ and $c_{2,3}$ (standard and non-standard, respectively). An initial string, say $C[i-1;\per]= c_{1,\per_1}\dots c_{i-1,\per_{i-1}}$ for $i\ge 1$ of $\Hes$ has only two ISs, since there are $(i+1)$ candidates, that is the number of the non-trivial elements of the $i$th row of $\Hes$, minus the number $(i-1)$ of the preceding factors, whose column indices have already occurred in the string. Since the column index $(i+1)$ has not previously used by preceding factors, we conclude that one of these ISs is the non-standard factor $c_{i,i+1}$, whence the other must be  standard. 
		3) It suffices to show that the standard IS of $C[i-1;\per]$ is $c_{i,i}$, provided that $c_{i-1,\per_{i-1}}$ is standard, that  is to show that $c_{m,i}$ for any  $m=1,2...,i-1$ is not a factor of this string. If $i-m>1$ (or $i-m\ge 2$), then $c_{m,i}$ is a trivial entry, because $c_{i-2,i},c_{i-3,i},...,c_{1,i}$ are all trivial entries, and therefore are not factors of the string. Moreover, the non-standard factor $c_{i-1,i}$ is not a factor of the string, since, by hypothesis, $c_{i-1,\per_{i-1}}$ is standard.
		Thus, the only available (not previously occurred) standard IS of $c_{i-1,\per_{i-1}}$ is $c_{i,i}$, as required. 4) Property 4 is a generalization of Property 3 and it is shown in Proposition \ref{standard IS} of the Appendix \ref{Appendix: ProofsOfPropositions} (see also the tree diagram below). 
	\end{proof}
	As a consequence of Property 2, the class $\ISCs[k]$ consists of $2^k$ initial strings, since  $\card(\ISCs[k])=2\card(\NSEPs_{k})$.
	We remark that Property 3 follows directly from Property 4, since the number of the non-standard factors between two successive standard ones in any initial string is $m=0$.
	Property 4 can be used to identify the ISs, say $c_{i,\per_{i}}$, of any initial string, by using the following method:  
	\paragraph{Method.}
	Let $C[i-1,\per]=c_{1,\per_{1}}c_{2,\per_{2}}...c_{i-1,\per_{i-1}}$ be an initial string. If $c_{i,\per_i}$ is the  non-standard IS of $C[i-1,\per]$, then $\per_i=i+1$ and $c_{i,\per_i}=c_{i,i+1}$. If $c_{i,\per_i}$ is the standard the IS of $C[i-1,\per]$, then, in order to identify it, we need to count all consecutive non-standard predecessors of $c_{i,\per_i}$. Call this number $m$ ($m$ could be: $0,1,...,i-1$). As $c_{i-m-1,\per_{i-m-1}}$ is a standard factor,  Property 4 entails that $\per_i=i-m$ and the standard IS of $C$ is: $c_{i,\per_i}=c_{i,i-m}$.
	\paragraph{Examples:} 1) Property 3 entails that the ISs of the string  $c_{1,1}c_{2,2}$ are  $c_{3,4}$ (non-standard) and $c_{3,3}$ (standard). In particular, $c_{3,3}$ is the standard IS of the string, since $c_{2,2}$ is standard. These are also the ISs of the string $c_{1,2}c_{2,1}$, since $c_{2,1}$ is standard. 2) The ISs of the string  $c_{1,1}c_{2,3}$  are  $c_{3,4}$ (non-standard) and $c_{3,2}$ (standard). To see how the latter follows from Property 4, denote the standard IS of the string $c_{1,1}c_{2,3}$ as $c_{3,\per_3}$. Since $c_{1,1}$ is standard and $c_{2,3}$ is non-standard, $c_{3,\per_3}$ has only one standard predecessor, whence $m=1$ and $c_{3,\per_3}=c_{3,3-1}=c_{3,2}$.
	
	Properties 2, 3 and 4 can be visualised by a tree diagram:
	
	\bigskip
	
	\hspace{1in}\begin{tikzpicture}
	[sibling distance=0pt]
	\tikzset{level distance=40pt}
	\tikzset{level 3/.style={sibling distance=0pt}}
	\tikzset{level 3/.style={level distance=50pt}}
	\tikzset{level 4/.style={sibling distance=10pt}}
	\tikzset{every tree node/.style={align=center,anchor=north}}
	\Tree [.\node[draw]{$c_{i-m-1,\per_{i-m-1}}$\\(standard)}; $c_{i-m,i-m}$\\(standard)
	[.$c_{i-m,i-m+1}$\\(non-standard)
	$c_{i-m+1,i-m}$\\(standard) [.$c_{i-m+1,i-m+2}$\\(non-standard)\vspace{-0.07in}\\{\vdots} $c_{i-1,i-m}$\\(standard) [.$c_{i-1,i}$\\(non-standard)
	$c_{i,i+1}$\\(non-standard) \node[draw]{$c_{i,i-m}$\\(standard)};  ] ] ] ] ]
	\end{tikzpicture}
	\subsection{Leibnizian Representation of  Hessenbergians} \label{TheHessenbergianRepresentationFormula}
	The indexing function $\sigma_{k,i}$, yielding the Leibnizian representation of a Hessenbergian, $|\Hes_{k}|$ (see eq. (\ref{compact-form})), is defined as a composite of two  functions: The outer  function $\z_{k,i}$ and the inner function $\tau_k$. These functions are established in terms of integer functions in the next two paragraphs of the present Subsection.
	\paragraph{Outer Function}\mbox{}\\
	In what follows, $\R_k$ will stand for the set of $k$-arrays $\br=(r_1,r_2,\dots ,r_k)$ with components either $r_i=0 \ {\rm or} \ 1$ for $1\le i\le k-1$ and $r_k=1$. Trivially $\R_1=(1)$ and the number of the elements in $\R_k$ is  $\card(\R_k)=2^{k-1}$. The outer building function, constructed below,  considerably reduces the number of SEPs in the Leibniz determinant formula of Hessenbergians, built up solely of the non-trivial entries $c_{i,\z_{k,i}(\br)}$ of $\Hes_{k}$ for $\br\in\R_k$ (see eq. (\ref{preclosed-form 3})).
	
	We introduce the function $f_k$, which maps every $C=c_{1,\per_1}c_{2,\per_2}\dots c_{k,\per_k}\in{\mathcal{E}_k}$, to $\br=(r_1,r_2,\dots ,r_{k-1},1)\in \R_k$, according to the rule: $r_i=0$, whenever $c_{i,\per_i}$ is non-standard and $r_i=1$, whenever $c_{i,\per_i}$ is standard. Since the last factor of a non-trivial SEP is always standard, the last component of the array $\br$ has been assigned to $1$, that is $r_k=1$.  As shown in Proposition \ref{identification theorem} of the Appendix, the above rule induces the bijective mapping: \begin{equation*}\label{bijective f}
	f_k:\NSEPs_{k}\ni C \mapsto f_k(C)\in \R_k
	\end{equation*}  
	For example a non-trivial SEP, say $C=c_{1,1}c_{2,3}c_{3,2}c_{4,5}c_{5,6}c_{6,4}c_{7,7}$ is mapped uniquely to the array $f_7(C)=(1,0,1,0,0,1,1)$ and vice versa. 
	In order to verify that $f_7^{-1}(1,0,1,0,0,1,1)=C$, we first assign  the $0$s, which are positioned at $i=2,4,5$, to the non-standard factors having the same positions in $C$, that is:  $c_{2,3},c_{4,5}, c_{5,6}$. By virtue of Property 4 in Proposition \ref{properties}, we assign the $1$s to corresponding standard factors as follows: The first $1$ is mapped to $c_{1,1}$, as $m=0$. The second $1$ is assigned to $c_{3,3-1}$, since there is only one $0$ between the first and third $1$s, that is $m=1$. Similarly, the third $1$ must be assigned to $c_{6,6-2}$, since $m=2$. The last $1$ is assigned to $c_{7,7}$, since $m=0$ and the verification is completed.
	
	More generally speaking, arrays in $\R_k$ represent non-trivial SEPs in $\NSEPs_{k}$, preserving their string structure, as arrays of standard and non-standard factors. 
		As a consequence we have:
	\begin{equation} \label{preclosed form 0}
	\det(\Hes_k)=\displaystyle\sum_{\br\in \R_k}f^{-1}_{k}(\br).\
	\end{equation}
	If $f_k^{-1}(\br)=c_{1,\per_1}c_{2,\per_2}\dots c_{i,\per_i}\dots c_{k,\per_k}\in\NSEPs_{k}$, we call $f^{-1}_{k,i}(\br)=c_{i,\per_i}$, that is the $i$th factor of the SEP, whence\vspace{-0.05in}
	\begin{equation*}\label{eq. ith factor}
	f^{-1}_k(\br)=\prod_{i=1}^kf^{-1}_{k,i}(\br).\vspace{-0.05in}
	\end{equation*}
	In what follows we assign $r_0=1$. It turns out that given $k,i\in\integers$ such that $1\le i\le k$, then for any  $\br\in \R_k$, we have:
	\begin{equation*} \label{def of varphi2} f^{-1}_{k,i}(\br)\!=\!\!\left\{\!\!\!\begin{array}{cll}   
	c_{i, i+1}, \!\!\!\!& {\rm if} \  r_i=0\ {\rm and}\ i\not= k \vspace{0.05in}\\
	c_{i, i-m}, \!\!\!\!&  {\rm if}\; r_{i-m-1}\!=\! r_i\!=\!1\ {\rm and}\
	r_{j}\!=\!0,\ {\rm whenever}\ j\in\integers: \ i-m\le j\le i-1.
	\end{array}\right.  
	\end{equation*}
	The above piecewise expression of $f^{-1}_{k,i}(\br)$ can be written in a single form as:
	\begin{equation}\label{single def varphi2}
	\begin{array}{cccccccl}
	f^{-1}_{k,i}(r_1,r_2,...,\!\!& r_{i-m-2},\!\! &  r_{i-m-1},   &\underbrace{0,\ 0,...,\ 0,} & r_{i} ,\ ...,\!\!&r_{k-1},1)&\!\!=\!\!& c_{i, i-m},   \ \ \  \  m=-1,0...,i-1.\\
	&&  & m &&& & \vspace{-0.1in}
	\end{array} 
	\end{equation}
	As $m$ is the number of successive $0$s between  $r_{i-m-1}=1$ and $r_i=1$,  the formula in eq. (\ref{single def varphi2}) gives: $f^{-1}_{k,i}(\br)=c_{i, i-m}$. If $r_{i}=0$, then $m$ is assigned to $m=-1$, that is the formula in eq. (\ref{single def varphi2}) ignores all the predecessors of $r_{i}$ and gives:
	$f^{-1}_{k,i}(\br)=c_{i, i-(-1)}=c_{i,i+1}$.
	Next, we consider two special cases: $i)$ Let  $r_{i-1}=r_i=1$. Then  $m=0$ and  $f^{-1}_{k,i}(\br)=c_{i,i}$, which is in accord with the fact  that $\{j\in\integers: i\le j\le i-1\}=\emptyset$ and $\card(\emptyset)=0$. $ii)$ Let $r_i=1$ and $m=i-1$. Then, on account of $r_{i-m-1}=r_{i-(i-1)-1}=r_0=1$, eq. (\ref{single def varphi2}) yields:
	\[ \begin{array}{lcl}f^{-1}_{k,i} (&\!\!\!\!\underbrace{0,0,\dots , 0,} &\!\!\!\! 1, r_{i+1}, \dots, r_{k-1},1)=c_{i,i-(i-1)}=c_{i, 1}.\\
	& i-1 &
	\end{array}\]
	
	In view of eq.  (\ref{eq. ith factor}), the  expression for Hessenbergians in  eq. (\ref{preclosed form 0}) can be rewritten as:
	\begin{equation} \label{preclosed form 1}
	\det(\Hes_k)=\displaystyle\sum_{\br\in \R_k}\prod_{i=1}^kf^{-1}_{k,i}(\br),
	\end{equation}
	which consists of $\card(\R_k)=\card(\NSEPs_{k})=2^{k-1}$ distinct non-trivial SEPs. 
	An equivalent expression to eq. (\ref{preclosed form 1}), but directly in terms of the entries $c_{i,j}$ of $\Hes_k$ in eq. (\ref{Modified Hessenberg Matrix}), is obtained in eq. (\ref{preclosed-form 3}) below. For this purpose, we introduce the function\vspace{-0.1in}
	\begin{equation} \label{def of zeta} \zeta_{k,i}(\br)\stackrel{{\rm def}}{=}\left\{\begin{array}{cll}   
	-1 & {\rm if} \  r_i=0\ {\rm and}\ i< k \vspace{0.05in}\\
	m &  {\rm if}\; r_{i-m-1}\!=\!r_i \!=\!1\ {\rm and}\
	r_{j}\!=\!0\ {\rm for\ all}\ j\in \lbracket i-m, i-1\rbracket,
	\end{array}\right. 
	\end{equation}
	which returns the number of consecutive $0$s preceding the component $r_i$  ($i\le k$) in an array $\br\in \R_k$.
	In Proposition \ref{unified formula} of the Appendix, we show a single formula, equivalent to the piecewise expression  of $\zeta_{k,i}$  in eq. (\ref{def of zeta}), which is additionally expressed in terms of elementary functions, as displayed below
	\begin{equation} \label{zeta in terms of elementary functions}
	\zeta_{k,i}  (\br)=r_i(i-\max_{0\le j<i}\{j\cdot r_j\})-1, \ \  i\le k,
	\end{equation}
	noting that: 
	$
	\max\{a_1,a_2\}=\displaystyle\frac{a_1+a_2+\sqrt{(a_1-a_2)^2}}{2}\ \ {\rm and}\ 
	\displaystyle\max_{1\le j< i}\{a_j\}=\max\{\max\{a_1,a_2\},a_3,\dots ,a_{i-1}\}
	$.\\
	Letting $i=1$ in eq. (\ref{zeta in terms of elementary functions}), it follows that $j=0$. Recalling the convention $r_0=1$, it follows from $\{0\cdot r_0\}=\{0\}$ that $\max\ \{0\}=0$, whence\vspace{-0.1in} \begin{equation}\label{special case i=1}
	\zeta_{k,1}(\br)=r_1(1-\displaystyle\max\{0\})-1=r_1(1-0)-1=\left\{\begin{array}{rl}
	-1 & {\rm if}\ r_1=0\\
	0  & {\rm if}\ r_1=1
	\end{array}\right. \ \  \text{for all} \ \ \br\in \R_k,
	\end{equation}
	which is in accord with the Definition in eq. (\ref{def of zeta}).
	Next, we define the function:
	\begin{equation}\label{j minus function of zeros count}
	\z_{k,i}(\br)\stackrel{{\rm def}}{=}i-\zeta_{k,i}(\br).
	\end{equation}
	It follows from eq. (\ref{single def varphi2}) that  $f^{-1}_{k,i}(\br)=c_{i,i-\zeta_{k,i}(\br)}$, whence:
	\begin{equation}\label{phi and z}
	f^{-1}_{k,i}(\br)=c_{i,\z_{k,i}(\br)}.
	\end{equation}
	As a demonstrative example, let  $\br=(1,0,\dots,0,1,1,0,0,1)\in\R_k$. Applying the Definition in eq. (\ref{def of zeta}), along with eqs. (\ref{j minus function of zeros count}) and  (\ref{phi and z}), we get: $\zeta_{k,1}(\br)=0$, $\z_{k,1}(\br)=1-0=1$ and $f^{-1}_{k,1}(\br)=c_{1,\, \z_{k,1}(\br)}=c_{1,1}$. As $r_{2}=r_{3}=...=r_{k-5}=r_{k-2}=r_{k-1}=0$,  we have: $\zeta_{k,2}(\br)=\zeta_{k,3}(\br)=...=\zeta_{k,k-5}(\br)=\zeta_{k,k-2}(\br)=\zeta_{k,k-1}(\br)=-1\,$. Thus,   
	$\z_{k,2}(\br)=2-(-1)=3,...,\z_{k,k-5}(\br)=k-5-(-1)=k-4,\ \z_{k,k-2}(\br)=k-2-(-1)=k-1,\ \z_{k,k-1}(\br)=k-1-(-1)=k$  and  $f^{-1}_{k,2}(\br)=c_{2,\, \z_{k,2}(\br)}=c_{2,3},...,f^{-1}_{k,k-5}(\br)=c_{k-5, \z_{k,k-5}(\br)}=c_{k-5,k-4},\, f^{-1}_{k,k-2}(\br)=c_{k-2, \z_{k,k-2}(\br)}=c_{k-2,k-1},\, f^{-1}_{k,k-1}(\br)=c_{k-1,\, \z_{k,k-1}(\br)}=c_{k-1,k}$. 
	As $r_{k-3}=r_{k-4}=1$, we conclude that the  number of preceding consecutive $0$s of $r_{k-3}$ is zero, whence: $\zeta_{k,k-3}(\br)=0$, 
	$\z_{k,k-3}(\br)=k-3-0=k-3$, and  $f^{-1}_{k,k-3}(\br)=c_{k-3,\z_{k,k-3}}(\br)=c_{k-3,k-3}$. Moreover, as the number of preceding consecutive $0$s of $r_{k-4}$ is $k-6$, we conclude that: $\zeta_{k,k-4}(\br)=k-6$, 
	$\z_{k,k-4}(\br)=k-4-(k-6)=2$, and  $f^{-1}_{k,k-4}(\br)=c_{k-4,\z_{k,k-4}}(\br)=c_{k-4,2}$.
	Finally, as the number of preceding consecutive $0$s of $r_{k}=1$ is 2 we have:
	$\zeta_{k,k}(\br)=2$, $\z_{k,k}(\br)=k-2$ and 
	$f^{-1}_{k,k}(\br)=c_{k,\, \z_{k,k}(\br)}=c_{k,k-2}$. Hence,
	$$f^{-1}_{k}(1,0,\dots,0,1,1,0,0,1)=c_{1,1}c_{2,3}...c_{k-5,k-4}c_{k-4,2}c_{k-3,k-3}c_{k-2,k-1}c_{k-1,k}c_{k,k-2}.$$
	Taking into account eq. (\ref{phi and z})  the formula in eq. (\ref{preclosed form 1}) can be expressed as
	\begin{equation} \label{preclosed-form 3}
	\det(\Hes_k)=\displaystyle\sum_{\br\in \R_k}\prod_{i=1}^kc_{i,\, \z_{k,i}(\br)},
	\end{equation}
	that is an expression of $\det(\Hes_k)$ solely in terms of non-trivial entries of $\Hes_k$ in (\ref{Modified Hessenberg Matrix}). 
	\paragraph{Inner Function}\mbox{}\\
	The inner  function $\tau_k$ is employed to convert integers from $\I_{k-1}$ to arrays in $\R_k$. This makes it possible to replace the indexing set $\R_k$ in eq. (\ref{preclosed-form 3}) with the integer indexing set $\I_{k-1}=\lbracket 0,2^{k-1}-1\rbracket$.
	
	In what follows we shall make use of the conventional notation 
	\[\begin{array}{lcccc} \1_k=&\underbrace{11...1}& (k\ {\rm number\ of} \ 1s),\ \ \ \ \ \0_k=&\underbrace{00...0}& (k\ {\rm number\ of} \ 0s) \\
	& k & &k &\end{array}\]
	that is, $\1_k$  (resp. $\0_k$) represents a binary integer consisting of $k$ consecutive $1$s (resp. $0$s).
	Let ${\mathcal{B}}_k$ be the set of binary integers from $0$ up to and including $\1_k$, that is ${\mathcal{B}}_k=\{0,1,10,...,\1_k\}$. The binary representation of the decimal integer $2^k$ is the binary integer $10^k$, and we write:
	$[2^k]_2=10^k\in{\mathcal{B}}_k$. Thus, $[2^k-1]_2=\1_k$ and ${\mathcal{B}}_k$ consists of $2^k$ binary integers.
	
	In the rest of this paper $\2^k$ will stand for the set of functions from  $\{1,2,...,k\}$ to $\{0,1\}$, that is the set of all $k$-arrays of $0$s and $1$s: $\2^k=\{(r_1,r_2,...,r_{k-1},r_k): r_i=0\ {\rm or}\ 1\}$. 
	The set $\2^k$ can be naturally  identified with the segment of the non-negative binary integers up to and including the integer $\1_k$, that is the sets $\2^k$ and  ${\mathcal{B}}_k$ are identified up to a bijection as follows: 	Let $b\in {\mathcal{B}}_k$. If $b=0$, then   ${\mathcal{B}}_k\ni 0 \equiv \0_k\in\2^k$. If $b\not= 0$, we can write  $b=1\, r_{i+1}...r_n...r_{k}$, where $r_n$ is $0\ \text{or}\ 1$ for $i+1\le n\le k$.  
	By adding $(i-1)$ zero digits to the left side of $b$, the latter is represented as
	\[\begin{array}{clll}1\, r_{i+1}...r_{k}\equiv & 0\ 0\ ...\ 0 & 1\ r_{i+1}...& r_{k} \\ 
	&\uparrow  &\uparrow  &\uparrow \\
	{\rm Binary \ Figures:}& 2^{k-1}  & 2^{k-i}  & 2^0\ {\rm (units)} \end{array}\]
	and $b\in{\mathcal{B}}_k$ turns into an element of $\2^k$ by identifying:  ${\mathcal{B}}_k\ni b \equiv (0,0,...,1,r_{i+1},...,r_{k})\in\2^k$. 
	For example, let $1011\in{\mathcal{B}}_5$. Adding one zero to the left side of $1011$, the binary $1011$ is identified with the array $(0,1,0,1,1)\in\2^5$. 
	
	The set $\mathfrak{R}_k$ can be defined as the subset of $\2^k$ consisting of the elements of $\2^k$ whose last component is $r_k=1$, that is:
	$\mathfrak{R}_k=\{\br\in \2^k: r_k=1\}$. Evidently $\card(\2^k)=2^k$ and $\card(\mathfrak{R}_k)=2^{k-1}$. 
	Taking into account that $\card({\mathcal{B}}_{k-1})=\card(\mathfrak{R}_k)=2^{k-1}$, we define the bijection $\rho_k: {\mathcal{B}}_{k-1}\mapsto \mathfrak{R}_k$:
	\begin{equation} \label{bijection rho} \begin{array}{cccccccccc}
	\rho_k\!\!\!\!\!\!&(\underbrace{00...01r_{i+1}...r_{k-1}})\!\!\!\!&=\!\!\!\!&(\underbrace{0,0,...,0,1,r_{i+1},...,r_{k-1},1})&  {\rm and} &
	\rho_k\!\!\!\!\!\!&(\underbrace{00...0})\!\!\!\!&=\!\!\!\!&(\underbrace{0,0,...,0,1})& \\
	&  k-1  & &  k   & & &  k-1  & &  k  \end{array} \end{equation}
	By identifying ${\mathcal{B}}_{k-1}$ with $\mathfrak{R}_k$ via $\rho_k$, the function $f_k^{-1}$, defined by  (\ref{bijective f}), associates every binary integer $\br$ in ${\mathcal{B}}_{k-1}$ with the SEP $f_k^{-1}(\br)$ in one-to-one fashion. 
	
	Let $x\in\integers_{0}$ and $y\in\integers_{1}$. The largest integer less than or equal to the rational number $\displaystyle x/y$ is usually denoted as  $\lfloor x/y \rfloor$, where ``$\lfloor \, \rfloor$" is the floor function.
	Therefore, $\lfloor x/y \rfloor$ coincides the quotient of the Euclidean division of $x$ by $y$. Moreover, $x \!\!\!\mod\! y$ stands for the remainder of the division of $x$ by $y$.  Thus
	$n \!\!\mod 2=\left\{\begin{array}{ll} 0 & \text{if}\ n\ {\rm is\ even}\\ 1  & \text{if}\ n\ {\rm is\ odd}\end{array}\right.\!\! $. If $x-y=nk$ for some $n\in\integers$, then 
	 $x,y$ is said to be congruent modulo $k$, that is $x,y$ have the same remainder when divided by $k$ and we write for it: $x \equiv y\!\! \mod k$. 

The conventional method for converting decimal integers into binary is based on the Euclidean division, which is applied here for converting a decimal integer $n\in \I_{k-1}$ for $k\ge 2$ into a binary $[n]_2=r_1r_2...r_{k-1}\in{\mathcal{B}}_{k-1}$. The digits $r_i$ ($0$ or $1$) in $[n]_2$ turns out to be the remainders of the following nested divisions:
\[n=2q_{k-1}+r_{k-1}, q_{k-1}=2q_{k-2}+r_{k-2}, ..., q_{2}=2q_{1}+r_1.\]
Taking into account that 
\[\begin{array}{ccc}q_{k-1}=\lfloor n:2\rfloor,\ q_{k-2}=\lfloor\lfloor n:2\rfloor:2\rfloor, ..., q_1=\lfloor\lfloor...\lfloor\lfloor n:\!\!\!\!\!&\underbrace{2\rfloor:2\rfloor...\rfloor :2}&\!\!\!\!\rfloor,\\
\!\!\!\!\!&            k-1       &  \end{array} \]
the $r_i$s in $[n]_2=r_1r_2...r_{k-1}$ can be expressed in terms of the floor function  and of the modulo function as:
\[
\begin{array}{ccc} r_{k-1}= n\!\!\!\!\! \mod 2,\ r_{k-2}=\lfloor n\!:\!2\rfloor\!\!\!\!\! \mod 2,..., r_1=\lfloor\lfloor...\lfloor\lfloor n\!:\!\!\!\!\!&\underbrace{2\rfloor:\!2\rfloor...\rfloor \!:\!2}&\!\!\!\!\rfloor\!\!\!\!\!  \mod 2.\\
\!\!\!\!\!&            k-2       & \vspace{-0.1in} \end{array}\vspace{-0.1in}\]
We can also write $r_{k-1}=\lfloor n:2^0\rfloor\!\! \mod 2$, yielding a unified expression 
\begin{equation*}\label{unified expression of ri}
\begin{array}{ccc} r_{i}=\lfloor\lfloor...\lfloor\lfloor n:\!\!\!\!\!&\underbrace{2\rfloor:2\rfloor...\rfloor :2}&\!\!\!\!\rfloor\!\!\!\! \mod 2 \ \ \ \ \text{for}\ 1\le i\le k-1.\\
\!\!\!\!\!&            k-i-1        &\vspace{-0.05in}
\end{array}
\end{equation*}
Thanks to the identity\vspace{-0.05in}	
\begin{equation*}\label{nested divisions} 
\begin{array}{ccc} \lfloor\lfloor...\lfloor\lfloor n:\!\!\!&\underbrace{2\rfloor:2\rfloor...\rfloor :2}&\!\!\!\!\rfloor=\lfloor n:2^m\rfloor\\                        &   m                   & 
\end{array}\end{equation*}
(see for a proof Proposition \ref{nestdiv} in the Appendix), the binary representation  $[n]_2=r_1r_2...r_{k-1}\in{\mathcal{B}}_{k-1}$ of  $n\in \I_{k-1}$ can be expressed as:\vspace{-0.1in}
	\begin{equation*} \label{binary equvalent}
	\begin{array}{ccccc} [n]_2=&\underbrace{\lfloor n\!:\!2^{k-2}\rfloor\!\!\!\!\!\mod\! 2}& \underbrace{\lfloor n\!:\!2^{k-3}\rfloor\!\!\!\!\!\mod\! 2}& ...&\underbrace{\lfloor n\!:\!2^0\rfloor\!\!\!\!\!\mod\! 2}. \\
	& r_1 & r_2 &  & r_{k-1} \end{array}\end{equation*}
This induces the following bijective transformation:
	\begin{equation*} \label{beta} 
	\beta_k: \I_{k-1}\ni n \mapsto \beta_k(n)=[n]_2\in {\mathcal{B}}_{k-1}.\end{equation*}
	The composite $\tau_{k}\stackrel{{\rm def}}{=}\rho_k\circ \beta_k$ determines a bijection, which converts decimal integers in $\I_{k-1}$ into arrays in $\mathfrak{R}_k$, that is 
	\begin{equation} \label{def tau}
	\tau_k(n)=(\lfloor n\! :\! 2^{k-2}\rfloor\!\!\!\!\!\mod 2,\lfloor n\! :\! 2^{k-3}\rfloor\!\!\!\!\!\mod 2,\dots ,\lfloor n\! :\! 2^{0}\rfloor\!\!\!\!\!\mod 2,1), \ \ \  n\in \I_{k-1}.   \end{equation}
	Some illustrative examples are given below:\\
	
	If $k=1$, then $\tau_1: \I_0\ \mapsto \mathfrak{R}_1$ is defined by: $\tau_1(0)=(1)$.
	
	If $k=2$, then:
	\[\begin{array}{lll}
	\tau_{2}(0)&=&(\lfloor 0: 2^{2-2}\rfloor\!\!\!\!\mod 2,1)=(0\!\!\!\!\mod 2,1 )=(0,1)\\
	\tau_{2}(1)&=&(\lfloor 1: 2^{2-2}\rfloor\!\!\!\!\mod 2,1)=(1\!\!\!\!\mod 2,1 )=(1,1).
	\end{array}\]
	
	If $k=3$, then:
	\[\begin{array}{llccc}
	\tau_{3}(0)\!\!\!\!&=(&\!\!\!\lfloor 0: 2^{3-2}\rfloor\!\!\!\!\mod 2,&\lfloor 0: 2^{3-3}\rfloor\!\!\!\!\mod 2,&1)\\
	&=(&\!\!\! 0\!\! \mod 2,&\ 0\!\! \mod 2,& 1)\\
	&=(&\!\!\!0,&0,&1) \end{array}
	\left|\left| \begin{array}{llccc}
	\tau_{3}(1)\!\!\!\!&=(&\!\!\!\lfloor 1: 2^{3-2}\rfloor\!\!\!\!\mod 2,&\lfloor 1: 2^{3-3}\rfloor\!\!\!\!\mod 2,&1)\\
	&=(&\!\!\! 0\!\! \mod 2,&\ 1\!\! \mod 2,& 1)\\
	&=(&\!\!\!0,&1,&1)\end{array}\right.\right.\]
	
	\[\begin{array}{llccc}
	\tau_{3}(2)\!\!\!\!&=(&\!\!\!\lfloor 2: 2^{3-2}\rfloor\!\!\!\!\mod 2,&\lfloor 2: 2^{3-3}\rfloor\!\!\!\!\mod 2,&1)\\
	&=(&\!\!\! 1\!\! \mod 2,&\ 2\!\! \mod 2,& 1)\\
	&=(&\!\!\!1,&0,&1) \end{array}\left|\left| \begin{array}{llccc}
	\tau_{3}(3)\!\!\!\!&=(&\!\!\!\lfloor 3: 2^{3-2}\rfloor\!\!\!\!\mod 2,&\lfloor 3: 2^{3-3}\rfloor\!\!\!\!\mod 2,&1)\\
	&=(&\!\!\! 1\!\! \mod 
	2,&\ 3\!\! \mod 2,&1)\\
	&=(&\!\!\!1,&1,&1)
	\end{array}\right.\right.\]
	
	If $x\in\integers$ and $y\in\integers_{1}$, then, applying the well known identity (see for example \cite{KnPaGr1994} p.82) 
	$$x\!\!\mod y=x-y\lfloor\frac{x}{y} \rfloor, $$ 
	for  $x= \lfloor n:2^{k-i}\rfloor$ ($2\le i\le k$) and $y=2$, to eq. (\ref{def tau}), the latter can be expressed in terms of elementary functions as: \vspace{-0.1in}
	\begin{equation} \label{tau bijection}
	\begin{array} {l} \tau_k(n)\!=\!(\lfloor n: 2^{k-2}\rfloor\!-\! 2\lfloor\frac{\lfloor n: 2^{k-2}\rfloor}{2} \rfloor,\lfloor n: 2^{k-3}\rfloor\!-\! 2\lfloor\frac{\lfloor n: 2^{k-3}\rfloor}{2} \rfloor,\dots ,\lfloor n: 2^{0}\rfloor\!-\!
	2\lfloor\frac{\lfloor n: 2^{0}\rfloor}{2} \rfloor,1). \end{array}
	\end{equation}
\paragraph{Leibnizian Formula}\mbox{}\\
	The main result of this Section is stated and proved by Theorem \ref{Compact form of Hessenbergian}.
	
	Taking into account that $\tau_k(n)$, $n\in \I_{k-1}$ in eq. (\ref{tau bijection}) defines a bijection, we can substitute $\tau_k(n)$ for $\br$ in eq. (\ref{preclosed form 0}) to get:
	\begin{equation} \label{preclosed form 2}
	\det(\Hes_k)=\displaystyle\sum_{n=0}^{2^{k-1}-1} f^{-1}_{k}(\tau_k(n)).
	\end{equation}
	For each $i\in\integers$ such that $1\le i\le k$, we define the function
	\begin{equation}\label{definition of sigma}
	\sigma_{k,i}(n)\stackrel{{\rm def}}{=}\z_{k,i}\circ \tau_k(n)\ \ {\rm for}\ n\in \I_{k-1},
	\end{equation} 
	that is a composite of elementary functions defined over intervals of integers. 
	In view of eq. (\ref{definition of sigma}), we can rewrite eq. (\ref{phi and z}) as:
	\begin{equation}\label{eq. phi and z 2}
	f^{-1}_{k,i}(\tau_k(n))=c_{i,\z_{k,i}(\tau_k(n))}=c_{i,\sigma_{k,i}(n)}\  \ \ {\rm for}\ n\in \I_{k-1}.\vspace{0.1in}
	\end{equation}
	Substituting $\tau_k(n)$ for $\br$ in eq. (\ref{eq. ith factor}), the latter takes the form 
	$$ f^{-1}_{k}(\tau_k(n))=f^{-1}_{k,1}(\tau_k(n))f^{-1}_{k,2}(\tau_k(n))\dots f^{-1}_{k,k}(\tau_k(n))\in\NSEPs_{k}\ \ \text{for}\ \ n\in\I_{k-1},$$ 
	that also defines a bijection, since $f^{-1}_{k}$, $\tau_k$ are bijections. 
	Applying eq. (\ref{eq. phi and z 2}) to the foregoing function formula, the latter  can be equivalently written as 
	\begin{equation}\label{composite of fk with tk_2 }
	(f^{-1}_{k}\circ \tau_k)(n)=c_{1,\sigma_{k,1}(n)}c_{2,\sigma_{k,2}(n)}\dots  c_{k,\sigma_{k,k}(n)}\in\NSEPs_{k}\ \ \text{for}\ \ n\in\I_{k-1},
	\end{equation}
	or in a mapping form as:
	$$f^{-1}_{k}\circ \tau_k:  \I_{k-1}\ni n\mapsto c_{1,\sigma_{k,1}(n)}c_{2,\sigma_{k,2}(n)}\dots  c_{k,\sigma_{k,k}(n)}\in\NSEPs_{k}.$$
	The latter has to be compared with the bijective mapping in eq. (\ref{bijective of SEPs}). 
	\begin{theorem}\label{Compact form of Hessenbergian}
		The Leibnizian representation of the $k$\emph{th} order Hessenbergian $\det(\Hes_k)$ in terms of non-trivial entries of $\Hes_k$, defined in eq. \emph{(\ref{Modified Hessenberg Matrix})}, is\emph{:}
		\begin{equation}\label{compact-form}
		\det(\Hes_k)=\sum_{n=0}^{2^{k\!-\!1}-1} \prod_{i=1}^kc_{i,\sigma_{k,i}(n)}.
		\end{equation}
	\end{theorem}
	\begin{proof}
		Taking into account that the function in  eq. (\ref{composite of fk with tk_2 }) is bijective and starting with eq. (\ref{preclosed form 2}), the result follows from
		\[\begin{array}{lll}
		\det(\Hes_k)&=&\displaystyle\sum_{n=0}^{2^{k-1}-1} f^{-1}_{k}(\tau_k(n))\\\\
		\text{(by eq. (\ref{composite of fk with tk_2 }))} &=&\displaystyle\sum_{n=0}^{2^{k\!-\!1}-1}\prod_{i=1}^kc_{i,\sigma_{k,i}(n)},\vspace{-0.02in}
		\end{array}\vspace{-0.05in}\]
		as claimed.
	\end{proof}
	The compact representation of Hessenbergians in eq. (\ref{compact-form})  must be compared with the corresponding nested sum representation of Hessenbergians  (see \cite{MaTo16} or eq. (\ref{Marrero}) herein) and Mallik's combinatorial formula in \cite{MaEx98}, as adjusted for Hessenbergians in \cite{MaTo16} (see  eq. (9) therein). The formula in eq. (\ref{compact-form}) is also an explicit and compact alternative representation to the recurrence formula (\ref{Hessenbergian recurrence}) for the $k$th order Hessenbergian. An algorithm for its evaluation associated with a computer program are presented in Appendix \ref{Appendix: Algorithms}, Algorithm \ref{algo. Compact}. The program is formulated by the Mathematica symbolic language and verifies the formula (\ref{compact-form}), by yielding an identical result to the one obtained directly by algorithms evaluating determinants, including the recurrence in eq. (\ref{Hessenbergian recurrence}).
	\section{Linear Difference Equations with Variable Coefficients}\label{sec. Linear Difference Equations with Variable Coefficients}
	Nonhomogeneous linear difference equations with variable coefficients of order $p\ge 1$ in normal form (in short VC-LDEs($p$)) are defined by recurrences of the form 
	
\begin{equation} \label{VC-LDE(p)}
	y_t=\sum_{m=1}^{p}\phi_{m}(t)y_{t-m}+v_{t}, \vspace{0.1in}
\end{equation}
where $\phi_m(t)$, $1\le m\le p$ (variable coefficients) and $v_{t}$ (forcing term) are complex valued functions defined for all  $t\in \integers_{s+1}$ for some fixed $s\in\integers$.  
We further assume that $\phi_p(t)\not=0$ for all $t\in\integers_{s+1}$, ensuring that eq. (\ref{VC-LDE(p)}) is of order $p$. 

By virtue of the existence and uniqueness of the solutions for an initial value problem (see \cite{ElInt05}, Theorem 2.7., p. 66), given a sequence of initial values $\{y_{r+1-p},...,y_r\}$ for $r\ge s$ fixed, a solution of eq. (\ref{VC-LDE(p)}) is defined as an explicit representation of $y_t$ for $t\ge r+1$, written in terms of the initial values, the variable coefficients and the forcing term. This is due to the fact that all the parameters in eq. (\ref{VC-LDE(p)}), that is the variable coefficients and the forcing term, are well defined functions for any integer $t\ge r+1$, whenever $r\ge s$. As a consequence, the solution sequence is well defined on the entire set $\integers_{r+1-p}$, whereas   $y_{r+1}=\phi_1(r+1)y_r+\phi_2(r+1)y_{r-1}+...+\phi_p(r+1)y_{r+1-p}$. Hereafter, in absence of ambiguity, we adopt the conventional solution notation $y_t$ in place of the formal notation $y_{t,r}$ of the corresponding initial value problem, since $r$ for $r\ge s$ is assumed to be fixed. 
\subsection{A Unified Construction Process of a Fundamental Solution Set}\label{subs. Infinite Gaussian Elimination}
	In this Subsection we introduce a method for  constructing simultaneously the elements (sequences)  of a fundamental solution set associated with VC-LDEs($p$) (see the first $p$ column sequences in eq. (\ref{QHF}) below). This is based upon the row-finite system representation of VC-LDEs($p$) (see eq. (\ref{infinite system representation} below) and the Gaussian elimination algorithm applied to it, but implemented with a rightmost pivoting. Unlike the recursive construction of the individual solution sequences of eq. (\ref{VC-LDE(p)}), which necessarily takes on a sequence of $p$ initial values, the unified process, presented here, constructs simultaneously all the fundamental solution sequences starting from their $(p+1)$ term,  without assuming any initial values. This is due to the row canonical form of the reduced system coefficient matrix constructed by the infinite Gaussian elimination algorithm, as discussed below.
	
	Row-finite $\omega\times\omega$ (infinite) linear systems, in their general form, were first studied by Toeplitz~\cite{Toeplitz1909} (1909), who extended some fundamental results, established on finite linear systems, to cover row-finite ones. Such systems also represent linear difference equations of irregular order, covering the case when $\phi_p(t)=0$  for some $t\ge s+1$.
	Their solution representation was further developed by Fulkerson~\cite{Fulkerson1951} (1951). He devised and proved the existence of a reduced form, identified here as Fulkerson's row reduced echelon form (FRREF) for any arbitrary row-finite matrix\footnote{A row-finite matrix is an $\mathbb{\omega}\times \mathbb{\omega}$ infinite matrix, each row of which comprises a finite number of non-zero entries.}. An FRREF of a row-finite matrix, say $\A$, satisfies three out of four postulates of finite matrices in row reduced echelon form (RREF). It turns out that $\A$ and an FRREF of $\A$ are left associates. Left association generalizes the notion of row-equivalence of finite matrices (see \cite{Paraskevopoulos2014}). The RREF of a finite matrix is uniquely associated with the matrix, and therefore it is called row canonical form of the matrix. An FRREF of a row-finite matrix, $\A$, is a quasi-canonical form of $\A$, in the sense that two  FRREFs of a row-finite matrix differ only by a permutation of rows. As a consequence, the advantages gained by the row-canonical RREFs for  the solution representation of finite systems, are extended to row-finite ones by their quasi-canonical FRREF. 
	The arguments in \cite{Fulkerson1951} establishing the existence of an FRREF for a row-finite matrix, invoked the countable axiom of choice. 
	In contrast to the non-constructive nature of this axiom, a modified version of the Gauss-Jordan elimination algorithm has been recently introduced by Paraskevopoulos in~\cite{Paraskevopoulos2012}, which constructs the FRREF of an arbitrary row-finite matrix and called infinite Gauss-Jordan elimination (IGJE) algorithm. In a companion paper (see \cite{Paraskevopoulos2014}), he further developed the IGJE algorithm  capitalizing on the type and the form of the general solution of row-finite linear systems. If the dimension of the column-null space of the coefficient matrix is infinite, then 
	the IGJE algorithm yields a Schauder basis of the column-null space, relative to the Fr\'{e}chet metric, otherwise it yields a finite basis of vector spaces. The latter type of basis coincides with a fundamental solution set associated with the row-finite system representation of VC-LDEs($p$).
	
	The solution sequence constructed by the IGJE algorithm is solely the result of a rightmost pivot elimination strategy. As a counter example, employing a VC-LDE($1$), it is shown in the above cited reference that the conventional Gauss-Jordan elimination algorithm, implemented by a leftmost pivoting, fails to construct both a row-equivalent reduced matrix and the solution of the original VC-LDE($1$). This was the main barrier preventing researchers from choosing the IGJE algorithm for solving VC-LDEs and more generally row-finite systems.
	
	Eq. (\ref{VC-LDE(p)}) can be equivalenly represented by a row-finite linear system of the form: 
	{\small
		\begin{equation}\label{infinite system representation}
		\left[ 
		\begin{array}{ccccccccc}\
		\hspace{-0.05in}\phi _{p}(r+1) & \phi _{p-1}(r+1) & \hspace{-0.05in}\phi
		_{p-2}(r+1)\ ... & \hspace{-0.05in}\phi _{1}(r+1) & \hspace{-0.05in}-1 & 
		\hspace{-0.05in}0 & 0 & ... &  \\ 
		0 & \phi _{p}(r+2) & \hspace{-0.05in}\phi _{p-1}(r+2)\ ... & \hspace{-0.05in}%
		\phi _{2}(r+2) & \hspace{-0.07in}\phi _{1}(r+2) & \hspace{-0.05in}-1 & 0 & 
		... &  \\ 
		0 & 0 & \hspace{-0.05in}\phi _{p}(r+3)... & \hspace{-0.05in}\phi _{3}(r+3) & 
		\hspace{-0.05in}\phi _{2}(r+3) & \hspace{-0.05in}\phi _{1}(r+3) & -1 & ...	 &  \\ 
		\vdots & \vdots & \vdots & \vdots & \vdots & \vdots & \vdots &  & 
		\end{array}%
		\hspace{-0.05in}\right] \!\!\left[ \!\!%
		\begin{array}{c}
		y_{r+1-p} \\ 
		y_{r+2-p} \\ 
		y_{r+3-p} \\ 
		\vdots \\ 
		y_{r} \\ 
		y_{r+1} \\ 
		y_{r+2}\\ 
		\end{array}
		\!\!\right] \hspace{-0.05in}=-\left[ \!\!
		\begin{array}{c}
		v_{r+1} \\ 
		v_{r+2} \\ 
		\vdots
		\end{array}
		\!\!\right].  
		\end{equation}} 
	The coefficient matrix of eq.  (\ref{infinite system representation}), say $\A$, is row-finite with elements the variable coefficients of eq. (\ref{VC-LDE(p)}) evaluated at corresponding instances, starting at fixed  instance $r+1$. Since $\A$ is in lower echelon form,  the infinite Gaussian elimination part (IGE)~\footnote{A recursive alternative to the infinite Gaussian elimination algorithm is presented in \cite{PaRA04}.} of the IGJE algorithm is only needed for reducing $\A$ to its FRREF and the reduced matrix is unique,  denoted by $\FRREF(\A)$. Implementing the IGE with rightmost pivoting, the pivot elements are the $(-1)$s occupying the $p$th superdiagonal of $\A$. A step-by step construction of $\FRREF(\A)$ is presented in Appendix \ref{Appendix: IGE}. As shown there, the elimination process leads to a recurrence analogous to that  in eq. (\ref{VC-LDE(p)}). 
	
	The IGE algorithm  constructs the $\FRREF(\A)$, which is given by
	{\normalsize\begin{equation}\label{QHF}
		\FRREF(\A)=\left[\begin{array}{cccccccc}
		-\xi^{(p)}_{r+1,r} & -\xi^{(p-1)}_{r+1,r}&...& -\xi^{(1)}_{r+1,r}&1&0&0&...\\
		-\xi^{(p)}_{r+2,r}& -\xi^{(p-1)}_{r+2,r}&...& -\xi^{(1)}_{r+2,r}&0&1&0&... \\
		-\xi^{(p)}_{r+3,r} & -\xi^{(p-1)}_{r+3,r} &...& -\xi^{(1)}_{r+3,r}&0&0&1&... \\
		\vdots & \vdots & \vdots\vdots\vdots & \vdots&\vdots & \vdots & \vdots
		\end{array}\right].
		\end{equation}}
An explicit form of each sequence   $\xi^{(m)}_{t,r}$ for $1\le m\le p$ and $t\ge r+1$ in  eq. (\ref{QHF}) is established in Subsection \ref{Fundamental and General Homogeneous Solutions} (see eqs. (\ref{Exdef:xi^m}) and (\ref{Def1: Phi^m_k})) and conveniently visualized by VC-LDEs of order $p=2$ in Example \ref{ex.1} below.

For each $m\in\lbracket 1,p \rbracket$, we define the $m$th unit vector by:  $\Unit_{p+1-m}=[\delta_{p+1-m,j}]_{1\le j\le p}$, where $\delta_{i,j}$ is the Kronecker delta. We further define the augmented sequence $\xi^{(m)}_{.,r}\stackrel{\rm def}{=}[\Unit_{p+1-m}\concat[\xi^{(m)}_{r+1,r},\xi^{(m)}_{r+2,r},...]]'$, where $[... ]'$ stands  for transposition and ``$\nconcat$'' for concatenation of vectors.  As $\FRREF(\A)$ is a row-canonical form of $\A$ the results in \cite{Paraskevopoulos2014} entail that: First, for any $m\in \lbracket1,p\rbracket$ the  sequence $\xi^{(m)}_{.,r}$, is a solution of the homogeneous system $\A\ \y= \0$ associated with eq. (\ref{infinite system representation}) (see Example \ref{ex.1} below). Equivalently, the terms $\xi^{(m)}_{r+i,r}$, $i\ge 1$, of each individual $\xi^{(m)}_{.,r}$ form a homogeneous solution sequence of  eq. (\ref{VC-LDE(p)}) (when $v_{t}=0$ for all $t\ge s+1$), taking on the components of $\Unit_{p+1-m}$ as initial values.
Second, the set $\varXi_{r}=\{\xi^{(m)}_{.,r}\}_{1\le m\le p}$ is  a fundamental solution set associated with eq. (\ref{infinite system representation})) (or eq. (\ref{VC-LDE(p)}). 
Both aforementioned results are independently established in the next Subsection.
We remark that the fundamental solution sequence $\xi^{(1)}_{.,r}$ is the first element of $\varXi_{r}$, but its terms $\xi^{(1)}_{t,r}$ for $t\ge r+1$ occupy the $p$th opposite signed column of $\FRREF(\A)$, namely: $\xi^{(1)}_{.,r}=[\Unit_{p}\concat[\xi^{(m)}_{r+1,r},\xi^{(1)}_{r+2,r},...]]'=[0,0,...,1,\xi^{(1)}_{r+1,r}, \xi^{(1)}_{r+2,r},...]'$.  
	\begin{example}\label{ex.1}
	{\rm In this Example we apply the IGE algorithm to a row-finite system representation of VC-LDE($2$). The IGE algorithm applies to the coefficient matrix, $\A$, of eq. (\ref{infinite system representation}) for $p=2$ and simultaneously constructs the terms $\{\xi^{(1)} (t,r)\}_{t\geq r+1}$ and $\{\xi^{(2)} (t,r)\}_{t\geq r+1}$ of the two fundamental solutions, whereas their opposite signed terms occupy the second and the first columns of $\FRREF(\A)$, respectively. The first three terms of $\{\xi^{(1)} (t,r)\}_{t\geq r+1}$ are:\vspace{-0.25in}}
	\end{example}
	\begin{align*}
	\xi^{(1)}_{r+1,r}&=\phi_{1}(r+1),\ \ \ 
	\xi^{(1)}_{r+2,r}=\phi _{1}(r+1)\phi _{1}(r+2)+\phi _{2}(r+2),\\ \xi^{(1)}_{r+3,r}&=\phi_1(r+1)[\phi_1(r+2)\phi_1(r+3)+\phi_2(r+3)]+\phi_2(r+2)\phi_1(r+3),...\vspace{-0.06in}
	\end{align*}
	(see the Appendix \ref{Appendix: IGE} for a verification of  $\A\ [0,1,\xi^{(1)}_{r+1,r},\xi^{(1)}_{r+2,r},...]'=\0$). According to \cite{Paraskevopoulos2014} the above sequence augmented on its left by the unit vector $\Unit_2=[0,1]$ yields the first fundamental solution.	
	The algorithmic outcomes are expansions of the following Hessenbergians:\vspace{-0.12in}
	\begin{equation*}
	\begin{array}{l}
	\xi^{(1)}_{r+1,r}=\phi
	_{1}(r+1),\ \xi^{(1)}_{r+2,r}=\left \vert 
	\begin{array}{lc}
	\phi _{1}(r+1) & -1   \\ 
	\phi _{2}(r+2) & \phi _{1}(r+2)  
	\end{array}\right \vert,\  
	\xi^{(1)}_{r+3,r}=\left \vert \begin{array}{ccc}
	\phi_{1}(r+1) & -1 & 0 \\ 
	\phi_{2}(r+2) & \phi_{1}(r+2) & -1  \\
     0	& \phi_{2}(r+3) & \phi_{1}(r+3)
	\end{array}\right \vert,...\vspace{-0.1in}
	\end{array}
	\end{equation*}
	The corresponding terms of $\{\xi^{(2)} (t,r)\}_{t\geq r+1}$ are:\vspace{-0.03in}
	\begin{equation*}
	\xi^{(2)}_{r+1,r}=\phi_{2}(r+1),\ \ \
	\xi^{(2)}_{r+2,r}=\phi _{2}(r+1)\phi _{1}(r+2),\ \ \ \xi^{(1)}_{r+3,r}=\phi_2(r+1)[\phi_1(r+2)\phi_1(r+3)+\phi_2(r+3)],...\vspace{-0.02in}
	\end{equation*}
	This sequence augmented on its left by the unit vector $\Unit_1=[1,0]$ yields the second fundamental solution. In this case, the algorithmic outcomes are also expansions of  Hessenbergians of the form:\vspace{-0.02in}
	\begin{equation*}
	\begin{array}{l}
	\xi^{(2)}_{r+1,r}=\phi
	_{2}(r+1),\ \xi^{(2)}_{r+2,r}=\left \vert 
	\begin{array}{cc}
	\phi _{2}(r+1) & -1   \\ 
	0 & \phi _{1}(r+2)  
	\end{array}\right \vert, \ 
	\xi^{(2)}_{r+3,r}=\left \vert \begin{array}{ccc}
	\phi_{2}(r+1) & -1 & 0 \\ 
	0	& \phi_{1}(r+2) & -1  \\
	 0   & \phi_{2}(r+3) & \phi_{1}(r+3)
	\end{array}\right \vert,...\vspace{-0.05in}
	\end{array}
	\end{equation*}
The terms $\xi^{(m)}_{r+i,r}$ for $m=1,2$ and $i\ge 1$, of each individual $\xi^{(m)}_{.,r}$ can also be constructed recursively via eq. (\ref{VC-LDE(p)}) for $p=2$, when $v_{r+i}=0$ for all $i\ge 1$, taking on the components of $\Unit_{2}, \Unit_{1}$ as  initial values, respectively. 

\noindent These results led us to propose a generalized Definition of $\xi^{(m)}_{t,r}$ in eqs. (\ref{Exdef:xi^m}) and (\ref{Def1: Phi^m_k}) below.
	
Applying the same sequence of row elementary operations, used by the IGE for the row reduction of $\A$ to $\FRREF(\A)$, but now to the sequence  $\{-v_{r+i}\}_{i\ge 1}$, a particular solution sequence is constructed (see Appendix \ref{Appendix: IGE}). The process leads to a recurrence, which equivalently constructs the particular solution. This solution sequence is also explicitly represented by a Hessenbergian, but not a banded one, as shown in Proposition \ref{particular solution determinant representation}. The general solution is a linear combination of the fundamental solutions with coefficients arbitrary initial condition values $y_{r+1-m}=a_m$ for $1\le m\le p$ (that is the general homogeneous solution, see  Proposition \ref{Homogeneous solution}) plus the aforementioned particular solution (see eq. (\ref{nonhomogeneous solution2})). 
	
\subsection{Fundamental and General Homogeneous Solutions}\label{Fundamental and General Homogeneous Solutions}
	A fundamental set of solutions associated with VC-LDEs($p$) plays a significant role in the explicit representation of the Green's function, the companion matrix product (or the Casorati matrix) as well as the general solution of VC-LDEs($p$). The existence of such solution sets is theoretically guaranteed by the fundamental Theorem of VC-LDEs($p$) (see \cite{ElInt05} p. 74).
	As a consequence of the superposition  principle (see the previously cited reference) the homogeneous solution of VC-LDEs($p$) can be expressed as a linear combination of fundamental solutions whose coefficients are expressions involving the initial condition values.
	
	Given some $r\ge s$, the homogeneous linear difference equation associated with eq. (\ref{VC-LDE(p)}) is of the form\vspace{-0.06in}
	\begin{equation} \label{HLDE001}
	\displaystyle y_{t}=\sum_{m=1}^{p}\phi_m(t)y_{t-m},\ \   t\ge r+1,\vspace{-0.06in}
	\end{equation}
	that is eq. (\ref{VC-LDE(p)}) applied with forcing terms $\upsilon_t=0$ for all $t\ge s+1$. The linear difference operator associated with eq. (\ref{HLDE001}) is\vspace{-0.06in}
	\begin{equation}\label{eq. LDO}
	\Phi_{t}(B)=1-\sum \limits_{m=1}^{p}\phi_{m}(t)B^{m}, \ \ t\ge r+1,
	\end{equation}
	where $B$ is the backshift (or lag) operator. 
	Eq. (\ref{HLDE001}) can be equivalently rephrased as: $\Phi_{t}(B)y_t=\0$.
	
	One of the objectives of this Subsection is to provide an explicit solution function (or sequence) $y_t$ of eq. (\ref{HLDE001}) on $\integers_{r+1}$ for a fixed $r\ge s$, solely expressed in terms of the initial values $\{y_{r+1-m}\}_{1\le m\le p}$ and the variable coefficients $\phi_m(t)$.
	
	In the previous Subsection, the IGE algorithm was employed to construct simultaneously the fundamental solution sequences $\xi^{(m)}_{.,r}$ for $1\le m\le p$ associated with eq. (\ref{HLDE001}). However, the banded Hessenbergian representation of the fundamental solutions constructed by the IGE must be formally established and we are doing so here. 
	Thanks to the uniqueness of an initial value problem, the aforementioned result can be established by showing independently that each sequence $\xi^{(m)}_{.,r}$ for $1\le m\le p$, defined in eq. (\ref{Exdef:xi^m}) below, also solves eq. (\ref{HLDE001}), assuming $\Unit_{p+1-m}$ as initial condition vectors (see Proposition \ref{KSIs solutions}). Additionally, Theorem \ref{Theo: fundamental solution set})  establishes independently that these solutions form a fundamental solution set associated with eq. (\ref{HLDE001}).
	
	For each $m\in \lbracket 1, p\rbracket$, we define the two variable function $\xi^{(m)}_{t,r}$ for $(t,r)\in \integers_{s+1-p}\times\integers_{s}$  associated with eq. (\ref{HLDE001}) (or (\ref{eq. LDO})) as follows\vspace{-0.15in}
	\begin{equation}\label{Exdef:xi^m}
	\xi^{(m)}_{t,r}=\left\{
	\begin{array}{clr} 
	\det(\BGF^{(m)}_{t,r}) & {\rm if }\  r+1 \le t,   \vspace{-0.1in}\\
	&& ( r\ge s\  \text{and}\  t\ge s+1-p), \\
	1 & {\rm if \ }  t=r-m+1,  \\
	0 & {\rm elsewhere \ } \end{array}\right.
	\end{equation}
	where $\BGF^{(m)}_{r+1,r}\stackrel{\rm def}{=}[\phi_{m}(r+1)]$ (i.e., $\BGF^{(m)}_{r+1,r}$ is assigned to a $1\times 1$ matrix) and
	\begin{equation}\label{Def1: Phi^m_k}
	\BGF^{(m)}_{t,r}\!\! =\!\! 
	\left[\!\! 
	\begin{array}{ccccccc}
	\phi_{m}(r+1) & -1 &  &  &  &  &    \\ 
	\phi_{m+1}(r+2) & \phi _{1}(r+2) & \ddots
	&  &  &  &   \vspace{-0.01in} \\ 
	\vdots & \vdots & \ddots &\ \  \ddots &  &  &   \vspace{-0.06in}\\ 
	\phi_{p}(r+p+1-m) & \phi _{p-m}(r +p+1-m) & \ddots
	& \ddots &  &  &    \vspace{-0.01in}\\ 	
	& \vdots & \ddots & \ddots & \ \ \ \ \ddots  &  &      \vspace{-0.06in}\\
	& \phi _{p}(r +1+p) & \ddots
	& \ddots & \ddots &\ddots &    \\ 
	&  & \ddots & \ddots & \ddots &  \ddots &\ddots  \\ 
	&  &  & \phi _{p}(t-1) & \phi _{p-1}(t-1)& \cdots  &  -1 \\ 
	&  &  &  & \phi _{p}(t) & \cdots
	& \phi _{1}(t)\end{array}\!\! \right].
	\end{equation} 
	Here and in what follows empty spaces in a matrix have to be replaced by zeros. 
	
	The matrices $\BGF^{(m)}_{t,r}$ for $t\ge r+1$ are banded lower Hessenberg matrices of order $k=t-r$. After a large enough $t$ ($t\ge r+p+1-m$), $\BGF^{(m)}_{t,r}$ admits a fixed total bandwidth $(p+1)$, noticing that the matrices $\BGF^{(m_1)}_{t,r}$ and $\BGF^{(m_2)}_{t,r}$ for $m_1\not= m_2$ differ only in their first column (see eq. (\ref{Def1: Phi^m_k})). For a fixed $r$ with $r\ge s$, $\xi^{(m)}_{t,r}$ is considered as an one variable function (or sequence) in $t$ with $t\ge r+1-p$, whose first $p$ values are:\vspace{-0.05in}
	\begin{equation}\label{xi^m initial values}
	\xi^{(m)}_{r+1-m,r}=1\  {\rm and}\ \xi^{(m)}_{r+1-i,r}=0,\ {\rm whenever}\  i\not= m.
	\end{equation}
We shall also use the sequence notation $\xi^{(m)}_{.,r}=\{\xi^{(m)}_{t,r}\}_{t\in \integers_{r+1-p}}$ for a fixed $r\ge s$.	Eqs. (\ref{xi^m initial values}) describe the initial condition unit vector $\Unit_{p+1-m}=[\xi^{(m)}_{r+1-p,r}, \xi^{(m)}_{r+2-p,r},...,\xi^{(m)}_{r,r}]$. If $m=1$, then $[\xi^{(1)}_{r+1-p,r}, \xi^{(1)}_{r-1,r},...,\xi^{(1)}_{r,r}]=[0,0,...,1]=\Unit_p$. 
Some useful values $\xi^{(m)}_{t,r}$ for $1\le m\le p$ and any $r\ge s$ are:\\ 
$\xi^{(1)}_{r,r}=\xi^{(m)}_{r-m+1,r}=1,\xi^{(1)}_{r,r+i}=0\ (i>0), \xi^{(1)}_{r+1,r}=\phi_1(r+1),  \xi^{(m)}_{r+1,r}=\phi_m(r+1),  \xi^{(m)}_{r+2,r}=\left|\begin{array}{ll} \phi_m(r+1) & -1\\
\phi_{m+1}(r+2) & \phi_1(r+2)\end{array}\right|$. 
 By an abuse of terminology, the two variable functions $\xi^{(m)}_{t,r}$ defined on $\integers_{s+1-p}\times\integers_{s}$ in eq. (\ref{Exdef:xi^m}) will be also referred to  as banded Hessenbergians.
Taking into account that $\BGF^{(m)}_{t,r}$ in eq. (\ref{Def1: Phi^m_k}) is a banded Hessenberg matrix, the matrix ${\Hes}_{t-r}$ in eq. (\ref{Modified Hessenberg Matrix}) can be identified with $\BGF^{(m)}_{t,r}$ via the assignment
	\begin{equation}\label{assignment}
	c_{i,j}=\left\{\begin{array}{cl}
	\phi_{i-1+m}(r+i) & {\rm if}\  j=1 \ {\rm and}\  1\le i\le p+1-m\\
	\phi_{i+1-j}(r+i) & {\rm if}\  2\le j\le i \le t-r \ {\rm and}\ 1\le i-j+1\le p  \\
	1 & {\rm if} \ j=i+1        \\
	0 & {\rm elsewhere},
	\end{array}\right. 
	\end{equation}
	provided that $m\in\lbracket 1,p\rbracket$   is  fixed, each time we referred to eq. (\ref{assignment}). 
	\begin{proposition}\label{KSIs solutions}
		Let $m\in\lbracket 1,p\rbracket$. The sequence $\{\xi^{(m)}_{t,r}\}_{t\in \integers_{r+1}}$ for any arbitrary but fixed $r\ge s$, defined in eq. {\rm (\ref{Exdef:xi^m})}, solves  eq. {\rm (\ref{HLDE001})}, assuming the initial condition unit vector $[\xi^{(m)}_{r+1-p,r}, \xi^{(m)}_{r+2-p,r},...,\xi^{(m)}_{r,r}]=\Unit_{p+1-m}$.
	\end{proposition}
	\begin{proof}
	Taking into account that $c_{i,i+1}=1$, the recurrence in eq. (\ref{modified recurrence2}) applied for $i=k=t-r$ (the order of the matrix) takes the form:\vspace{-0.1in} 
	\begin{equation}\label{eq. recurrence}
		|\Hes_{i}|=c_{i,1}|\Hes_0|+c_{i,2}|\Hes_1|+...+c_{i,i-1}|\Hes_{i-2}|+c_{i,i}|\Hes_{i-1}|.\vspace{-0.05in}
	\end{equation}	
	We examine the following cases:\\
	i) Let $1\le i\le p+1-m$. This inequality can be equivalently written as  $r+1\le t\le r+p+1-m$, which means that $\BGF^{(m)}_{t,r}$ in eq. (\ref{Def1: Phi^m_k}) is a full lower Hessenberg matrix. We can equivalently write the above inequality as $t=r+p+1-m-l$, whenever $0\le l\le p-m$. 
	As $i=t-r=p+1-m-l$, it follows from eq. (\ref{assignment}) that: $c_{i,1}=c_{t-r,1}=\phi_{p+1-m-l-1+m}(r+p+1-m-l)=\phi_{p-l}(t)$ and $c_{i,2}=c_{t-r,2}=\phi_{p+1-m-l-2+1}(r+p+1-m-l)=\phi_{p-m-l}(t)$. Proceeding in this way, the remaining values of $c_{i,j}$ for $3\le j\le p$ are:
	$c_{t-r,3}=\phi_{t-r-2}(t)=\phi_{p-m-l}(t),..., c_{t-r,t-r-1}=\phi_{2}(t), c_{t-r,t-r}=\phi_{1}(t)$. We then replace the above results to the right-hand side of eq. (\ref{eq. recurrence}) coupled with: $|\Hes_{0}|=1=\xi^{(m)}_{r+1-m,r}$, $|\Hes_{1}|=\phi_m(r+1)=\xi^{(m)}_{r+1,r}$,..., $|\Hes_{t-2-r}|=\xi^{(m)}_{t-2,r}$ and $|\Hes_{t-1-r}|=\xi^{(m)}_{t-1,r}$. As the left-hand side of eq. (\ref{eq. recurrence}) can be replaced  with $|\Hes_{i}|=|\Hes_{t-r}|=\xi^{(m)}_{t,r}$, it  takes the form \vspace{-0.02in}
	\[\begin{array}{lcll}
		\xi^{(m)}_{t,r}
		\!\!\!\!\!\!&=\!\underbrace{\phi_{p-l}(t)\xi^{(m)}_{r+1-\!m,r}\!\!+\!\phi_{p-l-\!1}(t)\xi^{(m)}_{r+2\!-m,r}\!\!+\!...\!+\!\phi_{p-l+1\!-m}(t)\xi^{(m)}_{r,r}}\!\!&\!\!\!\!+\phi_{p-l-m}(t)\xi^{(m)}_{r+1,r}\!\!+\!...\!+\!\phi_2(t)\xi^{(m)}_{t-2,r}\!+\!\phi_1(t)\xi^{(m)}_{t-\!1,r},\\
		& \text{initial values}& &\vspace{-0.06in} 
		\end{array}\]
	where the values of $\xi^{(m)}_{r+1-m,r}$ up to and including $\xi^{(m)}_{r,r}$ are initial values defined in eq. (\ref{xi^m initial values}), that is $\xi^{(m)}_{r+1-m,r}=1$ and the remaining initial values are zero, whenever $m\not=1$, while if $m=1$, then  $\xi^{(1)}_{r,r}=\xi^{(m)}_{r+1-m,r}=1$.
	
	\noindent ii) Let $ i > p+1-m$. As $t-r-1+m>p$, in view of eq. (\ref{Def1: Phi^m_k}), the matrix  $\BGF^{(m)}_{t,r}$ is a lower banded Hessenberg matrix. 
	This means that the first values of  $c_{i,j}$ in eq. (\ref{eq. recurrence})  starting from $c_{i, 1}$ up to and including $c_{i, i-p}$ are zero. Applying eq. (\ref{assignment}),  we have: $c_{i,i}=\phi_{1}(t), \ c_{i,i-1}=\phi_{2}(t),...,c_{i,i-p+1}=\phi_{p}(t)$, which is in accord with eq.  (\ref{assignment}), since:
	$c_{t-r,1}=...=c_{t-r,t-r-p}=0$.
	Substituting the above values of $c_{i,j}$ and replacing $\det(\Hes_{t-r-j})$ with $\xi^{(m)}_{t-j,r}$ for $0\le j \le p$ in the recurrence (\ref{modified recurrence2}), the latter takes the form:\vspace{-0.02in} 
	\[	\xi^{(m)}_{t,r}
		=\phi_p(t)\xi^{(m)}_{t-p,r}+\phi_{p-1}(t)\xi^{(m)}_{t-p+1,r}+...+\phi_2(t)\xi^{(m)}_{t-2,r}+\phi_1(t)\xi^{(m)}_{t-1,r}.\vspace{-0.01in} 
		\]
	In both cases $\xi^{(m)}_{.,r}$ satisfies eq. (\ref{HLDE001}), as required.
	\end{proof}
We conclude from the uniqueness of the initial value problem that the fundamental solutions constructed by the IGE algorithm or by recursion must all coincide with the banded Hessenbergian ones.
	\begin{theorem}\label{Theo: fundamental solution set}
	For each fixed $r\ge s$, the set   $\varXi_{r}=\{\xi^{(1)}_{.,r},\xi^{(2)}_{.,r},...,\xi^{(p)}_{.,r}\}$, consisting of $p$  sequences defined over the same domain  $\integers_{r+1-p}$, is a fundamental set of solutions associated with  eq. \emph{(\ref{HLDE001})}.  
	\end{theorem}
	\begin{proof}
	Notice first that  each sequence  $\xi^{(m)}_{.,r}$ (the $m$th element of  $\varXi_r$) starts with the value $\xi^{(m)}_{r+1-p,r}$, whence the domain of the function $\xi^{(m)}_{.,r}$ is  $\integers_{r+1-p}$. Taking into account that the elements of  $\varXi_{r}$ are solutions of eq. {\rm (\ref{HLDE001})} (see Proposition \ref{KSIs solutions}), it suffices to verify that the set $\varXi_{r}$ is linearly independent. Equivalently, it must be shown that the Casoratian of the matrix\vspace{-0.08in} 
		\begin{equation}\label{Casorati matrices}
	\BXI_{t,r}=\left[\begin{array}{cccc}
	\xi^{(1)}_{t,r} & \xi^{(2)}_{t,r}&...& \xi^{(p)}_{t,r}\vspace{0.05in}\\
	\xi^{(1)}_{t-1,r}& \xi^{(2)}_{t-1,r}&...& \xi^{(p)}_{t-1,r}\vspace{0.05in}\\
	\vdots & \vdots & \vdots\vdots\vdots & \vdots\\
	\xi^{(1)}_{t-p+1,r} & \xi^{(2)}_{t-p+1,r} &...& \xi^{(p)}_{t-p+1,r} \end{array}\right]
		\end{equation}
	associated with the solution set $\varXi_{r}$ is nonzero for all $t\ge r$. Definition (\ref{Exdef:xi^m}) entails that the matrix $\BXI_{r,r}$ is the identity matrix of order $p$, that is $\BXI_{r,r}=\BI_p$. Therefore the first Casoratian $|\BXI_{r,r}|$ of the set $\varXi_{r}$ is $|\BXI_{r,r}|=1\not=0$. It follows that $|\BXI_{t,r}|\not=0$ for all $t\ge r$ (see Lemma 1.3 in \cite{MiLD68}, applied for $a=r-p+1$) and the set $\varXi_{r}$ is linearly independent. That is $\varXi_{r}$ is a fundamental set of solutions of eq. (\ref{HLDE001}).
	\end{proof}
	For any $r\ge s$, the dimension of the homogeneous solution space of $\Phi_{t}(B)$ for $t\ge r+1$ in eq. (\ref{eq. LDO}) is $p$. As $\varXi_{r}$ spans this space and  $\card(\varXi_{r})=p$, it follows that the set $\varXi_{r}$ is a basis of the column-null space of the coefficient matrix in eq. (\ref{infinite system representation}).
	Besides, as $|\BXI_{t,r}|\not=0$ for all $t\ge r$, Corollary \ref{cor: invertible} below, follows immediately. 
	\begin{corollary}\label{cor: invertible}
	The Casorati matrix $\BXI_{t,r}$ in eq. {\rm (\ref{Casorati matrices})} is invertible \emph{(}or non-singular\emph{)} for all $t\ge r$ and any $r\ge s$.
	\end{corollary}
	Corollary (\ref{cor: invertible}) is also established independently, as a consequence of Theorem \ref{theo: explicit PCM} in Subsection \ref{subs. Companion Matrix Product}. 		
	\begin{proposition} \label{Homogeneous solution}
	Let $r\ge s$. A solution sequence $\{y_t\}_{t\in \integers_{r+1}}$ of eq. {\rm (\ref{HLDE001})},  can be explicitly and uniquely expressed in terms of any sequence of  prescribed values, say $\{y_{r+1-m}\}_{1\le m\le p}$,  as\vspace{-0.05in}
		\begin{equation} \label{homogeneous solution0} y_t=\displaystyle\sum_{m=1}^{p}\xi^{(m)}_{t,r}y_{r+1-m}.\vspace{-0.02in}
		\end{equation}
	\end{proposition}
	\begin{proof}
	As $\varXi_{r}$ (defined in Theorem \ref{Theo: fundamental solution set}  is a fundamental set of solutions whose elements are defined over $\integers_{r+1-p}$, that is  $\xi^{(m)}_{.,r}=\{\xi^{(m)}_{t,r}\}_{t\in \integers_{r+1-p}}$, there exist unique scalars $a_m$ for $1\le m\le p$ such that: \vspace{-0.05in}
	\begin{equation} \label{Homogeneous solution2}
		y_t=\displaystyle\sum_{m=1}^{p}a_m\xi^{(m)}_{t,r},\  {\rm for\ all }\ \ t\ge r+1-p.\vspace{-0.02in}
	\end{equation}
	Taking into account that  $[\xi^{(m)}_{r+1-p,r},\xi^{(m)}_{r+2-p,r},...,\xi^{(m)}_{r,r}]=\Unit_{p+1-m}$ for $1\le m\le p$ (see eqs. (\ref{xi^m initial values})), applying eq. (\ref{Homogeneous solution2}) for $t=r+1-j$ with $1\le j\le p$, we have: 
		$y_{r+1-j}=\sum_{m=1}^{p}a_{m}\xi^{(m)}_{r+1-j,r}=a_j,\ \ \text{for all}\ \ j=1,...,p$
	as claimed.
	\end{proof}
	\subsection{Principal Determinant}
	\label{sub:Principal Determinant} 
The building element for the remaining results of this paper is the first banded Hessenbergian function $\xi^{(1)}_{t,r}$. A generic feature of  $\xi^{(1)}_{t,r}$ is shown in Proposition \ref{prop. global feature}. It states that the terms $\{\xi^{(m)}_{t,r}\}_{t\ge r+1}$ of the  fundamental solution $\xi^{(m)}_{.,r}$ for any  $2\le m\le p$ can be expressed in terms of $\xi^{(1)}_{t,r}$ and the variable coefficients. This result yields an explicit representation of the general homogeneous solution of eq. (\ref{HLDE001}) exclusively in terms of $\xi^{(1)}_{t,r}$, the variable coefficients and the initial conditions (see eq. (\ref{eq. general homogeneous solution})).
	\begin{definition}\label{Principal determinant}
	The principal matrix associated with the difference operator in eq. {\rm (\ref{eq. LDO})} is denoted by $\BGF_{t,r}$ and is defined by setting $m=1$ in the first branch of eq. \emph{(\ref{Exdef:xi^m})}, that is
	\begin{equation}\label{eq. Principal Matrix}
		\BGF_{t,r}=
		\left[\!\! 
		\begin{array}{ccccccc}
		\phi_{1}(r+1) & -1 &  &  &  &  &    \\ 
		\phi_{2}(r+2) & \phi _{1}(r+2) & \ddots
		&  &  &  &    \vspace{-0.03in}\\ 
		\vdots & \vdots & \ddots &\ \  \ddots &  &  &   \vspace{-0.05in}\\ 
		\phi_{p}(r+p) & \phi _{p-1}(r +p) & \ddots
		& \ddots &  &  &    \vspace{-0.03in}\\ 	
		& \vdots & \ddots & \ddots & \ \ \ \ \ddots  &  &      \vspace{-0.05in}\\
		& \phi_{p}(r+p+1) & \ddots
		& \ddots & \ddots &\ddots &    \\ 
		&  & \ddots & \ddots & \ddots &  \ddots &\ddots  \\ 
		&  &  & \phi_{p}(t-1) & \phi_{p-1}(t-1)& \cdots  &  -1 \\
		&  &  &  & \phi _{p}(t) & \cdots
		& \phi_{1}(t)\end{array}\!\! \right] 
		\end{equation} 
		The determinant of $\BGF_{t,r}$ is called principal determinant and denoted by\emph{:}  $\xi_{t,r}\stackrel{\rm def}{=}\xi^{(1)}_{t,r}$. The two variable function $\xi_{t,r}$ on $\integers_{s+1-p}\times\integers_{s}$, defined in eq. \emph{(\ref{Exdef:xi^m})} for $m=1$, will be referred to as principal determinant function.
	\end{definition}
	In order to simplify proofs, in place of banded Hessenbergians, we shall use full Hessenbergians of order $(t-r)$, as follows: $\xi^{(m)}_{t,r}=\det(\BGF^{(m)}_{t,r})$, where
	\begin{equation}\label{complete Hessenbergian}
	\BGF_{t,r}^{(m)}=\left[
	\begin{array}[c]{cccccc}
	{\small \phi}_{m}{\small (r+1)} & {\small -1} &  &  &  & \\
	{\small \phi}_{m+1}{\small (r+2)} & {\small \phi}_{1}{\small (r+2)} & \ddots &
	&  & \\
	{\small \phi}_{m+2}{\small (r+3)} & {\small \phi}_{2}{\small (r+3)} & \ddots &
	\ddots &  & \vspace{-0.04in}\\
	\vdots & \vdots & \ddots & \ddots & \ddots & \vspace{-0.0in}\\
	{\small \phi}_{m+t-r-2}{\small (t-1)} & {\small \phi}_{t-r-2}{\small (t-1)} &
	{\small \cdots} & {\small \phi}_{2}{\small (t-1)} & {\small \phi}_{1}{\small (t-1)} & -1\\
	{\small \phi}_{m+t-r-1}{\small (t)} & {\small \phi}_{t-r-1}{\small (t)} &
	{\small \cdots} & {\small \phi}_{3}{\small (t)} & {\small \phi}_{2}
	{\small (t)} & {\small \phi}_{1}{\small (t)}
	\end{array}
	\right],
	\end{equation}
	coupled with the convention: 
	$\phi_{i}(t)=0, {\rm \ whenever\ } i>p\ {\rm and\ for\ all\ } t\in\integers_{r+1}.$
	In accordance this convention full Hessenbergians turn into banded ones and in this case the matrices in eqs. (\ref{Def1: Phi^m_k}) and (\ref{complete Hessenbergian}) coincide. Moreover, we remark that any term, say $\phi_l(n)$, of the first column of $\BGF_{t,r}^{(m)}$ satisfies: $n-l=r+1-m$.
	
	First we state and prove the following Lemma:
	\begin{lemma}\label{Lemma cofactor}
		{\rm i)}  The cofactor of the coefficient $\phi_{m+i-1}(r+i)$ for $1\le i\le t-r$ in the first column of $\BGF^{(m)}_{t,r}$ in eq. {\rm (\ref{complete Hessenbergian})} coincides with $\xi_{t,r+i}$. 
		{\rm ii)} The cofactor of the coefficient $\phi_{n}(t)$ for $1\le n\le t-r-1$, in the last row
		of $\xi_{t,r}^{(m)}$ coincides with $\xi^{(m)}_{t-n,r}$. In particular, If $n=m+t-r-1$, then $\xi^{(m)}_{t-n,r}=\xi^{(m)}_{r+1-m,r}=1$.
	\end{lemma}
	\begin{proof}
		i) Let us call $M^{(0)}_{1,1}$ the minor of the $(1,1)$ entry of $\BGF^{(m)}_{t,r}$ (occupied by $\phi_{m}(r+1)$). $M^{(0)}_{1,1}$ is the determinant of the $(t-r-1)\times(t-r-1)$ submatrix of $\BGF^{(m)}_{t,r}$, obtained by deleting the first row and column of $\BGF^{(m)}_{t,r}$. It follows directly that $M^{(0)}_{1,1}=\xi_{t,r+1}$. The latter result multiplied by $(-1)^{1+1}$ yields the cofactor of $\phi_{m}(r+1)$, that is:  $Cof[\phi_{m}(r+1)]=(-1)^{1+1}\xi_{t,r+1}=\xi_{t,r+1}$. In view of eq. (\ref{complete Hessenbergian}), it remains to show the assertion for any $i$ such that $1<i\le m+t-r-1$. 
		The minor of the $(i,1)$ entry of $\BGF^{(m)}_{t,r}$ (occupied by $\phi_{m+i-1}(r+i)$), call it as $M^{(0)}_{i,1}$,  is obtained by deleting the first column and the $i$th row of $\BGF^{(m)}_{t,r}$, that is:\vspace{0.12in}
		\[\hspace{-6.in} M^{(0)}_{i,1}=  \]
		{\small \begin{equation*}
			\left\vert
			\begin{array}[c]{lcccccccc}%
			\ \ \ \ \ \  -1 &  &  &  & &&&&\\
			\phi_{1} (r+2) & -1 &  &  & \\
			\phi_{2} (r+3)  &  \phi_{1} (r+3) &  & \vspace{-0.02in} \\  
			\ \ \ \ \ \vdots & \vdots &  &  &  &  &  & \vspace{-0.01in}\\
			\phi_{i-2} (r+i-1) & \phi_{i-3} (r+i-1) & ... & -1 & 0 &  &   & \\
			&&&&&&&&\vspace{-0.02in}\\
			\phi_{i} (r+i+1) & \phi_{i-1} (r+i+1) & ... & \phi_{2} (r+i+1) & \phi_{1} (r+i+1) & -1 & \vspace{-0.02in}\\
			\ \ \ \ \  \vdots & \vdots &  & \vdots &\vdots & &  &   \vspace{-0.01in}\\
			\phi_{t-r-2} (t-1) &  \phi_{t-r-3}(t-1) &\cdots &\phi_{t-1-r-i}(t-1) &\phi_{t-2-r-i}(t-1) & \phi_{t-3-r-i}(t-1) & \cdots 
			&  \phi_{1} (t-1) & -1\\
			\phi_{t-r-1} (t) &  \phi_{t-r-2} (t) &... &\phi_{t-r-i}(t)  &\phi_{t-1-r-i}(t) &\phi_{t-2-r-i}(t)&...& \phi_{2} (t) &  \phi_{1} (t)
			\end{array}
			\right\vert. \vspace{0.05in}
			\end{equation*}}
		We observe that the elements in the main diagonal  of $M^{(0)}_{i,1}$ up to and including the entry $(i-1,i-1)$ are occupied by $(-1)$s, while the main diagonal element next to $(i-1,i-1)$ is occupied by $\phi_{1}(r+i+1)$. In what follows we compute $M^{(0)}_{i,1}$ recursively.
		Deleting the first row and column of $M^{(0)}_{i,1}$ the determinant of the resulting submatrix is the minor of the $(1,1)$ entry of $M^{(0)}_{i,1}$, denoted by: $M^{(1)}_{i,1}$.  
		Proceeding in this way we denote  $M^{(j)}_{i,1}$ for $1\le j\le i-1$, the minor of the $(1,1)$ entry of $M^{(j-1)}_{i,1}$, that is the $(j,j)$  element of $M^{(0)}_{i,1}$, occupied by $(-1)$.  
		As the first row of $M^{(j-1)}_{i,1}$  is $(-1,0,...,0)$ for any  $j\in\lbracket1,i-1\rbracket$, expanding $M^{(j-1)}_{i,1}$ along the first row, we obtain the recurrence:
		\begin{equation}\label{recurrence}
		M^{(j-1)}_{i,1}=(-1)M^{(j)}_{i,1},\ \ \ \ \   1\le j\le i-1.
		\end{equation}
		In particular, if $j=1$ the recurrence in eq.  (\ref{recurrence}) gives: $M^{(0)}_{i,1}=(-1)M^{(1)}_{i,1}$.
		If $j=i-1$, the minor of the $(1,1)$ entry of  $M^{(i-2)}_{i,1}$, that is the $(i-1,i-1)$ element of $M^{(0)}_{i,1}$  (occupied by the last $(-1)$ in the main diagonal of $M^{(0)}_{i,1}$), is given by:\vspace{-0.16in}
		\begin{equation*}
		M^{(i-1)}_{i,1}=\left\vert\begin{array}
		[c]{ccccc}
			\phi_{m}(r+i+1) &  -1 &  &   & \\
		\phi_{m+1}(r+i+2) & \phi_{1}(r+i+2) & \ddots &
		& \vspace{-0.02in}\\
		\vdots & \vdots  & \ddots & \ddots & \vspace{-0.01in}\\
		\phi_{t-2-r-i}(t-1) & \phi_{t-3-r-i}(t-1) & \cdots &  \phi_{1} (t-1) & -1\\
		\phi_{t-1-r-i}(t) &\phi_{t-2-r-i}(t)&\cdots & \phi_{2} (t) &  \phi_{1} (t)
		\end{array}
		\right\vert=\xi_{t,r+i}.
		\end{equation*}	
	The recurrence (\ref{recurrence}) yields:
		$
		M^{(0)}_{i,1}=(-1)M^{(1)}_{i,j}=(-1)^2M^{(2)}_{i,j} =...=(-1)^{i-1}M^{(i-1)}_{i,j} = (-1)^{i-1}\xi_{t,r+i}$.
		Accordingly the cofactor of the $(i,1)$ entry of $\BGF^{(m)}_{t,r}$ is $Cof[\phi_{i-1}(r+i)]=(-1)^{i+1}M^{(0)}_{i,1}=(-1)^{i+1}(-1)^{i-1}\xi_{t,r+i}=\xi_{t,r+i}$, as claimed.\\
		ii) By working as in part (i), but expanding $\BGF^{(m)}_{t,r}$ along the last row (instead of the first column),  the result follows.	In particular, if $n=m+t-r-1$, then since the number of   $(-1)$s in the superdiagonal of $\BGF^{(m)}_{t,r}$ in eq. (\ref{complete Hessenbergian}) is $(t-r-1)$ and $\phi_{m+t-r-1}(t)$ occupies its $(t-r,1)$ entry, it follows that: $Cof[\phi_{m+t-r-1}(t)]=(-1)^{t-r+1}(-1)^{t-r-1}=(-1)^{2(t-r)}=1$, as expected.
	\end{proof}
	As a direct consequence of Lemma \ref{Lemma cofactor}(ii), the cofactor expansion of $\xi^{(m)}_{t,r}$ along the last row gives
	\begin{equation}\label{eq. ksi recurrence}
	\xi^{(m)}_{t,r}=\phi_1(t)\xi^{(m)}_{t-1,r}+\phi_2(t)\xi^{(m)}_{t-2,r}+...+\phi_p(t)\xi^{(m)}_{t-p,r}\ \text{for}\ t\ge r+1,
	\end{equation}
	which re-establishes the second part in the proof of Proposition \ref{KSIs solutions}. In Proposition \ref{prop. Casoratian Property} of Appendix \ref{Appendix: ProofsOfPropositions}, the recurrence in eq. (\ref{eq. ksi recurrence}) is employed to show from first principles a property of Casoratians associated with VC-LDEs($p$)  (see eq. (2.5), p. 39, in \cite{MiLD68}), that is, they satisfy a first order linear recurrence. 
	
	\begin{proposition} \label{prop. global feature}
		The terms  $\xi^{(m)}_{t,r}$ for $1\le m\le p$ and $t\ge r+1$ of each fundamental solution $\xi^{(m)}_{.,r}$ can be expressed in terms of the principal determinant function as:
		\begin{equation}\label{xi^m_k}
		\xi^{(m)}_{t,r}= \displaystyle\sum_{i=1}^{p-m+1}\phi_{m-1+i}(r+i)\xi_{t,r+i}.
		\end{equation} 
	\end{proposition}
	\begin{proof}
		Applying Lemma \ref{Lemma cofactor}(i) to $\xi^{(m)}_{t,r}$ in eq. (\ref{Def1: Phi^m_k}), the cofactor expansion of $\xi^{(m)}_{t,r}$ along the first column gives $\xi^{(m)}_{t,r}=\phi_{m}(r+1)\xi_{t,r+1}+\phi_{m+1}(r+2)\xi_{t,r+2}+...+\phi_{p}(r+p+1-m)\xi_{t,r+p+1-m}$, whence
		$$\xi^{(m)}_{t,r}=\displaystyle\sum_{i=m}^{p}\phi_{i}(r+i+1-m)\xi_{t,r+i+1-m}= \displaystyle\sum_{i=1}^{p+1-m}\phi_{m-1+i}(r+i)\xi_{t,r+i},\vspace{-0.02in}$$ 
		as required. 
	\end{proof}
	Taking into account that $\xi_{t,r}= \sum_{i=1}^{p}\phi_{i}(r+i)\xi_{t,r+i}$, by applying eq. (\ref{xi^m_k}) to eq.  (\ref{homogeneous solution0}), we obtain explicit representations of the general solution $y_t$ in eq. (\ref{HLDE001}) solely in terms of the principal determinant function and any sequence of prescribed values  $\{y_{r+1-m}\}_{1\le m\le p}$ for a fixed $r\ge s$, as follows:
\begin{equation}\label{eq. general homogeneous solution}
y_{t}=\displaystyle\sum_{m=1}^{p}\sum_{i=1}^{p+1-m}\phi_{m-1+i}(r+i)\xi_{t,r+i}y_{r+1-m}=\xi_{t,r}y_{r}+\displaystyle\sum_{m=2}^{p}\sum_{i=1}^{p+1-m}\phi_{m-1+i}(r+i)\xi_{t,r+i}y_{r+1-m} \   \  \text{for all} \  t\ge r+1. \vspace{-0.02in}
\end{equation}
	\subsection{Companion Matrix Product}\label{subs. Companion Matrix Product}
	We show in the current Subsection that the elements of the product of companion matrices associated with the difference operator in  eq. (\ref{eq. LDO}) can be explicitly represented by banded Hessenbergians. 
	
	Let $t\in\integers_{s+1}$. The companion matrix of order $p$ is given by\vspace{-0.02in}
	\begin{equation}\label{companion matrix}
	\BGAMMA_{t}=\left[\begin{array}{ccccc}
	\phi_1(t)& \phi_2(t) & ... & \phi_{p-1}(t) & \phi_p(t)\\
	1& 0 &... & 0 & 0 \\
	0& 1 &... & 0 & 0 \vspace{-0.05in}\\
	.&.&...&.&.\vspace{-0.1in}\\
	.&.&...&.&.\vspace{-0.1in}\\
	.&.&...&.&.\\
	0 & 0 &...& 1 & 0\end{array}\right].\vspace{-0.02in}
	\end{equation}
Eq. (\ref{HLDE001}) for $r\ge s$ can be expressed, as a vector equation, by:\vspace{-0.01in} 
	\begin{equation} \label{companion matrices1}
	\y_{t} = \BGAMMA_{t}\ \y_{t-1},\ \  t\ge r+1. \vspace{-0.01in}
	\end{equation}
	 An extended Definition of the companion matrix product, including the case $t=r$,  is given by
	\begin{equation}\label{product of companion matrices1}
	\F_{t,r}\stackrel{\rm def}{=}\left\{\begin{array}{cl}
	\BGAMMA_{t}\BGAMMA_{t-1}...\BGAMMA_{r+1},  & \text{if}\ t\ge r+1 \\
	\BI_p, & \text{if}\ t= r.
	\end{array}\right. 
	\end{equation}
	$\F_{t,r}$ is invertible, since $\BGAMMA_{i}$ is invertible for all $i\in \integers_{s+1}$.  As $\F_{r,r}=\BI_p$, we further conclude that $\F_{t,r}$  is invertible for all $t\in \integers_{r}$ and any $r\ge s$. Taking into account that the matrix multiplication is non-commutative, we can alternatively use the condense notation: $\BGAMMA_{t}\BGAMMA_{t-1}...\BGAMMA_{r+1}=\prod_{i=r}^{t-1}\BGAMMA_{t-i+r}$.
	
	Let $\y_{r}=[y_{r},y_{r-1},...,y_{r-p+1}]'$ be an initial condition vector associated with eq. (\ref{HLDE001}). Then the unique vector solution of the corresponding initial value problem associated with  eq. (\ref{HLDE001}) can also be described by the vector equation: \vspace{-0.05in} 
	\begin{equation} \label{product of companion matrices2}
	\y_{t}=\F_{t,r}\ \y_{r} \  \  \text{for} \  t\ge r. 
	\end{equation}
	This is an alternative interpretation to the solution in eq. (\ref{homogeneous solution0}). If $t=r$, then $\y_r=\BI_p\ \y_{r}$, as expected.
	
	In all that follows $\BXI_{t,r}$ stands for the Casorati matrix defined in eq. (\ref{Casorati matrices}). Before proving the main result of this Subsection in Theorem \ref{theo: explicit PCM}, we recall an elementary result from linear algebra:
	\begin{remark}\label{Matrix Algebra1}
		Let $\A, \B$ be $k\times k$ complex matrices. 
		If $\A\  \x = \B\ \x$ for all $\x\in \complex^{\it k}$, then $\A=\B$.
	\end{remark} 
	\begin{theorem}\label{theo: explicit PCM}
		The product of companion matrices $\F_{t,r}$ associated with eq. \emph{(\ref{HLDE001})} coincides with the Casorati matrix $\BXI_{t,r}$  for all $t\in\integers_{r}$ and any fixed $r\ge s$, given by eq. {\rm (\ref{Casorati matrices})}. 
	\end{theorem}
	\begin{proof}	
		Let $\y_{r}=[y_{r},y_{r-1},...,y_{r-p+1}]'$ be an arbitrary  initial condition vector.  Applying eq. (\ref{homogeneous solution0}) for $t,t-1,...,t-p+1$, the components of the solution vector $\y_t$ associated with eq. (\ref{HLDE001})  are given by:
		\[\begin{array}{ccccccc}
		y_{t}&=&\xi^{(1)}_{t,r}y_{r}&+&\xi^{(2)}_{t,r}y_{r-1}&+...+&\xi^{(p)}_{t,r}y_{r-p+1}\\
		y_{t-1}&=&\xi^{(1)}_{t-1,r}y_{r}&+&\xi^{(2)}_{t-1,r}y_{r-1}&+...+&\xi^{(p)}_{t-1,r}y_{r-p+1}\vspace{-0.02in}\\
		\vdots&&\vdots&&\vdots&&\vdots\vspace{-0.01in}\\
		y_{t-p+1}&=&\xi^{(1)}_{t-p+1,r}y_{r}&+&\xi^{(2)}_{t-p+1,r}y_{r-1}&+...+&\xi^{(p)}_{t-p+1,r}y_{r-p+1}.
		\end{array}\]
		The above $p\times p$ system of linear equations can be expressed in a vector equation form as: 
		\begin{equation}\label{coefficient of product matrix1}
		\y_t=\BXI_{t,r}\cdot \y_{r}\  \  \text{for} \  t\ge r.
		\end{equation}
		A comparison of eqs. (\ref{product of companion matrices2}) and (\ref{coefficient of product matrix1}), on account of the uniqueness of the solution vector $\y_{t}$, implies that: $\BXI_{t,r}\ \y_{r}=\F_{t,r}\ \y_{r}$ for all $\y_{r}\in \complex^{\it p}$. It follows from Remark \ref{Matrix Algebra1} that 
		\begin{equation}\label{matrix equations1}
		\BXI_{t,r}=\F_{t,r},
		\end{equation}
		as asserted.
	\end{proof}
	By virtue  of eq. (\ref{matrix equations1}) the entries of $\F_{t,r}$ for $t\ge r+1$ are the banded Hessenbergians,  whose elements are explicitly expressed in terms of the variable coefficients of eq. {\rm (\ref{HLDE001})}. 
	As $\F_{t,r}$ is invertible, we conclude from eq. (\ref{matrix equations1}) that $\BXI_{t,r}$ is invertible too. This statement recovers the result stated in Corollary \ref{cor: invertible}. In the following Example we apply Theorem \ref{theo: explicit PCM} to the second order VC-LDE($2$).
	\begin{example}
	{\rm	In this example we consider the second order homogeneous  VC-LDE: }
		$$\displaystyle y_{t}=\phi_1(t)y_{t-1}+\phi_2(t)y_{t-2}.$$
	\end{example}
	\noindent Let $s\in\integers$ and $\phi_2(t)\not=0$ for all $t\ge s+1$.
	The Definition in eqs. (\ref{Exdef:xi^m}) and (\ref{Def1: Phi^m_k}) is applied for $r=t-1, t-2, t-3$ to verify the identity in eq.  (\ref{coefficient of product matrix1}), assuming that $r\ge s$.\\ 
	i) If $r=t-1\ge s$, then we conclude that: \ \
	$\xi_{t,t-1}=\phi_1(t),\ \xi^{(2)}_{t,t-1}=\phi_2(t), \xi_{t-1,t-1}=1,\ \xi^{(2)}_{t-1,t-1}=0$.\\
	The associated companion matrix is given by
	\[\F_{t,t-1}=\BGAMMA_{t}=\left[\begin{array}{cc}
	\phi_1(t) & \phi_2(t) \\
	1     &   0 \end{array}\right]=\left[\begin{array}{cc}
	\xi_{t,t-1} & \xi^{(2)}_{t,t-1} \\
	\xi_{t-1,t-1}     &  \xi^{(2)}_{t-1,t-1} \end{array}\right]=\BXI_{t,t-1}.
	\]
	ii) If $r=t-2\ge s$, then we conclude that:
	\[\xi_{t,t-2}=\left|\begin{array}{cc}
	\phi_1(t-1) & -1 \\
	\phi_2(t) & \phi_1(t) 
	\end{array}\right|, \ \xi^{(2)}_{t,t-2}=\left|\begin{array}{cc}
	\phi_2(t-1) & -1 \\
	0 & \phi_1(t) 
	\end{array}\right|,\ \xi_{t-1,t-2}=\phi_1(t-1),\ \xi^{(2)}_{t-1,t-2}=\phi_2(t-1). \]
	The product of the first two companion matrices is given by
	\[\begin{array}{lll}
	\F_{t,t-2}=\BGAMMA_{t}\BGAMMA_{t-1}=\left[\begin{array}{cc}
	\phi_1(t) & \phi_2(t) \\
	1     &   0 \end{array}\right]\left[\begin{array}{cc}
	\phi_1(t-1) & \phi_2(t-1) \\
	1     &   0 \end{array}\right]\!\!\!\!&=&\!\!\!\!
	\left[\begin{array}{cc}
	\phi_1(t)\phi_1(t-1)+\phi_2(t) & \phi_1(t)\phi_2(t-1) \\
	\phi_1(t-1)     &   \phi_2(t-1) \end{array}\right]\\\\
	&=&\left[\begin{array}{cc}
	\xi_{t,t-2} & \xi^{(2)}_{t,t-2} \\
	\xi_{t-1,t-2}     &  \xi^{(2)}_{t-1,t-2} \end{array}\right]=\BXI_{t,t-2}.
	\end{array}
	\]
	iii) If $r=t-3\ge s$, then we conclude that:
	\begin{align*}
	\xi_{t,t-3}=&\left|\begin{array}{ccc}
	\phi_1(t-2) &  -1 & 0\\
	\phi_2(t-1) & \phi_1(t-1) & -1\\
	0 & \phi_2(t) & \phi_1(t)  
	\end{array}\right|=\phi_1(t-2)[\phi_1(t)\phi_1(t-1)+\phi_2(t)]+\phi_2(t-1)\phi_1(t),\\\\ 
	\xi^{(2)}_{t,t-3}=&\left|\begin{array}{ccc}
	\phi_2(t-2) & -1 & 0\\
	0 & \phi_1(t-1) & -1\\
	0 & \phi_2(t) & \phi_1(t) 
	\end{array}\right|=\phi_2(t-2)[\phi_1(t)\phi_1(t-1)+\phi_2(t)],\\\\
	\xi_{t-1,t-3}=&\left|\begin{array}{cc}
	\phi_1(t-2) & -1 \\
	\phi_2(t-1) & \phi_1(t-1) \\
	\end{array}\right|=\phi_1(t-1)\phi_1(t-2)+\phi_2(t-1),\\\\
	\xi^{(2)}_{t-1,t-3}=&\left|\begin{array}{cc}
	\phi_2(t-2) & -1 \\
	0 & \phi_1(t-1) \\
	\end{array}\right|=\phi_1(t-1)\phi_2(t-2).
	\end{align*}
	The product of the first three companion matrices is given by
	given by
	\[\begin{array}{lll}
	\F_{t,t-3}=\BGAMMA_{t}\BGAMMA_{t-1}\BGAMMA_{t-2}\!\!\!\!&= &\!\!\!\! \left[\begin{array}{cc}
	\phi_1(t) & \phi_2(t) \\
	1     &   0 \end{array}\right]\left[\begin{array}{cc}
	\phi_1(t-1) & \phi_2(t-1) \\
	1     &   0 \end{array}\right]\left[\begin{array}{cc}
	\phi_1(t-2) & \phi_2(t-2) \\
	1     &   0 \end{array}\right]\\\\
	&=&\!\!\!\!\left[\begin{array}{cc}
	\phi_1(t)\phi_1(t-1)+\phi_2(t) & \phi_1(t)\phi_2(t-1) \\
	\phi_1(t-1)     &   \phi_2(t-1) \end{array}\right]\left[\begin{array}{cc}
	\phi_1(t-2) & \phi_2(t-2) \\
	1     &   0 \end{array}\right]\\\\
	&=&\!\!\!\!\left[\begin{array}{cc}
	\phi_1(t\!-\!2)[\phi_1(t)\phi_1(t\!-\!1)+\phi_2(t)]+\phi_1(t)\phi_2(t\!-\!1)  & \phi_2(t\!-\!2)[\phi_1(t)\phi_1(t\!-\!1)+\phi_2(t)] \\
	\phi_1(t-1)\phi_1(t-2)+\phi_2(t-1)     &   \phi_1(t-1)\phi_2(t-2) \end{array}\right]\\\\
	&=&\!\!\!\!\left[\begin{array}{cc}
	\xi_{t,t-3} & \xi^{(2)}_{t,t-3} \\
	\xi_{t-1,t-3}     &  \xi^{(2)}_{t-1,t-3} \end{array}\right]=
	\BXI_{t,t-3}.
	\end{array}
	\]
	\section{One Sided Green's Function}
As the Green's function determinant ratio formula (see \cite{MiLD68}, eq. (2.6), p. 39  or  \cite{AGLDE00}, eq. (2.11.7), p. 77) is independent of the choice of the fundamental solution set, having the set $\varXi_{s}$ at our disposal, the Green's function is explicitly represented in Theorem \ref{theo: Explicit Green's Function} of this Section.  Moreover, Theorem \ref{theo. extension of the Green's function domain} shows that a domain restriction of the Green's function coincides with the corresponding restriction of the principal determinant function. This result enables us to express the general homogeneous solution of VC-LDEs($p$) explicitly in terms of the Green's function in eq. (\ref{Green's representation of the homog. solution}) below. Some fundamental properties of the Green's function are also recovered. We start our discussion with an analogous representation of the Green's matrix. 
	
Let  $(t,r)\in \integers^2_{s}$.  The Green's matrix (see \cite{MiLD68}, p. 14) associated with the difference operator in eq. (\ref{eq. LDO}) is a two variable function defined via the companion matrix product $\F_{t,s}$ as follows:
	\begin{equation}\label{eq. Green's Matrix0}
	\G_{t,r}=\F_{t,s}\ \F^{-1}_{r,s}.
	\end{equation}
	As $t\ge s$ and $r\ge s$,  Theorem \ref{theo: explicit PCM} entails that $\F_{t,s}=\BXI_{t,s}$ and $\F_{r,s}=\BXI_{r,s}$, therefore eq. (\ref{eq. Green's Matrix0}) can be  written in terms of Casorati matrices  as:\vspace{-0.02in}
	\begin{equation}\label{eq. Green's Matrix}
	\G_{t,r}=\BXI_{t,s}\ \BXI^{-1}_{r,s}.
	\end{equation}
	Following Miller (see \cite{MiLD68}, p. 39), the Green's function $H(t,r)$ associated the difference operator in eq. (\ref{eq. LDO}) is defined to be the entry in the upper left-hand corner of $\G_{t,r}$, that is 
	\begin{equation}\label{Definition of Green function}
	H(t,r)\stackrel{\rm def}{=\joinrel=}\Unit_1\ \G_{t,r}\ \Unit'_1,
	\end{equation} 
	where $\Unit_1$ is the row unit vector: $\Unit_1=[1,0,...,0]$. An extension of the above Definition of $H(t,r)$ to cover all the domain values $(t,r)\in \integers_{s+1-p}\times \integers_{s}$ is given in Theorem \ref{theo: Explicit Green's Function}.
	
	In Lemma \ref{lem. restriction of the Green function} and Theorems \ref{theo: Explicit Green's Function} and \ref{theo. extension of the Green's function domain} below, we provide an explicit representation of the Green's function in terms of banded Hessenbergians. Let $t\ge s$. We define the sets: $\FX_{t}=\{(t,s),(t,s+1),...,(t,t)\}$ and $\FY_s=\bigcup_{t\ge s}\FX_{t}$. It follows that $(t,r)\in \FY_{s}$ if and only if $(t,r)\in \integers_{s}\times\integers_s$ and  $r \le t$, whence  $\FY_{s}\subset\integers_{s}\times\integers_s$.
	\begin{lemma}\label{lem. restriction of the Green function}
		Let $H(t,r)|_{\FY_s}$ be the restriction of the Green's function, $H(t,r)$, to  $\FY_s$. Then  $H(t,r)|_{\FY_s}=\xi_{t,r}$. In particular, if $t=r$, then $H(t,t)=H(t,t)|_{\FY_s}=\xi_{t,t}=1$, that is a well known property of the Green's function.		 
		If $t \ge r+1$, then $H(t,r)|_{\FY_s}$ can be represented by the principal determinant function, that is \emph{:}
		\begin{equation} \label{Explicit Green's function1}
		H(t,r)|_{\FY_s}=
		\left|\!\! 
		\begin{array}{ccccccc}
		\phi_{1}(r+1) & -1 &  &  &  &  &    \\ 
		\phi_{2}(r+2) & \phi _{1}(r+2) & \ddots
		&  &  &  &    \vspace{-0.02in}\\ 
		\vdots & \vdots & \ddots &\ \  \ddots &  &  &   \vspace{-0.06in}\\ 
		\phi_{p}(r+p) & \phi _{p-1}(r +p) & \ddots
		& \ddots &  &  &    \vspace{-0.02in}\\ 	
		& \vdots & \ddots & \ddots & \ \ \ \ \ddots  &  &      \vspace{-0.06in}\\
		& \phi_{p}(r+p+1) & \ddots
		& \ddots & \ddots &\ddots &    \\ 
		&  & \ddots & \ddots & \ddots &  \ddots &\ddots  \\ 
		&  &  & \phi_{p}(t-1) & \phi_{p-1}(t-1)& \cdots  &  -1 \\
		&  &  &  & \phi _{p}(t) & \cdots
		& \phi_{1}(t)\end{array}\!\! \right| 
		\end{equation}
	\end{lemma}
	\begin{proof}
Starting with the Definition in eq. (\ref{eq. Green's Matrix0}) the following equalities hold:
	\begin{equation*}
\begin{array}{ccl}
	\G_{t,r}&=& \F_{t,s}\F^{-1}_{r,s}\\ 
\text{(apply twice the Definition in eq. (\ref{product of companion matrices1}))}	
	&=&(\BGAMMA_{t}\BGAMMA_{t-1}...\BGAMMA_{r+1}\BGAMMA_{r}\BGAMMA_{r-1}...\BGAMMA_{s+1})(\BGAMMA_{r}\BGAMMA_{r-1}...\BGAMMA_{s+1})^{-1}\\
\text{(by an elementary property of invertible matrices)}	&=& 
	\BGAMMA_{t}\BGAMMA_{t-1}...\BGAMMA_{r+1}\BGAMMA_{r}\BGAMMA_{r-1}...\BGAMMA_{s+1}\BGAMMA^{-1}_{s+1}...\BGAMMA^{-1}_{r-1}\BGAMMA^{-1}_{r}\\
(\text{since}\ s\le r\le t)	&=& 	\BGAMMA_{t}\BGAMMA_{t-1}...\BGAMMA_{r+1}\\
\text{(apply the Definition in eq. (\ref{product of companion matrices1}))}	&=&\F_{t,r}\\
\text{(by Theorem \ref{theo: explicit PCM})}	&=&\BXI_{t,r}.
\end{array}
	\end{equation*}
	Applying the above result to the Definition  of the Green's function in eq. (\ref{Definition of Green function}), on account of eq. (\ref{Casorati matrices}), we conclude that:
		$H(t,r)=\Unit_1\ \G_{t,r}\ \Unit'_1=\Unit_1\ \BXI_{t,r}\ \Unit'_1=\xi_{t,r}$
		for all $(t,r)\in \FY_s$, as asserted. If $t=r$, then the above result and the Definition of $\xi_{t,r}$ in eqs. (\ref{xi^m initial values}), applied for $m=1$, entail that
		$H(t,t)=H(t,t)|_{\FY_s}=\xi_{t,t}=1$. 
		Finally, if $t\ge r+1$,  then eq. (\ref{Explicit Green's function1}) follows from the Definition of $\xi_{t,r}$ below eq. (\ref{eq. Principal Matrix}).
	\end{proof}
	In the following Theorem, we use the  set  $\varXi_s$, along with the  Definition in eq. (\ref{Definition of Green function}) to re-establish the determinant ratio  formula of the Green's function over the domain $\integers_{s+1-p}\times\integers_s$ (see the previously cited reference), but now in a fully explicit form expressed directly in terms of the elements of  $\varXi_s$ (see eq. (\ref{eq: Explicit Green's Function}) below) and therefore of the variable coefficients of eq. (\ref{VC-LDE(p)}). 
	\begin{theorem}\label{theo: Explicit Green's Function}
		The  Green's function $H(t,r)$ for  $(t,r)\in\integers_{s+1-p}\times\integers_s$ associated the difference operator in eq. {\rm(\ref{eq. LDO})} can be explicitly expressed as a ratio of determinants:
		\begin{equation}\label{eq: Explicit Green's Function}
		H(t,r)= \left|\begin{array}{cccc}
		\xi^{(1)}_{t,s} & \xi^{(2)}_{t,s}&...& \xi^{(p)}_{t,s}\\
		\xi^{(1)}_{r-1,s}& \xi^{(2)}_{r-1,s}&...& \xi^{(p)}_{r-1,s}\\
		\vdots & \vdots & \vdots\vdots\vdots & \vdots\\
		\xi^{(1)}_{r-p+1,s} & \xi^{(2)}_{r-p+1,s} &...& \xi^{(p)}_{r-p+1,s}
		\end{array}\right|\left|\begin{array}{cccc}
		\xi^{(1)}_{r,s} & \xi^{(2)}_{r,s}&...& \xi^{(p)}_{r,s}\\
		\xi^{(1)}_{r-1,s}& \xi^{(2)}_{r-1,s}&...& \xi^{(p)}_{r-1,s}\\\
		\vdots & \vdots & \vdots\vdots\vdots & \vdots\\
		\xi^{(1)}_{r-p+1,s} & \xi^{(2)}_{r-p+1,s} &...& \xi^{(p)}_{r-p+1,s}
		\end{array}\right|^{-1}.
		\end{equation}
	\end{theorem}
	\begin{proof}
	We remark that the elements in the first row of the two matrices involved in eq. (\ref{eq: Explicit Green's Function}),
	have the same cofactors and therefore  the cofactor expansion of the first of the above determinants, expanded along its first row, can be expressed as: 
	\begin{equation}\label{eq. numerator of Green function}
	\left|\begin{array}{cccc}
	\xi^{(1)}_{t,s} & \xi^{(2)}_{t,s}&...& \xi^{(p)}_{t,s}\\
	\xi^{(1)}_{r-1,s}& \xi^{(2)}_{r-1,s}&...& \xi^{(p)}_{r-1,s}\\
	\vdots & \vdots & \vdots\vdots\vdots & \vdots\\
	\xi^{(1)}_{r-p+1,s} & \xi^{(2)}_{r-p+1,s} &...& \xi^{(p)}_{r-p+1,s}
	\end{array}\right|=
	\xi^{(1)}_{t,s}Cof[\xi^{(1)}_{r,s}]+\xi^{(2)}_{t,s}Cof[\xi^{(2)}_{r,s}]+ ...+ \xi^{(p)}_{t,s}Cof[\xi^{(p)}_{r,s}]. 
	\end{equation}
	Let $(t,r)\in\integers_{s}\times\integers_s$. In view of eqs. (\ref{Definition of Green function}) and (\ref{eq. Green's Matrix})  and using the well known cofactor formula of the inverse matrix $\BXI^{-1}_{r,s}$, we have:
	\[\begin{array}{lll}
	H(t,r) &=& \Unit_1\  \G_{t,r} \Unit'_1
	= \Unit_1\ \BXI_{t,s}\ \BXI^{-1}_{r,s}\ \Unit'_1\\\\
	&=&  \Unit_1\left[\begin{array}{ccc}
	\xi^{(1)}_{t,s}& ...& \xi^{(p)}_{t,s}\\
	\xi^{(1)}_{t-1,s}&...& \xi^{(p)}_{t-1,s}\\
	\vdots  & \vdots\vdots\vdots & \vdots\\
	\xi^{(1)}_{t-p+1,s} &...& \xi^{(p)}_{t-p+1,s}
	\end{array}\right]\left[\begin{array}{ccc}
	\xi^{(1)}_{r,s} &...& \xi^{(p)}_{r,s}\\
	\xi^{(1)}_{r-1,s}&...& \xi^{(p)}_{r-1,s}\\
	\vdots  & \vdots\vdots\vdots & \vdots\\
	\xi^{(1)}_{r-p+1,s} &...& \xi^{(p)}_{r-p+1,s}
	\end{array}\right]^{-1}\left[\begin{array}{l}
		1\\
		0\\
		\vdots  \\
		0 \end{array}\right]\\\\
		&=&\Unit_1
		\left[\begin{array}{ccc}
		\xi^{(1)}_{t,s}& ...& \xi^{(p)}_{t,s}\\
		\xi^{(1)}_{t-1,s}&...& \xi^{(p)}_{t-1,s}\\
		\vdots  & \vdots\vdots\vdots & \vdots\\
		\xi^{(1)}_{t-p+1,s} &...& \xi^{(p)}_{t-p+1,s}
		\end{array}\right]\left(\frac{1}{\displaystyle\left|\BXI_{r,s}\right|}\left[\begin{array}{ccc}
		Cof[\xi^{(1)}_{r,s}]& ...& Cof[\xi^{(1)}_{r-p+1,s}]\\
		Cof[\xi^{(2)}_{r,s}]&...& Cof[\xi^{(2)}_{r-p+1,s}]\\
		\vdots  & \vdots\vdots\vdots & \vdots\\
		Cof[\xi^{(p)}_{r,s}] &...& Cof[\xi^{(p)}_{r-p+1,s}]
		\end{array}\right]\right)\left[\begin{array}{l}
		1\\
		0\\
		\vdots \\
		0 \end{array}\right]\\\\
		&=&\frac{\displaystyle\Unit_1}{\left|\displaystyle\BXI_{r,s}\right|}
		\left[\begin{array}{cccc}
		\xi^{(1)}_{t,s}&\xi^{(2)}_{t,s} & ...& \xi^{(p)}_{t,s}\\
		\xi^{(1)}_{t-1,s}&\xi^{(2)}_{t-1,s} &...& \xi^{(p)}_{t-1,s}\\
		\vdots  & \vdots  &\vdots\vdots\vdots & \vdots\\
		\xi^{(1)}_{t-p+1,s} &\xi^{(2)}_{t-p+1,s} &...& \xi^{(p)}_{t-p+1,s}
		\end{array}\right]\left[\begin{array}{c}
		Cof[\xi^{(1)}_{r,s}]\\
		Cof[\xi^{(2)}_{r,s}]\\
		\vdots  \\
		Cof[\xi^{(p)}_{r,s}]
		\end{array}\right]\\\\
		&=&\frac{\displaystyle[1,0,...,0]}{\displaystyle\left|\BXI_{r,s}\right|}
		\left[\begin{array}{ccccc}
		\xi^{(1)}_{t,s}Cof[\xi^{(1)}_{r,s}]&\!\! +\!\! &\xi^{(2)}_{t,s}Cof[\xi^{(2)}_{r,s}]&+\  ...\ +  & \xi^{(p)}_{t,s}Cof[\xi^{(p)}_{r,s}]\\
		\xi^{(1)}_{t-1,s}Cof[\xi^{(1)}_{r,s}]&\!\!+\!\!&\xi^{(2)}_{t-1,s}Cof[\xi^{(2)}_{r,s}]&+\ ...\ +& \xi^{(p)}_{t-1,s}Cof[\xi^{(p)}_{r,s}]\\
		\vdots&\vdots&\vdots&\vdots\ \vdots\vdots\vdots\ \vdots&\vdots \\
		\xi^{(1)}_{t-p+1,s}Cof[\xi^{(1)}_{r,s}] &\!\!+\!\!&\xi^{(2)}_{t-p+1,s}Cof[\xi^{(2)}_{r,s}]& +\ ...\ +& \xi^{(p)}_{t-p+1,s}Cof[\xi^{(p)}_{r,s}]
		\end{array}\right]\\\\
		&=& \left|\BXI_{r,s}\right|^{-1} (\xi^{(1)}_{t,s}Cof[\xi^{(1)}_{r,s}]+\xi^{(2)}_{t,s}Cof[\xi^{(2)}_{r,s}]+ ...+ \xi^{(p)}_{t,s}Cof[\xi^{(p)}_{r,s}]).
		\end{array}\]
		Comparing eq. (\ref{eq. numerator of Green function}) with the above last expression of $H(t,r)$ for $(t,r)\in\integers_{s}\times\integers_s$, eq. (\ref{eq: Explicit Green's Function}) follows. Moreover, for any $m\in\lbracket 1,p \rbracket$, $H(s+1-m,r)$ in 
		eq. (\ref{eq: Explicit Green's Function}) is well defined on the extended domain $\integers_{s+1-p}\times\integers_s$.
		\end{proof}	
An equivalent form of $H(t,r)$ for $s+1-p\le t\le s$ is derived below 
	 $$\begin{array}{lll}
		H(s+1-m,r)&=&\displaystyle\frac{\displaystyle \xi^{(1)}_{s+1-m,s}Cof[\xi^{(1)}_{r,s}]+\xi^{(2)}_{s+1-m,s}Cof[\xi^{(2)}_{r,s}]+ ...+ \xi^{(p)}_{s+1-m,s}Cof[\xi^{(p)}_{r,s}]}{\displaystyle\left|\BXI_{r,s}\right|}\\
		&& (\text{as}\ \xi^{(i)}_{s+1-m,s}=1, \text{if}\ i=m, \text{or $0$, if $i\not=m$})\\
		&=&\frac{\displaystyle Cof[\xi^{(m)}_{r,s}]}{\displaystyle\left|\BXI_{r,s}\right|}=\frac{\displaystyle (-1)^{m+1}}{\displaystyle\left|\BXI_{r,s}\right|}\left|\begin{array}{cccccc}
		\xi^{(1)}_{r-1,s}&... &\xi^{(m-1)}_{r-1,s}& \xi^{(m+1)}_{r-1,s}&...& \xi^{(p)}_{r-1,s}\\
		\vdots  & &\vdots&\vdots & & \vdots\\
		\xi^{(1)}_{r-p+1,s}& ... & \xi^{(m-1)}_{r-p+1,s} & \xi^{(m+1)}_{r-p+1,s} &...& \xi^{(p)}_{r-p+1,s}
		\end{array}\right|,
		\end{array}$$ 
	provided that if $m=1$ (resp. $m=p$), then the first two (resp. last two) columns illustrated in the latter determinant expression above are vanished. 

	Taking into account that $\{\xi^{(m)}_{.,s}\}_{1\le m \le p}$ are fundamental solutions defined for $t\in\integers_{s+1-p}$ (see Theorem \ref{Theo: fundamental solution set}), some crucial values of $H(t,r)$ on $\integers_{s+1-p}\times\integers_s$ are verified below  (see Theorem \ref{theo. extension of the Green's function domain} and the discussion below Corollary \ref{cor. Green's function  homogeneous solution representation}).
	
	In the following Theorem, we further enlarge the domain $\FY_s$ of Lemma \ref{lem. restriction of the Green function}, showing that the Green's function $H(t,r)$ coincides with the principal determinant function $\xi_{t,r}$ on an extended domain $\FZ$ defined as follows: Let $p>1$ and  $\J_i=\{(s-i,s),(s-i+1,s+1),..., (t, t+i),...\}$ for $1\le i\le p-1$. Let us call $\J=\bigcup_{i=1}^{p-1}\J_i$ and  $\FZ=\FY_s\cup\J$. Notice that if $p=1$, then $\J=\emptyset$ and $\FZ=\FY_s$. In what follows we assume that $p\ge 2$. Formally, $(t,r)\in \J$ if and only if  (iff for short) there exists some  $i\in\lbracket 1, p-1\rbracket$ such that $r-t=i$ (or $r=t+i$). Equivalently $(t,t+i)\in \J$ iff $i\in\lbracket 1, p-1\rbracket$ and $t\in\integers_{s-i}$. As $t\ge s-i$, it follows that $t+i\ge s$, whence $(t+i)\in\integers_{s}$.
  \begin{lemma}\label{lem: restriction properties}
  The following statements hold:\\
  i)  $\J_i\cap \J_k=\emptyset$, whenever $i\not= k$.\\
  ii) $\FZ\subset\integers_{s+1-p}\times \integers_{s}$.\\  
  iii) $\FY_s\cap\J=\emptyset$, that is  $\FY_s,\J$ are disjoint sets.\\
  iv) $\J=\FZ\setminus \FY_{s}$.\\
  v) $(t,r)\in\FZ$ iff $t\in\integers_{s+1-p}$ and $r\in \lbracket s, t-p+1\rbracket$. 
  \end{lemma}
  \begin{proof}
  i) Let  $(t,r)\in\J_i$. Then $r-t=i\not=k$,  whence $(t,r)\not\in\J_k$ and the result follows. \\
  ii)  By Definition we have: $\FY_s\subset \integers_{s}\times \integers_{s}\subset\integers_{s+1-p}\times \integers_{s}$. Also $(t,t+i)\in\J$ iff   $t\in\integers_{s-i}$ and $i\in\lbracket 1, p-1\rbracket$. As $\integers_{s-i}\subset\integers_{s+1-p}$, it follows that $t\in\integers_{s+1-p}$. Moreover, as  $(t+i)\in\integers_{s}$, it follows that $(t,t+i)\in\integers_{s+1-p}\times \integers_{s}$, whence $\J\subset\integers_{s+1-p}\times \integers_{s}$. As both $\FY_s$ and $\J$ are subsets of $\integers_{s+1-p}\times \integers_{s}$, we conclude that  $\FZ=\FY_s\cup\J\subset\integers_{s+1-p}\times \integers_{s}$.\\
  iii) Let  $(t,r)\in\FY_s$. Then as $p\ge 2$ we have: $s\le r\le t< t+i$ for any $i\in\lbracket 1, p-1\rbracket$. Accordingly $r\not= t+i$ for all $i\in\lbracket 1, p-1\rbracket$. Thus $(t,r)\not\in \J_i$ for all $i\in\lbracket 1, p-1\rbracket$ and so  $(t,r)\not\in\bigcup_{i=1}^{p-1}\J_i$. The latter implies that $\FY_s\cap(\bigcup_{i=1}^{p-1}\J_i)=\emptyset$ and the assertion follows from the Definition  $\J=\bigcup_{i=1}^{p-1}\J_i$.\\
  iv) As  $\FZ=\FY_s\cup\J$ and $\FY_s\cap \J=\emptyset$,  the assertion follows. \\
  v) In what follows we shall use the statements (a) to (c) below:
  
 	a) $t\in\integers_{s-i}$ iff ($t\ge s-i$ and $t\in\integers_{s+1-p}$) iff ($t\in\integers_{s+1-p}$ and $s \le t+i$)  (notice that $t\in\integers_{s+1-p}$ is redundant).
 	
  	b) ($r=t+i$ and $1\le i\le p-1$) iff $t+1\le r\le t+1-p$. 
  	
    c) [($s\le r$ and $t+1\le r\le t+1-p$) or $s\le r\le t$] iff  $s\le r\le t+1-p$. 
    
\noindent The Definition of $\J$ followed by statements (a) and (b) imply: $(t,r)\in\J$ iff ($t\in\integers_{s-i}$ and $1\le i\le p-1$ and $r=t+i$) iff  ($t\in\integers_{s+1-p}$ and $s\le t+i=r$ and $t+1\le r\le t+1-p$). Taking into account that  $(t,r)\in\FY_s$ iff $t\in\integers_{s}$ and $s\le r\le t$, it follows from (c) that: $(t,r)\in\J\cup\FY_s$ iff
  ($t\in\integers_{s+1-p}$ and $s\le r$ and $t+1\le r\le t+1-p$) or ($t\in\integers_{s}$ and $s\le r\le t$) iff  ($t\in\integers_{s+1-p}$ and $s\le r\le t+1-p$), as asserted.
  \end{proof}
  
	\begin{theorem}\label{theo. extension of the Green's function domain} 
	Let $H(t,r)|_{\FZ}$ be the restriction of the Green's function to $\FZ$.  Then  $H(t,r)|_{\FZ}=\xi_{t,r}$.
	\end{theorem}
	\begin{proof}
		The Definition in eq. (\ref{Exdef:xi^m}), applied for $m=1$, implies that $\xi_{t,r}=0$, whenever $r>t$ (or $\xi_{t,t+j}=0$, whenever  $j\ge 1$) for all  $t\in\integers_{s+1-p}$. 
	In view of Lemma \ref{lem. restriction of the Green function}, it suffices to show that $H(t,r)=0$ on the set $\FZ\setminus\FY_s$, and therefore on account Lemma \ref{lem: restriction properties} (iv), it suffices to show  that $H(t,r)=0$ on the set $\J$. Now, the result follows from the fact that for any $i\in \lbracket 1,p-1 \rbracket$ and any $t\in\integers_{s-i}$, the numerator of $H(t,t+i)$ in eq. (\ref{eq: Explicit Green's Function}) is zero, that is 
		\[\left|\begin{array}{cccc}
		\xi^{(1)}_{t,s} & \xi^{(2)}_{t,s}&...& \xi^{(p)}_{t,s}\\
		\xi^{(1)}_{t+i-1,s}& \xi^{(2)}_{t+i-1,s}&...& \xi^{(p)}_{t+i-1,s}
		\vspace{-0.04in}\\
		\vdots & \vdots & \vdots\vdots\vdots & \vdots\vspace{-0.03in}\\
		\xi^{(1)}_{t+i-p+1,s} & \xi^{(2)}_{t+i-p+1,s} &...& \xi^{(p)}_{t+i-p+1,s}
		\end{array}\right|=0, \]
		since its first row  coincides with one of its remaining rows for any $i=1,2,...,p-1$, while its  denominator is nonzero, i.e.,   $\left|\BXI_{t+i,s}\right|\not=0$, since $t+i\ge s$, whence  $\BXI_{t+i,s}$ is invertible (see Corollary \ref{cor: invertible}).
	\end{proof}

	\begin{corollary}\label{cor. Green's function  homogeneous solution representation}
		The general solution of eq. {\rm (\ref{HLDE001})} (or the general homogeneous solution of eq. {\rm (\ref{VC-LDE(p)})}) can be explicitly expressed in terms of the Green's function, the varying coefficients and the sequence of prescribed  values $\{y_{r+1-m}\}_{1\le m\le p}$ as:\vspace{-0.16in}
		\begin{equation} \label{Green's representation of the homog. solution} y_{t}=\displaystyle\sum_{m=1}^{p}\sum_{i=1}^{p+1-m}\phi_{m-1+i}(r+i)H(t,r+i)y_{r+1-m}\  \  \  \text{for all} \  t\ge r+1.\vspace{-0.02in}
		\end{equation}
	\end{corollary}
	\begin{proof}
		The range of values of $i$ in  the second sum of eq. (\ref{eq. general homogeneous solution}) is:  $1\le i\le p+1-m$. Taking into account that $s+1\le r+1\le t$ and $1\le i \le p+1-m\le p$ (since $m\ge 1$), the following chain of inequalities holds:
		\begin{equation}\label{eq. chain of inequalities}
		s\le s+1 \le r+1\le r+i \le r+p+1-m\le r+p=(r+1)-1+p\le t-1+p. 
		\end{equation}
		Thus $s\le r+i\le t-1+p$ for any $i\in \lbracket 1,p-1 \rbracket$ and therefore for all $i$ such that $1\le i \le p+1-m$. We infer that $t\in\integers_{s+1-p}$ and $s\le r+i\le t-1+p$. Thus, Lemma \ref{lem: restriction properties}(v) implies that $(t,r+i)\in\FZ$ and Theorem \ref{theo. extension of the Green's function domain}  allows us to replace $\xi_{t,r+i}$ in eq. (\ref{eq. general homogeneous solution}) with $H(t,r+i)$, which, in turn, implies that eq. (\ref{Green's representation of the homog. solution}) holds true, as claimed. 
	\end{proof}
	In the proof of Theorem \ref{theo. extension of the Green's function domain}, we have established a  property of the Green's function, that is $H(t,r)=0$ for all $(t,r)\in(\FZ\setminus\FY_s)$, showing there its equivalence to a well known result, that is $H(t,t+i)=0$ for any $1\le i\le p-1$ and any $t\in\integers_{s-i}$ (see \cite{{MiLD68}} eqs. (2.12),  p. 41, applied for $q=p$, $s=a+q-1$ and $i=k$). Proposition \ref{prop. Green's function property} in the Appendix, recovers an additional property of the Green's function, that is  $H(t,t+p)=\frac{1}{\phi_p(t+p)}$\vspace{0.05in} for all $t\in\integers_{s+1-p}$ (see the above cited reference, additionally applied with $a_q(t)=-\phi_p(t)$). As $\xi_{t,t+p}=0$ (see eq. (\ref{Exdef:xi^m}), applied for $m=1$), we conclude that $H(t,t+p)\not=\xi_{t,t+p}$. 
	
	The computational time complexity of the Green's function involved in eq.  (\ref{Green's representation of the homog. solution}) is linear. This is due to the identification $H(t,r)|_{\FZ}=\xi_{t,r}$ (see  Theorem \ref{theo. extension of the Green's function domain}), combined with the fact that the Gaussian elimination
	process computing banded determinants uses approximately  $\displaystyle\frac{k(p+1)^{2}}{4}$ multiplications, where $k$ is the order of the matrix and $(p+1)$ is the bandwidth of the matrix (see \cite{Thorson2000}). Accordingly,  the time complexity of the Green's function, involved in the solution of VC-LDEs($p$) is $O(k)$. This is computationally tractable and comparable with the time complexity of algorithms computing the same restriction of the Green's function by recursion.
	\section{Explicit Green's Function Solution Representation}
	In this Section, the banded Hessenbergian representation of the Green's function restriction on $\FZ$ (see Theorem \ref{theo. extension of the Green's function domain})  along with an analogous particular solution
	representation (see Proposition \ref{particular solution determinant representation} below), are employed to obtain an explicit expression to the general solution of nonhomogeneous VC-LDEs($p$) in eq. (\ref{VC-LDE(p)}). This is solely expressed in terms of the Green's function, the variable coefficients and the forcing terms (see eq. (\ref{nonhomogeneous solution})). Furthermore, we show the full equivalence between the aforementioned Green's function solution representation and the single determinant representation of the solution, established by Kittappa in \cite{KitRep93}.
	
	\subsection{Particular Solution}
	In the following Proposition, we provide the Hessenbergian representation of the solution associated with zero initial values. 
	\begin{proposition}\label{particular solution determinant representation}
	The particular solution of eq. {\rm (\ref{VC-LDE(p)})}, taking on the initial values $y_r=y_{r-1}=...=y_{r-p+1}=0$, can be expressed as a Hessenbergian function of $t\ge r+1$ for a fixed $r\ge s$:
		\begin{equation}\label{particular solution determinant representation2}
		\begin{array}{lll}
		y^{par}_{t}=\left|
		\begin{array}{cccccc}
		v_{r+1} & -1 &  &  &  &    \\ 
		v_{r+2} & \phi _{1}(r +2) 
		&  &  &  &    \\ 
		\vdots & \vdots  &\ \  \ddots &  &  &   \\ 
		v_{r+p-m+1} & \phi_{p-m}(r +p-m+1) 
		& \ddots &  &  &    \\ 	
		\vdots& \vdots  & \ddots &\ddots  &  &      \\
		v_{r +p+1}& \phi_{p}(r +p+1) 
		& \ddots & \ddots & &    \\ 
		\vdots &   & \ddots & \ddots &  \ddots &  \\ 
		v_{t-1}& & & \phi_{p-1}(t-1)& \cdots  &  -1 \\ 
		v_t&  &  & \phi_{p}(t) & \cdots
		& \phi _{1}(t)\end{array} \right|.\end{array}
		\end{equation}
	An equivalent formula to $y^{par}_{t}$ in eq. {\rm(\ref{particular solution determinant representation2})} expressed in terms of the principal determinant function and the forcing terms is given by:
		\begin{equation}\label{particular solution principal determinant representation}
	y^{par}_{t}=\sum_{i=1}^{t-r}v_{r+i}\xi_{t,r+i}.
	\end{equation}
	\end{proposition}
	\begin{proof}
		As the cofactor of an entry $\phi_{m}(t)$ in  the last row of eq. (\ref{particular solution determinant representation2}) is $y^{par}_{t-m}$, expanding the determinant in eq. (\ref{particular solution determinant representation2}) along the last row we obtain  $y^{par}_{t}=\sum_{m=1}^{p}\phi_m(t)y^{par}_{t-m}+v_t$, 
		which shows that $y^{par}_{t}$ solves  eq. {\rm (\ref{VC-LDE(p)})}.
		Let us now apply eq. (\ref{particular solution determinant representation2}) for $t=r+1,...,r+p$. We have:\vspace{-0.1in} $$y^{par}_{r+1}=v_{r+1},\ \  y^{par}_{r+2}=
		\phi_1(r+2)y^{par}_{r+1}+v_{r+2}, \ ..., \  y^{par}_{r+p}=\displaystyle\sum_{m=1}^{p-1}\phi_m(r+p)y^{par}_{r+p-m}+v_{r+p}.\vspace{-0.1in}$$
		As a consequence, for any $i\ge 1$,  we can write  $$y^{par}_{r+i}=\sum_{m=1}^{i-1}\phi_m(r+i)y^{par}_{r+i-m}+v_{r+i}+\sum_{m=i}^{p}\phi_m(r+i)y_{r+i-m},$$
		whenever $y_{r+i-m}=0$ for any $m$ such that $i\le m\le p$. In particular, if $i=1$, then $y_{r+1-m}=0$ for all $m$ such that $1\le m\le p$ and $y^{par}_{r+1}=v_{r+1}$. Thus setting $y_{r+1-m}=0$ for any $m\in\lbracket 1, p \rbracket$ in eq. (\ref{VC-LDE(p)}), the latter equation is satisfied by $y^{par}_{t}$ associated with zero initial values, that is $\{y_{r+1-p}=0,...,y_r=0\}$. Accordingly,  the result follows from the uniqueness of the initial value problem. Finally, expanding the determinant in eq. (\ref{particular solution determinant representation2}) along the first column, the expression in eq. (\ref{particular solution principal determinant representation}) follows immediately.
	\end{proof}
	\subsection{General Nonhomogeneous Solution}	
	Let us call $y^{hom}_{t}$ the general homogeneous solution given by eq. (\ref{Green's representation of the homog. solution}).
	Adding eqs. (\ref{Green's representation of the homog. solution}) and (\ref{particular solution principal determinant representation}), we obtain the general nonhomogeneous solution $y_t=y^{hom}_{t}+y^{par}_{t}$ for $t\ge r+1$  of eq. (\ref{VC-LDE(p)}), explicitly in terms of the principal determinant function $\xi_{t,r}$, the varying coefficients $\phi_{m}(t)$, the forcing terms $v_{t}$, and the prescribed values $y_{r+1-m}$ for $1\le m\le p$, as formulated below:
	\begin{equation} \label{nonhomogeneous solution2} 
	y_t=\displaystyle\sum_{m=1}^{p}\sum_{i=1}^{p+1-m}\phi_{m-1+i}(r+i)\xi_{t,r+i}y_{r+1-m}+\displaystyle\sum_{i=1}^{t-r} \xi_{t,r+i}v_{r+i}.
	\end{equation}
	Taking into account that $1\le i\le t-r$ (in the second summation term of eq. (\ref{nonhomogeneous solution2})),  it follows that $r+1\le r+i\le t$, thus Lemma \ref{lem. restriction of the Green function} allows us to use the identification: $H(t,r+i)=\xi_{t,r+i}$. Also the first term of the right-hand side of eq. (\ref{nonhomogeneous solution2}) is the homogeneous solution part, given by eq.
	 (\ref{Green's representation of the homog. solution}), thus we can rewrite eq. (\ref{nonhomogeneous solution2}) as:
	\begin{equation} \label{nonhomogeneous solution} 
	y_t=\displaystyle\sum_{m=1}^{p}\sum_{i=1}^{p-m+1}\phi_{m-1+i}(r+i)H(t,r+i)y_{r+1-m}+\displaystyle\sum_{i=1}^{t-r} H(t,r+i)v_{r+i}\  \  \  \text{for all} \  t\ge r+1.
	\end{equation}
	
	\subsection{Equivalent Solution Representations}
	The equivalence between the Green's function  explicit representations of the general solution to  VC-LDEs($p$) and that obtained by Kittapa in \cite{KitRep93} is demonstrated in the following Proposition.
	\begin{proposition}\label{equivalence}
		The Green's function solution representation in eq. \emph{(\ref{nonhomogeneous solution})} is equivalent to the single determinant solution representation of eq. \emph{(\ref{VC-LDE(p)})} established in \emph{\cite{KitRep93}}.
	\end{proposition}
	\begin{proof}
		Replacing the homogeneous solution part of eq. (\ref{nonhomogeneous solution}) (or eq. (\ref{nonhomogeneous solution2})) with its equivalent expression in eq. (\ref{homogeneous solution0}), we can rewrite eq. (\ref{nonhomogeneous solution}) as: 
		\begin{equation}\label{nonhomogeneous solution3}
		y_t=\displaystyle\sum_{m=1}^{p}\xi^{(m)}_{t,r}y_{r+1-m}+\displaystyle\sum_{i=1}^{t-r} \xi_{t,r+i}v_{r+i} \  \  \text{for} \  t\ge r+1.
		\end{equation}
		Applying the Hessenbergian expression of $\xi^{(m)}_{t,r}$ in eq. (\ref{Exdef:xi^m}) and the Hessenbergian expression of the particular solution in eq. (\ref{particular solution determinant representation2}) to eq. (\ref{nonhomogeneous solution3}), the latter solution representation of eq. (\ref{VC-LDE(p)}),  can be expressed in more detail as
		
		\begin{equation*}\label{general solution}
		\begin{array}{lll}
		y_t&=&\displaystyle\sum_{m=1}^{p}y_{r+1-m}\left|\!\! 
		\begin{array}{cccccc}
		\phi_{m}(r+1) & -1 &  &  &  &     \\ 
		\phi_{m+1}(r+2) & \phi _{1}(r +2) 
		&  &  &  &    \\ 
		\vdots & \vdots &  &  &  &    \\ 
		\phi_{p}(r+p-m+1) & \phi _{p-m}(r +p-m+1) 
		& \ddots &  &  &    \\ 	
		& \vdots  &  &  &  &      \vspace{-0.05in}\\
		& \phi _{p}(r +p+1) & \ddots
		&  \ddots & &    \\ 
		&  & \ddots & \ddots &   \ddots & \\ 
		&  &   & \phi _{p-1}(t-1)& \cdots  &  -1 \\ 
		&  &  & \phi _{p}(t) & \cdots
		& \phi _{1}(t)
		\end{array}\!\!\right|\vspace{0.2in}\\
		&&\hspace{0.62in} +\left|
		\begin{array}{cccccc}
		v_{r+1} & -1 &  &  &  &    \\ 
		v_{r+2} & \phi _{1}(r +2) 
		&  &  &  &    \\ 
		\vdots & \vdots  &\ \  \ddots &  &  &   \\ 
		v_{r+p-m+1} & \phi _{p-m}(r +p-m+1) 
		& \ddots &  &  &    \\ 	
		\vdots& \vdots  & \ddots &\ddots  &  &      \\
		v_{r +p+1}& \phi _{p}(r +p+1) 
		& \ddots & \ddots & &    \\ 
		\vdots &   & \ddots & \ddots &  \ddots &  \\ 
		v_{t-1}& & & \phi _{p-1}(t-1)& \cdots  &  -1 \\ 
		v_t&  &  & \phi _{p}(t) & \cdots
		& \phi _{1}(t)\end{array} \right|.\end{array} \end{equation*}
		As a result of the multi-linearity of determinants in columns, the right-hand side of the above expression of $y_t$ takes a single determinant form as:
		\begin{equation}\label{eq. single determinant sol.}
		y_{t}\!=\!\left|\!\! 
		\begin{array}{cccccc}
		\displaystyle\sum_{m=1}^{p}y_{r+1-m}\phi_{m}(r+1)+v_{r+1} & -1 &  &  &  &    \\ 
		\displaystyle\sum_{m=1}^{p}y_{r+1-m}\phi_{m+1}(r+2)+v_{r+2} & \phi _{1}(r+2) &   &  &  &    \vspace{-0.05in}\\ 
		\vdots & \vdots &  &  &  &    \vspace{-0.05in} \\ 
		\displaystyle\sum_{m=1}^{p}y_{r+1-m}\phi_{p}(r\!+\!1\!+p\!-\!m\!)+v_{r+1+p-m} & \phi _{p-m}(r\!+\!1+\!p\!-\!m\!) & \ddots &   &  &    \vspace{-0.05in} \\  	
		\vdots& \vdots & \ddots & \ddots &   &      \vspace{-0.05in} \\ 
		v_{r +p+1}& \phi _{p}(r +p+1)
		& \ddots & \ddots &\ddots &   \\ 
		\vdots&  & \ddots & \ddots &   \ddots &  \\ 
		v_{t-1}&  &  &  \phi _{p-1}(t-1)& \cdots  &  -1 \\ 
		v_{t}&    &  & \phi _{p}(t) & \cdots
		& \phi _{1}(t)
		\end{array}\!\!\right|.
		\end{equation}
		The latter expression coincides with the single determinant solution representation in \cite{KitRep93}.
	\end{proof}
	The equivalence result in Proposition \ref{equivalence} makes it possible to introduce an alternative scheme to the generation of the fundamental set of solutions $\varXi_{r}$, by using as starting point the single determinant solution representation in place of the IGE, as follows:
	Applying eq. (\ref{eq. single determinant sol.}) with $v_{r+i}=0$ for all $i\ge 1$ and assigning for each fixed $m\in\lbracket 1, p \rbracket$, the initial conditions $y_{r+1-m}=1$ and $y_{r+1-j}=0$, whenever $m\not= j$, we recover the sequences $\xi^{(m)}_{.,r}$ constructed by the IGE in Subsection \ref{subs. Infinite Gaussian Elimination} and established in Subsection \ref{Fundamental  and General Homogeneous Solutions}.   
	
	\section{Compact Representations}\label{Compact Representations}
In the present Section we apply the Leibnizian and nested sum representations of Hessenbergians  to derive compact representations for the domain restriction of the Green's function $H(t,r)$ to $\FZ$ (see Subsections \ref{Leibnizian  Representation of Green Function} and  \ref{Nested Sum Representation of Green Function} respectively). Moreover, compact representations for the elements of the companion matrix product and the elements of the Green's function determinant ratio formula are provided in Subsections \ref{Compact Companion Matrix Product} and \ref{compact Green's Function Determinant-Ratio Formula}, respectively. A  Leibnizian  compact representation for the general solution of nonhomogeneous VC-LDEs($p$) is obtained in Subsection \ref{Compact Solution  Representation} . 

In what follows, we consider a $k\times k$ full lower Hessenberg matrix $[h_{i,j}]_{1\le i,j\le k}$ with superdiagonal elements $h_{i,i+1}=-1$.
Applying eq. (\ref{assignment}) for $m=1$ we get:
\begin{equation}\label{assignment0}
h_{i,j}=\left\{\begin{array}{cl}
\phi_{i-j+1}(r+i) & {\rm if}\ 1\le i-j+1\le p\\
-1 & {\rm if} \ i-j+1=0        \\
\hspace{0.115in}  0 & {\rm elsewhere}.
\end{array}\right. 
\end{equation}
Under the assignment (\ref{assignment0}), the matrix ${\Hes}_{t-r}$ in eq. (\ref{Hessenberg Matrix}), applied for $k=t-r$, turns to a banded Hessenberg matrix, being identical to $\BGF_{t,r}$ in eq. (\ref{Def1: Phi^m_k}). A few distinct elements of ${\Hes}_{t-r}$  are provided below:\\
$h_{1,1}=\phi_{1}(r+1)$, $h_{1,2}=-1$, $h_{2,1}=\phi_{2}(r+2)$,
$h_{t-r,t-r}=\phi_{1}(t)$ and $h_{t-r,1}=\left\{\begin{array}{cl}
\phi_{t-r}(t) & {\rm if}\ 1\le t-r\le p\\
0  & {\rm if}\ t-r>p
\end{array}\right. $.
\subsection{Leibnizian Representation of the  Green's Function Restriction}\label{Leibnizian  Representation of Green Function}
Applying  the assignment in eq. (\ref{assignment0}) to eq. (\ref{compact-form}), 
we establish in \ref{compact-form2}  the Leibnizian compact representation of the the principal determinant function $\xi_{t,r}$.  By virtue of Theorem \ref{theo. extension of the Green's function domain}, $\xi_{t,r}$ is a banded Hessenbergian representation the  Green's function restriction $H(t,r)|_{\FZ}$, whence: 
\begin{equation}\label{compact-form2}
H(t,r)|_{\FZ}=\xi_{t,r}=\left\{ \begin{array}{cl}
\displaystyle \sum_{m=0}^{2^{t-r-1}-1}\ \  \prod_{i=1}^{t-r}\phi_{i-\sigma_{t-r,i}(m)+1}(r+i), &\text{if} \ \ \  s\le r< t  \\
1, &\text{if} \ \ \   t=r     \\
0,  &\text{elsewhere} 
\end{array}\right.
\end{equation}
\subsection{Nested Sum Representation of the of the  Green's Function Restriction}\label{Nested Sum Representation of Green Function}
The nested sum  representation of Hessenbergians,  established in \cite{MaTo16}  (see eq. (19) in their Corollary 4.1.), can be expressed according to our notation in eq. (\ref{Hessenberg Matrix}) (using the adjustment: $k=n$, $h_{i,i+1}=-1$ and $h_{i,j}=b_{i,j}$, elsewhere), as:\vspace{-0.05in}
\begin{equation}\label{Marrero}
\det(\Hes_k) = h_{k,1}+\sum_{j=2}^{k}\sum_{k_1=j}^{k}\sum_{k_2=j-1}^{k_1-1}...
\sum_{k_{j-1}=2}^{k_{j-2}-1}h_{k,k_1}\prod_{m=2}^{j-1}h_{k_{m-1}-1,k_m}h_{k_{j-1}-1, 1}. 
\end{equation}
Applying the assignment in eq. (\ref{assignment0}) to eq. (\ref{Marrero}), the nested sum representation of $H(t,r)|_{\FZ}$  (or $\xi_{t,r}$) takes the form:
\begin{equation}\label{Nested compact form}	
\mbox{\hspace{-5.0in}} H(t,r)|_{\FZ}=\xi_{t,r}\vspace{-0.05in}
\end{equation}
\begin{align*}
=\left\{ \begin{array}{ccl}
\!\! \phi_{t-r} (t) + \displaystyle\sum_{j=2}^{t-r}\sum_{k_1=j}^{t-r}\sum_{k_2=j-1}^{k_1-1}...&\!\!\!\!
\displaystyle\sum_{k_{j-1}=2}^{k_{j-2}-1}\phi_{t-r-k_1+1}(t) &\!\!\!\!\displaystyle\prod_{m=2}^{j-1}\phi_{k_{m-1}-k_m}(r+k_{m-1}\!-\!1) 
\phi_{k_{j-1}-1}(r+k_{j-1}\!-\!1), \\ 
& &\text{if} \ \ \   s\le r< t  \\
& 1,  &   \text{if} \ \ \  t=r     \\
& 0, & \text{elsewhere} 
\end{array}\right. 
\end{align*}
Proceeding with the above mentioned assignment, the Green's function restriction  $H(t,r)|_{\FZ}$ can also be represented by Mallik's combinatorial formula in \cite{MaEx98}, as adjusted for Hessenbergians in \cite{MaTo16} (see  eq. (9) therein).
\subsection{Companion Matrix Product}\label{Compact Companion Matrix Product}
By virtue of Theorem \ref{theo. extension of the Green's function domain}, it follows from the chain of inequalities in eq. (\ref{eq. chain of inequalities}) that we can replace $\xi_{t,r+j}$ with the  Green's function restriction $H(t,r+j)|_\FZ$ in eq. (\ref{xi^m_k}) for $1\le m\le p$  to obtain the expressions
\begin{equation}\label{xi^m_k2}
\xi^{(m)}_{t,r}=\left\{\begin{array}{lll}
\displaystyle\sum_{j=1}^{p+1-m}\phi_{m-1+j}(r+j)H(t,r+j),&  2\le m\le p,\\
&& t\ge r+1, \\
H(t,r),&  m=1,  \end{array}\right. 
\end{equation} 
noticing that if $t\le r$, then the corresponding values of $\xi^{(m)}_{t,r}$ are given by eq. (\ref{xi^m initial values}). Therefore the elements of the companion matrix product can be expressed directly in terms of the Green's function. 
Applying eq. (\ref{compact-form2}) to  (\ref{xi^m_k2}), we conclude that $\xi^{(m)}_{t,r}$ are equipped with the following Leibnizian representations:	
\begin{equation}\label{compact CMP}
\xi^{(m)}_{t,r}=\left\{\begin{array}{lll} \displaystyle\sum_{j=1}^{p-m+1}\phi_{m-1+j}(r+j)	\displaystyle\sum_{q=0}^{2^{t-r-1-j}-1} \ \prod_{i=1}^{t-r-j}\phi_{i-\sigma_{t-r-j,i}(q)+1}(r+j+i), &  2\le m\le p,\\
& & t\ge r+1.\\
\displaystyle \sum_{m=0}^{2^{t-r-1}-1}\ \  \prod_{i=1}^{t-r}\phi_{i-\sigma_{t-r,i}(m)+1}(r+i), & m=1,  \end{array}\right. 
\end{equation}
As  $\F_{t,r}=\BXI_{t,r}$ whenever $t\ge r\ge s$ (see Theorem \ref{theo: explicit PCM}), the  expressions in eq. (\ref{compact CMP}) yield compact representations for the elements of the companion matrix product in eq. (\ref{product of companion matrices1}),  respectively. The latter result is to be compared with Theorem 2.1. in (\cite{LimDay11}).	
\subsection{Green's Function Determinant-Ratio Formula}\label{compact Green's Function Determinant-Ratio Formula}

Applying  the expression in eq. (\ref{compact CMP}) for $r=s$ to the right-hand side determinant elements of eq. (\ref{eq: Explicit Green's Function}), we establish compact representations for the elements of the Green's function determinant ratio formula in eq. (\ref{eq: Explicit Green's Function}), formulated  by Leibnizian representations.

Similarly, applying the expression in eq. (\ref{Nested compact form}) for $r=s$  to eq. (\ref{xi^m_k2}), we obtain nested sum representations for the elements of Green's function determinant ratio formula in eq. (\ref{eq: Explicit Green's Function}).

\subsection{Leibnizian Solution Representation}\label{Compact Solution  Representation}
Given any sequence of prescribed values $\{y_{r+1-p},...,y_r\}$ for $r\ge s$ fixed,  the expressions in eq. (\ref{compact-form2}) applied to eq. (\ref{nonhomogeneous solution3}) yield the compact Leibnizian representation of the solution to eq. (\ref{VC-LDE(p)}) for $t\ge r+1$,  that is:
\begin{equation} \label{Leibniz representation to nonhomogeneous solution} 
\begin{array}{lll}
y_t&=&\displaystyle\sum_{m=1}^{p}y_{r+1-m}\sum_{j=1}^{p+1-m}\phi_{m-1+j}(r+j)\sum_{q=0}^{2^{t-r-j-1}-1} \prod_{i=1}^{t-r-j}\phi_{i-\sigma_{t-r-j,i}(q)+1}(r+j+i)\vspace{0.05in}\\
&&+\ \  \displaystyle\sum_{j=1}^{t-r} v_{r+j}\  \sum_{q=0}^{2^{t-r-j-1}-1} \   \prod_{i=1}^{t-r-j}\phi_{i-\sigma_{t-r-j,i}(q)+1}(r+j+i).\end{array}
\end{equation}
Eq. (\ref{Leibniz representation to nonhomogeneous solution}) is to be compared with the nested sum representation of the general solution established in  \cite{MaTo16} (see eq. (17) therein). 
In Appendix  \ref{Appendix: Algorithms}, both Algorithms \ref{algo. Compact} and \ref{algo. general solution} are applied  along with the assignment in eq. (\ref{assignment0}) to verify the compact solution representation in eq. (\ref{Leibniz representation to nonhomogeneous solution}),  evaluated  for any given sequence of prescribed values. 
	
\section{Future Work}
	\noindent The results of this work can be extended in multiple directions. We highlight two of them:
	
	Our results can be extended to cover an explicit representation to the solution of infinite order linear difference equations with constant or variable coefficients (ILDEs). This can be established by extending our results to cover linear difference equations of unbounded order, but of finite kernel index $p$ (ULDE($p$)) (see \cite{Paraskevopoulos2014}).~\footnote{The term LDEs of ascending order of index $N$ is equivalently used there for the ULDEs($p$).} An ULDE($p$) is  naturally derived as a $p$ order truncation of the ILDE,  yielding an approximation of $p$  order to the original ILDE. The  fundamental solution set obtained here can be similarly formulated  
	considering full lower Hessenberg matrices in place of banded ones, as in eq. (\ref{complete Hessenbergian}) of Lemma \ref{Lemma cofactor}. As a consequence, the general solution of ULDEs($p$) is given by eq. (\ref{nonhomogeneous solution2}), using full Hessenbergians in place of $\xi_{t,r+i}$. In a similar manner the Leibnizian and nested sum  representations of the solution can be directly derived. The corresponding solution of the ILDE turns out to be the limit of the associated ULDEs($p$), as $p\to \infty$.
	
	Our methodology can be generalized to the case of multivariate VC-LDEs($p$), where square matrices with elements variable coefficients are used in place of scalar valued variable coefficients, as in eq. (\ref{VC-LDE(p)}). 
	Working on the algebra of noncommutative rings,  an explicit form to  the general solution representation of multi-variate difference equations with variable matrix coefficients is obtained. This result has   some remarkable consequences on the fundamental properties of multivariate ARMA models.
	
	\bigskip

	\bibliographystyle{plain}

\bigskip
	
	\appendix
	
	\numberwithin{equation}{section}
	\renewcommand{\theproposition}{\Alph{section}\arabic{proposition}}
	\setcounter{proposition}{0}
	\renewcommand{\thecorollary}{\Alph{section}\arabic{corollary}}
	\setcounter{corollary}{0}
	\renewcommand{\thetheorem}{\Alph{section}\arabic{theorem}}
	\setcounter{theorem}{0} \renewcommand{\thelemma}{\Alph{section}\arabic{lemma}} \setcounter{lemma}{0}
	
\bigskip
	
	\section*{\hspace{2.5in}Appendices}
	\begin{appendices}
		In Appendix \ref{Appendix: ProofsOfPropositions}, we provide proofs for the results reported in the main body of the paper. The infinite Gaussian elimination algorithm is presented in Appendix \ref{Appendix: IGE}. Finally, in Appendix \ref{Appendix: Algorithms}, we provide two algorithms translated into automatically  executable computer programs. The first, constructs and verifies the Leibnizian compact representation of Hessenbergians in eq. (\ref{compact-form}). The second, constructs the Leibnizian representation of the Green's function $H(t,r)$ on $\FZ$, followed by the corresponding representation of the general solution of a VC-LDE($p$).
		\section{[Proofs]}\label{Appendix: ProofsOfPropositions}
		\begin{proposition}\label{positively signed} 
		The following statements hold:\\
		\textit{i}\emph{)}
		The recurrence in eq.  \emph{(\ref{Hessenbergian recurrence})}  can be equivalently expressed by eq. {\rm (\ref{modified recurrence1})}. 
			
		\noindent\textit{ii}\emph{)} The number of non-trivial \emph{SEPs} of Hessenbergians  is $2^{k-1}$, that is $\card(\NSEPs_{k})=2^{k-1}$.
		\end{proposition}
		\begin{proof}
			\textit{i}) 
			Applying the assignments $h_{i,j}=c_{i,j}$, whenever $j\not=i+1$ and $h_{i,i+1}=-c_{i,i+1}$ to  (\ref{Hessenbergian recurrence}) after some algebraic manipulations, demonstrated  below, eq. (\ref{modified recurrence1}) follows:
	\[\begin{array}{ll}
			\det(\Hes_k)\!\!\!&=\displaystyle h_{k,k}\det(\Hes_{k-1})+\sum_{i=1}^{k-1}(-1)^{k-i}h_{k,i}\prod_{j=i}^{k-1}h_{j,j+1}\det(\Hes_{i-1})\\
			\!\!\!&=\displaystyle
			c_{k,k}\det(\Hes_{k-1})\!+\sum_{i=1}^{k-1}(-1)^{k-i}c_{k,i}\prod_{j=i}^{k-1}(-1)c_{j,j+1}\det(\Hes_{i-1})\\
			&=\displaystyle
			c_{k,k}\det(\Hes_{k-1})\!+\sum_{i=1}^{k-1}(-1)^{k-i}c_{k,i}(-1)^{k-i}\prod_{j=i}^{k-1}c_{j,j+1}\det(\Hes_{i-1})\\
			&=\displaystyle
			c_{k,k}\det(\Hes_{k-1})\!+\sum_{i=1}^{k-1}(-1)^{2(k-1)}\prod_{j=i}^{k-1}c_{k,i}c_{j,j+1}\det(\Hes_{i-1})\\	
			&=\displaystyle c_{k,k}\det(\Hes_{k-1})+\sum_{i=1}^{k-1}\prod_{j=i}^{k-1}c_{k,i}c_{j,j+1}\det(\Hes_{i-1}).
			\end{array}\]
			\textit{ii}) Let $\n(r)$ be the number of distinct non-trivial SEPs associated with $\Hes_r$. Taking into account that $\n(0)=\n(1)=1$, the recurrence in eq. (\ref{modified recurrence2}) implies that $\n(r)=n(0)+n(1)+\sum_{i=2}^{r-1}\n(i)$ for all $r\ge 1$, provided that $\sum_{i=j}^{l}a(i)=0$, whenever $j>i$. This can be rewritten as:
			\begin{equation}\label{app. number of SEPs}
			\n(r)=1+\sum_{i=1}^{r-1}\n(i)\ \ \ \text{for all}\ \  r\ge 1.
			\end{equation}
			Working with the weak induction on $k\ge 1$ we shall show that $\n(k)=2^{k-1}$ for all $k\ge 1$. The basis step  $\n(1)=2^{0}=1$ holds true. The induction hypothesis assumes that the statement $\n(k-1)=2^{k-2}$ holds true.
			Starting with eq. (\ref{app. number of SEPs}), applied for $r=k$, we obtain:\vspace{-0.1in}
			\begin{align*}
			\n(k)&=1+\sum_{i=1}^{k-1}\n(i)\\
			(\text{equivalently}) &=(1+\sum_{i=1}^{k-2}\n(i))+\n(k-1)\\
			(\text{apply eq. (\ref{app. number of SEPs}) for}\ r=k-1) &=\n(k-1)+\n(k-1)\\
			(\text{equivalently})&=2\cdot\n(k-1)\\
			(\text{by the induction Hyhpothesis}) &=2\cdot 2^{k-2}\\
			(\text{equivalently})&=2^{k-1}
			\end{align*}
			This satisfies the induction step, and the proof is completed. 
		\end{proof}
		
		\begin{proposition}\label{standard IS}
			The standard IS, say $c_{i,j}$, of any initial string $C[i-1;\per]$ is uniquely determined by the number, say $m$ {\rm ($0\le m\le i-1$)}, of consecutive non-standard predecessors of $c_{i,j}$ and in this case $j=i-m$.
		\end{proposition}
		\begin{proof} 
			The hypothesis entails that the initial string can be expressed as:
			\[\begin{array}{cccc}
			C[i-1;\per]=
			c_{1\per_1}...c_{i-m-2,\per_{i-m-2}}&
			\hspace{-0.1in} \underbrace{c_{i-m-1,\per_{i-m-1}}} & \hspace{-0.1in}\underbrace{c_{i-m, i-m+1}c_{i-m+1, i-m+2}...c_{i-1, i}}.\\
			&{\rm standard} &  m\ {\rm non-standard\ factors}&
			\end{array}\]
			As $c_{i,j}$ is a standard IS of $C[i-1;\per]$, we can write $j=i-n$ for some $n=0,1,2,..,i-1$.
			In order to show that this standard IS of $C[i-1;\per]$ is $c_{i,i-m}$ (or $n=m$), 
			it suffices to show that none of the factors of $C[i-1;\per]$ has column index $i-m$. First,  the non-standard factors next to $c_{i-m-1,\per_{i-m-1}}$ have column indices  $i-m+1,...,i$. Thus $(i-m)\not\in\{i-m+1,...,i\}$. Moreover, as  $c_{i-m-1,\per_{i-m-1}}$ is standard, we infer that $\per_{i-m-1}\not=i-m$, since otherwise $c_{i-m-1,\per_{i-m-1}}=c_{i-m-1,i-m}$ which is non-standard. 
			Finally if $\per_{i-m-2}=i-m$, then $c_{i-m-2, \per_{i-m-2}}=c_{i-m-2,i-m}$, which is a trivial entry,  since $i-m-(i-m-2)=2$. The same holds for all the preceding factors of $c_{i-m-2, \per_{i-m-2}}$ and the result follows.
		\end{proof}

		\begin{proposition} \label{identification theorem} The function $f_{k}: \NSEPs_k \mapsto \mathfrak{R}_k$ defined in eq. \emph{(\ref{bijective f})} is bijective. 
		\end{proposition}
		\begin{proof}
			As the set $\mathfrak{R}_k$ and the set ${\mathcal{E}}_k$ have the same number of elements ($2^{k-1}$) it suffices to show that $f_{k}$ is injective. Let us consider  $Q=c_{1,\per_1}c_{2,\per_2}\dots c_{k,\per_k}$ and $P=c_{1,l_1}c_{2,l_2}\dots c_{k,l_n}$ in ${\mathcal{E}}_k$ such that $f_{k}(C)=f_{k}(P)$. We need to show that $Q=P$ or equivalently that $\per=l$. Let us call $f_{k}(C)=f_{k}(P)=\br$, where $\br=(r_1,r_2,...,r_{k-1},1)$. We examine the following cases:
			\begin{description}
				\item[I)] Let $r_i=0$. The Definition of $f_{k}$ implies that the $i$th non-trivial factor of $C$ and $P$ is non-standard. As there is only one such factor, that is the entry $(i, i+1)$, it must be the factor $c_{i,i+1}$. Thus $\per_i=l_i=i+1$.\vspace{-0.07in} 
				\item[II)] Let $r_i=1$. The Definition of $f_{k}$ implies that the $i$th non-trivial factors of $Q$ and $P$, say $c_{i,\per_i}$ and $c_{i,l_i}$, are
				standard. Property 4 in Proposition \ref{properties}, entails that $c_{i,\per_i}$ and $c_{i,l_i}$ are completely determined by the number of the consecutive non-standard predecessors of $c_{i,\per_i}$ and $c_{i,l_i}$. The result follows from  case I, which entails that both SEPs have identical non-standard factors occupying the same order positions. 
			\end{description}
			Therefore in all cases  $\per_i=l_i$, whence $C=P$ as required.  
		\end{proof}
		\begin{proposition} \label{unified formula} The function $\zeta_{k,i}(\br)$ in defined in \emph{(\ref{def of zeta})} can be expressed as an elementary integer function, which is given by :\vspace{-0.05in}
			\begin{equation} \label{zeta in terms of elementary functions2}
			\zeta_{k,i}(\br)=r_i(i-\max_{0\le j<i}\{j\cdot r_j\})-1\vspace{-0.1in}
			\end{equation}	
		\end{proposition}
		\begin{proof} Let us call $z_{k,i}(\br)=r_i(i-\displaystyle\max_{0\le j<i}\{j\cdot r_j\})-1$, while $\zeta_{k,i}(\br)$ is given by eq. (\ref{def of zeta}). We shall show that $z_{k,i}(\br)=\zeta_{k,i}(\br)$ for all $\br=(r_1,r_2,...,r_i,...,r_{k-1},1)\in \mathfrak{R}_k$. 
			First we notice that if $i=1$, then, in view of  eq. (\ref{special case i=1}), the equality  $z_{k,1}(\br)=\zeta_{k,1}(\br)$  holds true for all $\br\in\R_k$. It remains to show that $z_{k,i}(\br)=\zeta_{k,i}(\br)$ for all $i\ge 2$. In this case we have: $$\max_{0\le j<i}\{j\cdot r_j\})=\max\{0\cdot r_0,\ 1\cdot r_1,..., (i-1)\cdot r_{i-1}\}=\max\{1\cdot r_1,..., (i-1)\cdot r_{i-1}\}=\max_{1\le j<i}\{j\cdot r_j\}).$$
			Therefore in the case when $i\ge 2$, we can use the expression: $z_{k,i}(\br)=r_i(i-\displaystyle\max_{1\le j<i}\{j\cdot r_j\})-1$.
			We examine the following cases:\\
			i) Let $r_i=0$. Then a simple evaluation gives  $z_{k,i}(\br)=-1=\zeta_{k,i}(\br)$.\\
			ii) Let $r_i=1$. We examine the following sub-cases:\vspace{-0.05in}
			\begin{enumerate}[a)] 
				\item Let $\displaystyle\max_{1\le j<i}\{j\cdot r_j\}=0$. Then $r_j=0$ for all $j$ such that $1\le j<i$. Applying the formula of $z_{k,i}(\br)$  we get $z_{k,i}(\br)=i-1=\zeta_{k,i}(\br)$.\vspace{-0.02in}
				\item Let $\displaystyle\max_{1\le j<i}\{j\cdot r_j\}=M$ and $M>0$. Then we can write:\vspace{-0.1in}
				$$\{j\cdot r_j\}_{1\le j<i}=\{1\cdot r_1,...,(M-1)r_{M-1},Mr_M,(M+1)r_{M+1},...,(i-1)r_{i-1}\} .$$
				We shall show that $r_{M}=1$. On the contrary we assume that $r_{M}=0$. Then the following equality must hold
				$$M=\max\{1r_1,...,(M-1)r_{M-1},0,(M+1)r_{M+1},...,(i-1)r_{i-1}\},$$	
				which is contradictory, because  $M\not\in \{1,2,...,M-1,0,M+1,...,i-1\}$.
				
				In this case we further conclude that $r_{M+1}=r_{M+2}=...=r_{i-1}=0$ ; for if otherwise $\displaystyle\max_{1\le j<i}\{j\cdot r_j\}>M$. Therefore we conclude that:\vspace{-0.01in}
				$\{j\cdot r_j\}_{1\le j\le i}=\{1r_1,...,Mr_M,0\}$.
				As the number of consecutive $0$s between $r_M=1$ and $r_i=1$ is $i-M-1$, Definition (\ref{def of zeta}) gives $\zeta_{k,i}(\br)=i-M-1$. Also the formula of $z_{k,i}(\br)$ yields $z_{k,i}(\br)=r_i(i-M)-1=1(i-M)-1=i-M-1$, whence $z_{k,i}(\br)=\zeta_{k,i}(\br)$.
			\end{enumerate}
			The proof of Proposition is complete.
		\end{proof}
	\begin{proposition}\label{nestdiv}
		Let $n,d\in \integers$ and $d\ge 1$. The following identity of nested divisions holds:
		\begin{equation}\label{nested divisions1} 
		\begin{array}{ccc} \lfloor\lfloor...\lfloor\lfloor n:\!\!\!&\underbrace{d\rfloor:d\rfloor...\rfloor :d}&\!\!\!\!\rfloor=\lfloor n:d^m\rfloor.\\                        &   m                   & 
		\end{array}\end{equation}
	\end{proposition}
	\begin{proof} Let $x$ be a real number and $p,q$ be positive integers.
		We shall use the well known identity
		\begin{equation} \label{nested division identity 1} \lfloor\lfloor x\rfloor:p\rfloor=\lfloor x:p\rfloor.
		\end{equation}
		(see  \cite{Xi2017}, eq. (P15), p. 47).\\
		Taking into account that $ (x:q):p=x:(p\cdot q)$, it follows from (\ref{nested division identity 1}) that:
		\begin{equation} \label{nested division identity 2} 
		\lfloor\lfloor x:q\rfloor:p\rfloor=\lfloor (x:q):p\rfloor=\lfloor x:(p\cdot q)\rfloor.
		\end{equation}
		To verify (\ref{nested divisions1}) we use induction on $m\in \integers_{0}$. Clearly, the identity holds for $m=0$.
		Let us assume that the identity (\ref{nested divisions1}) holds for $m=i$, that is:\vspace{-0.05in} 
		\[\begin{array}{ccc} \lfloor\lfloor...\lfloor\lfloor n:\!\!\!\!\!&\underbrace{d\rfloor:d\rfloor...\rfloor :d}&\!\!\!\!\rfloor=\lfloor n:d^i\rfloor.\\ 
		&    i                   & 
		\end{array}\vspace{-0.1in} \]
		The induction hypothesis combined with (\ref{nested division identity 2}) implies: 
		\[\begin{array}{ccc} \lfloor\!\!&\underbrace{\lfloor\lfloor...\lfloor\lfloor n:d\rfloor:d\rfloor...\rfloor :d\rfloor}&\!\! : d\rfloor= \lfloor\lfloor n:d^i\rfloor:d\rfloor=\lfloor n:(d^i\cdot d)\rfloor=\lfloor n:d^{i+1}\rfloor.\\ 
		&   \lfloor n:d^i\rfloor   & 
		\end{array}\]
		This completes the induction.
	\end{proof}
		
	\begin{proposition}\label{prop. set equalities}	
The following equality of sets holds:	$$\integers_{s+1-p}=\bigcup_{j=1}^{p-1}\integers_{s-j}.$$
\end{proposition}
\begin{proof}
As $\integers_{s-j}\subseteq\integers_{s-(p-1)}$ for all $j\in\lbracket 1, p-1\rbracket$, it follows that $\integers_{s-(p-1)}=\bigcup_{j=1}^{p-1}\integers_{s-j}$
Now the equality follows from
$$\integers_{s+1-p}=\integers_{s-(p-1)}=\bigcup_{j=1}^{p-1}\integers_{s-j},$$
as asserted
\end{proof}

\begin{proposition}\label{prop. Casoratian Property}
The Casoratian $|\BXI_{t,r}|$ defined in eq. {\rm (\ref{Casorati matrices})} satisfies the  first order linear difference equation:
\begin{equation}\label{eq. Casoratian recurrence}
|\BXI_{t,r}|=(-1)^{p-1}\phi_p(t)|\BXI_{t-1,r}|.
\end{equation}
\end{proposition}
\begin{proof}
If we replace the elements $\xi^{(m)}_{t,r}$ for $1\le m\le p$ in the first row of $|\BXI_{t,r}|$ with the right-hand side of the recurrence (\ref{eq. ksi recurrence}) $|\BXI_{t,r}|$ takes the form:
\begin{equation}\label{Casorati matrices2}
|\BXI_{t,r}|=\left|\begin{array}{ccc}
\phi_p(t)\xi^{(1)}_{t-p,r}+...+\phi_{1}(t)\xi^{(1)}_{t-1,r}&...& \phi_p(t)\xi^{(p)}_{t-p,r}+...+\phi_{1}(t)\xi^{(p)}_{t-1,r}\vspace{0.05in}\\
\xi^{(1)}_{t-1,r}&...& \xi^{(p)}_{t-1,r}\vspace{0.05in}\\
\vdots & \vdots \vdots\vdots & \vdots\\
\xi^{(1)}_{t-p+1,r} &...& \xi^{(p)}_{t-p+1,r} \end{array}\right|. 
\end{equation}
Using the multi-linearity of determinants in rows, eq. (\ref{Casorati matrices2})	can be written as	\begin{equation*}\label{Casorati matrices3}
\begin{array}{llr}
|\BXI_{t,r}|&=&
\phi_p(t)\left|\begin{array}{ccc}
\xi^{(1)}_{t-p,r}&...& \xi^{(p)}_{t-p,r}\vspace{0.05in}\\
\xi^{(1)}_{t-1,r}&...& \xi^{(p)}_{t-1,r}\vspace{0.05in}\\
\vdots & \vdots \vdots\vdots & \vdots\\
			\xi^{(1)}_{t-p+1,r} &...& \xi^{(p)}_{t-p+1,r} \end{array}\right|+\phi_{p-1}(t)\left|\begin{array}{ccc}
			\xi^{(1)}_{t-p+1,r}&...& \xi^{(p)}_{t-p+1,r}\vspace{0.05in}\\
			\xi^{(1)}_{t-1,r}&...& \xi^{(p)}_{t-1,r}\vspace{0.05in}\\
			\vdots & \vdots \vdots\vdots & \vdots\\
			\xi^{(1)}_{t-p+1,r} &...& \xi^{(p)}_{t-p+1,r} \end{array}\right|\\\\
			&&+...+\phi_1(t)\left|\begin{array}{ccc}
			\xi^{(1)}_{t-1,r}&...& \xi^{(p)}_{t-1,r}\vspace{0.05in}\\
			\xi^{(1)}_{t-1,r}&...& \xi^{(p)}_{t-1,r}\vspace{0.05in}\\
			\vdots & \vdots \vdots\vdots & \vdots\\
			\xi^{(1)}_{t-p+1,r} &...& \xi^{(p)}_{t-p+1,r} \end{array}\right|. 
			\end{array}
			\end{equation*}
			The values of the determinants from the second term up to and including the last term of the right-hand side of the above equality are zero, since they have two identical rows, whence
			\begin{equation*}\label{Casorati matrices4}
			|\BXI_{t,r}|=
			\phi_p(t)\left|\begin{array}{cccc}
			\xi^{(1)}_{t-p,r}&\xi^{(2)}_{t-p,r}&...& \xi^{(p)}_{t-p,r}\vspace{0.05in}\\
			\xi^{(1)}_{t-1,r}&\xi^{(2)}_{t-1,r}&...& \xi^{(p)}_{t-1,r}\vspace{0.05in}\\
			\vdots &\vdots &  \vdots \vdots\vdots & \vdots\\
			\xi^{(1)}_{t-p+1,r} &\xi^{(2)}_{t-p+1,r} &...& \xi^{(p)}_{t-p+1,r} \end{array}\right|.
			\end{equation*}
			One needs $(p-1)$ successive row interchanges to move the first row  to the last row position and the above equality can be written as
			\begin{equation*}\label{Casorati matrices5}
			|\BXI_{t,r}|=(-1)^{p-1}
			\phi_p(t)\left|\begin{array}{cccc}
			\xi^{(1)}_{t-1,r}&\xi^{(2)}_{t-1,r}&...& \xi^{(p)}_{t-1,r}\vspace{0.05in}\\
			\xi^{(1)}_{t-2,r}&\xi^{(2)}_{t-2,r}&...& \xi^{(p)}_{t-2,r}\\
			\vdots & \vdots &\vdots \vdots\vdots & \vdots\\
			\xi^{(1)}_{t-p+1,r} &\xi^{(2)}_{t-p+1,r} &...& \xi^{(p)}_{t-p+1,r}\vspace{0.05in}\\
			\xi^{(1)}_{t-p,r}&	\xi^{(2)}_{t-p,r}&...& \xi^{(p)}_{t-p,r}
			\end{array}\right|=(-1)^{p-1}
			\phi_p(t)|\BXI_{t-1,r}|.
			\end{equation*}
			This completes the proof of Proposition.
		\end{proof}
		\begin{proposition}\label{prop. Green's function property}
			The Green's function associated with the difference operator in  eq. {\rm (\ref{eq. LDO})} has the property:\\
			\[H(t,t+p)=\frac{1}{\phi_p(t+p)}\ \ \ \ \text{for}\ \   t\in\integers_{s+1-p}. \]
		\end{proposition}
		\begin{proof}
			Eq. (\ref{eq: Explicit Green's Function}) applied for $r=t+p$ yields:
			\[H(t,t+p)= \left|\begin{array}{cccc}
				\xi^{(1)}_{t,s} & \xi^{(2)}_{t,s}&...& \xi^{(p)}_{t,s}\vspace{0.05in}\\
				\xi^{(1)}_{t+p-1,s}& \xi^{(2)}_{t+p-1,s}&...& \xi^{(p)}_{t+p-1,s}\vspace{0.05in}\\
				\vdots & \vdots & \vdots\vdots\vdots & \vdots\\
				\xi^{(1)}_{t+1,s} & \xi^{(2)}_{t+1,s} &...& \xi^{(p)}_{t+1,s}
				\end{array}\right|\left|\begin{array}{cccc}
				\xi^{(1)}_{t+p,s} & \xi^{(2)}_{t+p,s}&...& \xi^{(p)}_{t+p,s}\vspace{0.05in}\\
				\xi^{(1)}_{t+p-1,s}& \xi^{(2)}_{t+p-1,s}&...& \xi^{(p)}_{t+p-1,s}\\
				\vdots & \vdots & \vdots\vdots\vdots & \vdots\\
				\xi^{(1)}_{t+1,s} & \xi^{(2)}_{t+1,s} &...& \xi^{(p)}_{t+1,s}
				\end{array}\right|^{-1}.\]
			Applying $(p-1)$ successive row interchanges to the numerator determinant of $H(t,t+p)$ of the above equality, its first row is moved  to occupy the last row position, whence:
			\begin{equation}\label{Green function2}
			H(t,t+p)= (-1)^{p-1}\frac{\left|\begin{array}{cccc}	\xi^{(1)}_{t+p-1,s}& \xi^{(2)}_{t+p-1,s}&...& \xi^{(p)}_{t+p-1,s}\vspace{0.05in}\\
				\vdots & \vdots & \vdots\vdots\vdots & \vdots\\
				\xi^{(1)}_{t+1,s} & \xi^{(2)}_{t+1,s} &...& \xi^{(p)}_{t+1,s}\\
				\xi^{(1)}_{t,s} & \xi^{(2)}_{t,s}&...& \xi^{(p)}_{t,s}\vspace{0.05in}\\
				\end{array}\right|}{\left|\begin{array}{cccc}
				\xi^{(1)}_{t+p,s} & \xi^{(2)}_{t+p,s}&...& \xi^{(p)}_{t+p,s}\vspace{0.05in}\\
				\xi^{(1)}_{t+p-1,s}& \xi^{(2)}_{t+p-1,s}&...& \xi^{(p)}_{t+p-1,s}\vspace{0.05in}\\
				\vdots & \vdots & \vdots\vdots\vdots & \vdots\\
				\xi^{(1)}_{t+1,s} & \xi^{(2)}_{t+1,s} &...& \xi^{(p)}_{t+1,s}
				\end{array}\right|}=(-1)^{p-1}\frac{|\BXI_{t+p-1,s}|}{|\BXI_{t+p,s}|}.  
			\end{equation}
			The Casoratian recurrence in eq. (\ref{eq. Casoratian recurrence})
			takes the form
			\begin{equation}\label{Casoratian equivalent}
			|\BXI_{t+p,s}|=(-1)^{p-1}\phi_p(t+p)|\BXI_{t+p-1,s}| 
			\end{equation}
			Taking into account that $\phi_p(q)\not=0$ for all $q\ge s+1$, setting $q=t+p$, it follows that  $q=t+p\ge s+1$, whence $t\ge s+1-p$. and $\phi_p(t+p)\not=0$ for all  $t\in\integers_{s+1-p}$. We conclude  that eq. (\ref{Casoratian equivalent}) can be equivalently written as:
			\begin{equation}\label{eq. Casoratian recurrence2}
			\frac{1}{\phi_p(t+p)}=(-1)^{p-1}\frac{|\BXI_{t+p-1,s}|}{|\BXI_{t+p,s}|}.
			\end{equation}
			As the right-hand side members of  eqs. (\ref{Green function2}) and (\ref{eq. Casoratian recurrence2}) coincide, the result follows.
		\end{proof}
		\section{[The Infinite Gaussian  Elimination]}\label{Appendix: IGE}
		We present here the basic steps of the 
		IGE implemented with rightmost pivot
		 elements the $(-1)$s in eq. (\ref{infinite system representation}). It constructs the rows of $\FRREF(\A)$ along with the particular solution (see Subsection \ref{subs. Infinite Gaussian Elimination}), yielding equivalent row-recurrences. At the end of this Appendix we give some supplementary results, as reported in Example \ref{ex.1}.
		
		In what follows, the rows of the coefficient matrix,  $\A$,  in eq. (\ref{infinite system representation}) are denoted as $\BR_i$ for $i\ge 1$. In the first algorithmic step $\BR_1$ is normalized by multiplying $\BR_1$ with $(-1)$, yielding: $\tilde{\BR}_1=(-1)\BR_1$. The new row $\tilde{\BR}_1$ replaces $\BR_1$ and remains invariant during the forthcoming process. That is $\tilde{\BR}_1$ is the first row of $\FRREF(\A)$. 
		In the second step the algorithm uses $\tilde{\BR}_1$ as pivot row to eliminate the entry $\phi_1(r+2)$ of $\BR_2$, positioned in the same column and below the (rightmost) pivot entry $1$ of $\tilde{\BR}_1$. This is obtained by multiplying $(-\tilde{\BR}_1)$ (or $\BR_1$) with $\phi_1(r+2)$ and adding the result to $\BR_2$. After normalization, the second step is described by:  $\tilde{\BR}_2=-[\phi_1(r+2)(-\tilde{\BR}_1)+\BR_2]$.  The new row $\tilde{\BR}_2$ replaces  $\BR_2$,  yielding the second row of $\FRREF(\A)$.
		In the third step the algorithm uses $\tilde{\BR}_1$  and $\tilde{\BR}_2$ as pivot rows  to eliminate the entries  $\phi_2(r+3)$ and  $\phi_1(r+3)$ of $\BR_3$, respectively. After normalization, the new row $\tilde{\BR}_3$ is given by: $\tilde{\BR}_3=-[\phi_2(r+3)(-\tilde{\BR}_1)+\phi_1(r+3)(-\tilde{\BR}_2)+\BR_3]$. The new row $\tilde{\BR}_3$ replaces $\BR_3$ yielding the third row of $\FRREF(\A)$. Proceeding in this way the algorithm constructs the rows of a FRREF of $\A$, which are given by\vspace{-0.15in} 		
		\begin{equation}\label{eq. row recurrence}
		\tilde{\BR}_i=\phi_{1}(r+i)\tilde{\BR}_{i-1}+\phi_{2}(r+i)\tilde{\BR}_{i-2}+...+\phi_{p}(r+i)\tilde{\BR}_{i-p}- \BR_i
		\end{equation}
		provided that $\tilde{\BR}_i= \0$, whenever $1-p\le i\le 0$. The set $\{\tilde{\BR}_i\}_{1-p\le i\le 0}$,  consisting of $p$ zero rows, can be viewed as a set of initial conditions associated with the recurrence in eq. (\ref{eq. row recurrence}), generating the rows of $\FRREF(\A)=[\tilde{\BR}_{i}]_{i\ge 1}$. 
		For example if $i=1$, then  eq. (\ref{eq. row recurrence}) gives: $\tilde{\BR}_1=-\BR_1$, since  $\tilde{\BR}_{0}=\tilde{\BR}_{-1}=...=\tilde{\BR}_{1-p}=\0$. 
		If $i=2,3$, then eq. (\ref{eq. row recurrence}) gives:\vspace{-0.1in} 
		\begin{align*}
		\tilde{\BR}_2=&
		\phi_1(r+2)\tilde{\BR}_1-\BR_2=-[\phi_1(r+2)(-\tilde{\BR}_1)+\BR_2]\\
		\tilde{\BR}_3=&\phi_2(r+3)\tilde{\BR}_1+\phi_1(r+3)\tilde{\BR}_2-\BR_3=-[\phi_2(r+3)(-\tilde{\BR}_1)+\phi_1(r+3)(-\tilde{\BR}_2)+\BR_3],
		\end{align*}
		as expected, and so forth.
		
		As an alternative, the recurrence in  eq. (\ref{eq. row recurrence}) can be viewed as a VC-LDE($p$) with initial condition sequences $\tilde{\BR}_{m}$ for $1\le m\le p$, constructed  by the finite Gaussian elimination algorithm after a sequence of $p$ steps, whereas  the algorithm is implemented with rightmost pivoting and forcing terms $-\BR_i$ for $i\ge 1$. Thereafter, the recurrence in eq. (\ref{eq. row recurrence}) generates the remaining rows: $\{\tilde{\BR}_{i}\}_{i\ge p+1}$.
		
		Next, we employ Example \ref{ex.1} to verify the first two zero outcomes of the  product 
		$\A [0,1,\xi^{(1)}_{r+1,r},\xi^{(1)}_{r+2,r},...]'$:
		\begin{align*}
		[\phi_2(r+1),\phi_1(r+1),-1,0,...][0,1,\xi^{(1)}_{r+1,r}, \xi^{(1)}_{r+2,r},...]'&=\phi_1(r+1)\cdot 1+(-1)\phi_1(r+1)=0,\\
		[0,\phi_2(r+2),\phi_1(r+2),-1,0,...][0,1,\xi^{(1)}_{r+1,r}, \xi^{(1)}_{r+2,r},...]'&=\phi_2(r+2)\cdot 1+\phi_1(r+2)\phi_1(r+1)+\\
		& (-1)(\phi _{1}(r+1)\phi _{1}(r+2)+\phi_{2}(r+2))=0,\\
		\hspace{-1in}... &\hspace{1in} ...
		\end{align*}

		A particular solution is also constructed by the IGE algorithm, by  applying the same sequence of row elementary operations, used by the IGE for the row reduction of $\A$ to $\FRREF(\A)$, but now to the sequence  $\{-v_{r+i}\}_{i\ge 1}$. The algorithm gives rise to a  recurrence, which  similarly follows as in eq. (\ref{eq. row recurrence}) 
		\[\tilde{v}_{r+i}=\phi_{1}(r+i)\tilde{v}_{r+i-1}+\phi_{2}(r+i)\tilde{v}_{r+i-2}+...+\phi_{p}(r+i)\tilde{v}_{r+i-p}+ v_{r+i}, \ \ i\ge 1,  \]
		taking  on  zero initial values, that is $\tilde{v}_{r+1-m}=0$ for all $m\in\lbracket 1, p\rbracket$. The so  constructed solution  sequence $\{\tilde{v}_{t}\}_{t\ge r+1}$ is a particular solution, which is also represented explicitly by a Hessenbergian function that is  $y^{par}_t=\tilde{v}_{t}$ for $t\ge r+1$ (see Proposition \ref{particular solution determinant representation}).
		
		The general solution of eq.  (\ref{VC-LDE(p)}) is a linear combination of the fundamental solutions with coefficients arbitrary initial condition values $y_{r+1-m}=a_m$ for $1\le m\le p$ (that is the general homogeneous solution, see  Proposition \ref{Homogeneous solution}) plus the particular solution mentioned above (see also eq. (\ref{nonhomogeneous solution2})).

		\section{[Algorithms] }\label{Appendix: Algorithms} 
		Two Algorithms are presented in this Appendix. The first, returns the Leibnizian representation of Hessenbergians given in eq.  (\ref{compact-form}). The second, returns the restriction of the Green'r function $H(t,r)$, involved in the general homogeneous solution of VC-LDEs($p$), which coincides with $\xi_{t,r}$ (see Theorem 5).  This  algorithm is completed by the construction of the general nonhomogeneous solution of eq. (\ref{nonhomogeneous solution}), expressed in terms of the Green's function $H(t,r)$ (or $\xi_{t,s}$). Both Algorithms are followed by automatically executable computer programs written in Mathematica's symbolic language. The instructions of the algorithms follow the structure of the paper and use the corresponding formulas established in it. This can be viewed as a verification scheme for the validity of the results derived and used in the paper.
		
		In order to run the first program one needs to insert the order $k$ of the matrix. This is the only one external input, whereas the other inputs are internal instructions defined within the program and remain invariant in each new call of the program. Using Mathematica symbolic computation, the program  returns an expression of eq. (\ref{compact-form}) exclusively in terms of the non-trivial entries $h_{i,j}$ of $\Hes_k$. This program is to be compared with corresponding routines evaluating Hessenbergians.  
		
		The functions $\z_{k,i}(\br)$ and $\tau_k(m)$ along with their composite $\sigma_{k,i}(m)$ are defined within the program expressing the corresponding formulas in the chosen language. In their program notation, the variable $k$ (the order of the matrix) is omitted. Instead they are designated  as $\z_i, \tau, \sigma_i$, respectively, since all these functions are redefined for each new input of $k$. 
		
		In both algorithms each algorithmic step is followed by the corresponding instruction of the program, which is directly executable by Mathematica.
		
		\begin{algorithm}[Leibnizian representation of Hessenbergians]\label{algo. Compact}
			\begin{enumerate}[i)]
				\begin{align*}
				In[1]:\ \  & \$ \text{\rm Assumptions} = k > 0\ \&\&\ k \in {\rm Integers};\\
				\intertext{\rm \item Enter the order of the Hessenberg matrix:}
				In[2]:\ \  & k:=... \vspace{-0.1in}\\
				\intertext{\rm	\item Define the Hessenberg matrix $\Hes_k=(h_{i,j})_{i,j\in[[1,k]]}$ of order $k$:}
				In[3]:\ \   &  \Hes[k]:={\rm Table}[\ {\rm If}\ [j\le i+1,\ h[i,j],\ 0],\{i,1,k\},\{j,1,k\}]\vspace{-0.1in}\\
				\intertext{\rm	\item Define the entries $c_{i,j}$ of $\Hes_k$, according to eq. (\ref{Modified Hessenberg Matrix}):}
				In[4]:\ \   &  c[i\_,j\_]:=\ {\rm If}\ [j\not=i+1,\ \Hes[k][[i,j]],\ -\Hes[k][[i,j]]]\vspace{-0.1in} \\
				\intertext{\rm \item  Define the $i$th component, say $\tau_i(m)$, of $\tau(m)$  (given by eq. (\ref{tau bijection})) and assign  $\tau_1(m)=1$, whenever $i\not\in\lbracket 1,k-1\rbracket$:}
				In[5]:\ \   & \tau[i\_,m\_]:=\ {\rm If}\ [1 \le i \le k-1, \lfloor m \div 2^{k-i-1} \rfloor - 2\lfloor \frac{\lfloor m\div 2^{k-i-1}\rfloor}{2}\rfloor,1]\vspace{-0.1in}\\
				\intertext{\rm\item Define the composition of $\z_{i}$ in eq. (\ref{j minus function of zeros count}) and  $\tau$ in eq. (\ref{tau bijection}), using the function $\zeta_{i}$ in eq. (\ref{zeta in terms of elementary functions}), which, in turn, is constructed in a step by step procedure as follows:\newline
					a) Define the list of products $(j\cdot \tau(j,m))_{j=0,1,...,i-1}$:}
				In[6]:\ \   &  {\rm Prod}[i\_,m\_]:={\rm Table}[\ j\times \tau[j,m], \{j,0,i-1\}]\vspace{-0.1in}\\
				\intertext{\rm\ b) Evaluate the maximum value of ${\rm Prod}[i,m]$ and group these values in lists \newline
					${\rm M}[m]=\{\max({\rm Prod}[i,m],\ i\in\lbracket 1,k\rbracket)\}$: }
				In[7]:\ \   &  {\rm M}[m\_]:={\rm Table}[\ {\rm Max}[{\rm Prod}[i,m]], \{i,1,k\}]\vspace{-0.1in}\\
				\intertext{\rm	c) Define the function $Z(i,m)=\zeta_{i}\circ \tau(m)$ for $(i,m)\in\lbracket 1,k\rbracket\times\I_{k-1}$, according to eq. (\ref{zeta in terms of elementary functions}), that is $Z(i,m)$ is the number of consecutive zero predecessors of the factor $\tau(i,m)$ :}
				In[8]:\ \   & Z[i\_,m\_]:=\tau[i,m]\times (i-{\rm M}[m][[i]])-1\vspace{-0.1in}\\
				\intertext{\rm \item Define $\sigma_{i}(m)$ in terms of $Z(i,m)$, defined as  $\sigma_{i}(m)=\z_{i}\circ\tau(m)=i-\zeta_{i}(\tau(m))$: }
				In[9]:\ \   & \sigma[i\_,m\_]:=i-Z[i,m]\vspace{-0.1in}\\
				\intertext{\rm\item  Define the Hessenbergian formula (\ref{compact-form}):}
				In[10]:\ \   &  {\rm Hsb}[k]:=\sum_{m=0}^{2^{k\!-\!1}-1} \prod_{i=1}^kc[i,\sigma[i,m]]\vspace{-0.1in}\\
				\intertext{\rm\item  Expand the Hessenbergian formula:}
				In[11]:\ \   &  {\rm Expand}[{\rm Hsb}[k]]
				\intertext{\rm\item   Check whether the equation ${\rm Hsb}[k]={\rm Det}[\Hes[k]]$ holds true, where ${\rm Det}[\ ]$ stands for Mathematica's symbolic evaluation of determinants\vspace{-0.1in}:}
				In[12]:\ \   &  {\rm Hsb}[k]-{\rm Det}[\Hes[k]]==0
				\end{align*}
			\end{enumerate}
		\end{algorithm}
		As an example, setting $k=4$ and running the above program, it returns ${\rm Hsb}[4]$:
		\vspace{0.1in}
		\[\begin{array}{ll}\hspace{-0.625in} Out[1] =&-h[1,2] h[2,3] h[3,4] h[4,1] + h[1,1] h[2,3] h[3,4] h[4,2]\\ & + h[1,2] h[2,1] h[3,4] h[4,3] -h[1,1] h[2,2] h[3,4] h[4,3] \\ &+h[1,2] h[2,3] h[3,1] h[4,4]-h[1,1] h[2,3] h[3,2] h[4,4]\\ &-h[1,2] h[2,1] h[3,3] h[4,4]+h[1,1] h[2,2] h[3,3] h[4,4]
		\end{array}\]
		The program replies to the instruction (ix): "TRUE". This can be repeated by any value of $k\ge 2$, and the program replies "TRUE".
		
		The second Algorithm computes the Green's function restriction stated in Theorem \ref{theo. extension of the Green's function domain} as a Hessenbergian, followed be the construction of the general solution of eq. (\ref{VC-LDE(p)}) in terms of the Green function in eq. (\ref{nonhomogeneous solution}).
		
		\begin{algorithm}[Green's function and  the general solution of VC-LDEs($p$)]\label{algo. general solution}
			
			\begin{enumerate}[i)]
				\begin{align*}
				In[1]:\ \  & \$ \text{\rm Assumptions} = p > 0\ \&\&\ p \in {\rm Integers}\ \&\&\  s \in {\rm Integers}\ \&\&\  r \in {\rm Integers}\ \&\&\  t \in {\rm Integers};\\
				\intertext{\rm \item Enter the order of the linear difference equation:}
				In[2]:\ \  & p:=... \\
				\intertext{\rm	\item Enter the value of the variable $r\ge s$:}
				In[3]:\ \   & r:=... \\
				\intertext{\rm	\item Enter the value of the variable $t$ such that $t>r$:}
				In[4]:\ \   &  t:=... \\
				\intertext{\rm \item Replace the entries of $\Hes_k$ with the entries of $\BGF_{t,r}$ according to the assignment  \ref{assignment0}:}
				In[5]:\ \   & h[i\_, j\_ ]:= {\rm Which}[i == j - 1, -1, 1 \le i - j + 1 \le p, \phi_{i - j + 1}[r+i], {\rm True}, 0]\\
				\intertext{\rm\item Define the principal matrix as a function of $n,m$ for $n>m$:}
				In[6]:\ \   & \BGF[n\_, m\_] := {\rm Table}[h[i, j], \{i, m+1-r, n-r \}, \{j,   m+1-r, n-r\}]\\
				\intertext{\rm	\item Define the Green's function restriction:\newline $H(n,j)=\xi_{n,j}$ for $n>j$,  $H(n,j)=0$ for  $n<j$, $H(n,j)=1$ for $n=j$:}
				In[7]:\ \   & H[n\_, j\_] := {\rm Which}[n < j, 0, n == j, 1, n == j + 1, \phi_1[j + 1], n\ge j + 2,{\rm Det}[\BGF_{n,j}]]\\
				\intertext{\rm	\item Define the general solution formula in eq. (\ref{nonhomogeneous solution}) (or eq. (\ref{nonhomogeneous solution2})) with initial condition values $\{y_{r-p+1}, ...,y_r\}$, and forcing terms $v_{r+i}$, as a function of $n$:}
				In[8]:\ \   &
				y[n\_]:=\displaystyle\sum_{m=1}^{p}\sum_{i=1}^{p+1-m}\phi_{m-1+i}(r+i)H(n,r+i)y_{r+1-m}+\displaystyle\sum_{i=1}^{n-r} H(n,r+i)v_{r+i}.\\
				\intertext{\rm \item Apply the Definition of the  Green's function in  (vi) for $n=t$ and $j=r$:}
				In[9]:\ \   & {\rm Expand} [H[t, r]]\\
				\intertext{\rm\item Apply the general solution with $n=t$:}
				In[10]:\ \   &  {\rm Expand}[y[t]]
				\end{align*}
			\end{enumerate}
		\end{algorithm}
		As an illustrative example, setting $p=2$, $t=5$, $s=1$ and $r=2$ and running the above program, it returns the following expression\vspace{-0.2in}
		\begin{enumerate}[\mbox{}]
			\item
			\begin{align*}
			Out[1]: & =\ \phi_1(3) \phi_1(4) \phi_1(5)+\phi_1(5)\phi_2(4) +\phi_1(3) \phi_2(5).\\
			\intertext{\rm\item This is an expansion of the Green's function $H(5,2)=\xi_{5,2}$ associated with the second order VC-LDE.\vspace{-0.15in}} 
			\intertext{\rm\item The solution $y_5$ of the initial value problem $y_2=a, y_1=b$ with  forcing terms $v_3, v_4, v_5$ is also recoved by the program yielding:}
			Out[2]:  =& \ \phi_1(3) \phi_1(4) \phi_1(5) a + \phi_2(4) \phi_1(5) a + \phi_1(3) \phi_2(5) a + \phi_1(4) \phi_1(5) \phi_2(3) b + \phi_2(3) \phi_2(5)b \\
			& +v_4 \phi_1(5) + v_3 \phi_1(4) \phi_1(5)+v_3\phi_2(5) +v_5.
			\end{align*}
		\end{enumerate}
		This result is in accord with the solution expansion obtained directly by recursion.
	\end{appendices}
	
\end{document}